\documentclass{ceb-l}

\usepackage{amssymb}

\newtheorem{theorem}{Theorem}[section]
\newtheorem{lemma}[theorem]{Lemma}
\newtheorem{proposition}[theorem]{Proposition}
\newtheorem{corollary}[theorem]{Corollary}
\newtheorem*{weinsteinconjecture}{Weinstein Conjecture}
\newtheorem*{claim}{Claim}
\newtheorem*{question}{Question}

\theoremstyle{definition}
\newtheorem{definition}[theorem]{Definition}
\newtheorem{example}[theorem]{Example}

\theoremstyle{remark}
\newtheorem{remark}[theorem]{Remark}

\numberwithin{equation}{section}

\newcommand{\eqdef}{\;{:=}\;}

\renewcommand{\frak}{\mathfrak}

\newcommand{\C}{{\mathbb C}}
\newcommand{\Q}{{\mathbb Q}}
\newcommand{\R}{{\mathbb R}}

\newcommand{\Z}{{\mathbb Z}}
\newcommand{\op}{\operatorname}
\newcommand{\dbar}{\overline{\partial}}

\newcommand{\Spinc}{\op{Spin}^c}
\newcommand{\SO}{\op{SO}}
\newcommand{\SU}{\op{SU}}
\newcommand{\U}{\op{U}}

\newcommand{\Spin}{\op{Spin}}
\newcommand{\End}{\op{End}}

\newcommand{\Ker}{\op{Ker}}

\newcommand{\tensor}{\otimes}

\newcommand{\bpm}{\begin{pmatrix}}
\newcommand{\epm}{\end{pmatrix}}

\newcommand{\mc}[1]{{\mathcal #1}}

\begin{document}

\title[Taubes's proof of the Weinstein conjecture]{Taubes's proof of
  the Weinstein conjecture in dimension three}
\author{Michael Hutchings}
\address{Mathematics Department, 970 Evans Hall, University of
  California, Berkeley CA 94720}
\email{hutching@math.berkeley.edu}
\thanks{Partially supported by NSF grant DMS-0806037}


\date{}

\begin{abstract}
  Does every smooth vector field on a closed three-manifold, for
  example the three-sphere, have a closed orbit?  No, according to
  counterexamples by K.\ Kuperberg and others.  On the other hand
  there is a special class of vector fields, called Reeb vector
  fields, which are associated to contact forms.  The
  three-dimensional case of the Weinstein conjecture asserts that
  every Reeb vector field on a closed oriented three-manifold has a
  closed orbit.  This conjecture was recently proved by Taubes using
  Seiberg-Witten theory.  We give an introduction to the Weinstein
  conjecture, the main ideas in Taubes's proof, and the bigger picture
  into which it fits.
\end{abstract}

\maketitle

Taubes's proof of the Weinstein conjecture is the culmination of a
large body of work, both by Taubes and by others.  In an attempt to
make this story accessible to nonspecialists, much of the present
article is devoted to background and context, and Taubes's proof
itself is only partially explained.  Hopefully this article will help
prepare the reader to learn the full story from Taubes's paper
\cite{tw}.  More exposition of this subject (which was invaluable in
the preparation of this article) can be found in the online video
archive from the June 2008 MSRI hot topics workshop \cite{msri}, and
in the article by Auroux \cite{auroux}.

Below, in \S\ref{sec:SWC}--\S\ref{sec:CS3D} we introduce the statement
of the Weinstein conjecture and discuss some examples.  In
\S\ref{sec:strategy}--\S\ref{sec:CCH} we discuss a natural strategy
for approaching the Weinstein conjecture, which proves it in many but
not all cases, and provides background for Taubes's work.  In
\S\ref{sec:BigPicture} we give an overview of the big picture
surrounding Taubes's proof of the Weinstein conjecture.  Readers who
already have some familiarity with the Weinstein conjecture may wish
to start here.  In \S\ref{sec:SW3D}--\S\ref{sec:SWF} we recall
necessary material from Seiberg-Witten theory.  In \S\ref{sec:OTP} we
give an outline of Taubes's proof, and in \S\ref{sec:MD} we explain
some more details of it.  To conclude, in \S\ref{sec:BWC} we discuss
some further results and open problems related to the Weinstein
conjecture.

\section{Statement of the Weinstein conjecture}
\label{sec:SWC}

The Weinstein conjecture asserts that certain vector fields must have
closed orbits.  Before stating the conjecture at the end of this
section, we first outline its origins.  This is discussion is only
semi-historical, because only a sample of the relevant works will be
cited, and not always in chronological order.

\subsection{Closed orbits of vector fields}

Let $Y$ be a closed manifold (in this article all manifolds and all
objects defined on them are smooth unless otherwise stated), and let
$V$ be a vector field on $Y$.  A {\em closed orbit\/} of $V$ is a map
\[
\gamma:\R/T\Z\to Y
\]
for some $T>0$, satisfying the ordinary
differential equation
\[
\frac{d\gamma(t)}{dt}=V(\gamma(t)).
\]

Given a closed manifold $Y$, one can ask: Does every vector field on
$Y$ have a closed orbit?  If the Euler characteristic $\chi(Y)\neq 0$,
then by the Poincar\'{e}-Hopf index theorem every vector field on $Y$
has zeroes, which give rise to constant closed orbits.  In this
article we will mainly focus on the case where $Y$ is
three-dimensional.  Then $\chi(Y)=0$, which implies that the signed
count of zeroes of a generic vector field is zero; and it is
relatively easy to cancel these in order to construct a vector field
on $Y$ with no zeroes.  But understanding nonconstant closed orbits is
much harder.

Although for some special three-manifolds such as the 3-torus it is
easy to construct vector fields with no closed orbit, already for
$Y=S^3$ the question of whether all vector fields have closed orbits
is very difficult.  It turns out that the answer is no.  Examples of
vector fields on $S^3$ with no closed orbit, with increasing degrees
of regularity, were constructed by P.\ Schweizer ($C^1$), J.\ Harrison
($C^2$), K.\ Kuperberg ($C^\infty$), G.\ Kuperberg and K.\ Kuperberg
(real analytic), and G.\ Kuperberg ($C^1$ and volume preserving), see
\cite{schweizer,harrison,kkuperberg,kuperbergs,gkuperberg}.  Moreover,
as explained in \cite{kuperbergs}, such vector fields exist on any
3-manifold.  The constructions of these examples involve modifying a
given vector field by inserting a suitable ``plug'' in a neighborhood
of a point on a closed orbit, which destroys that closed orbit without
creating new ones.  The fact that one can do this indicates that to
gain any control over closed orbits, one needs to make some additional
assumption on the vector field.

\subsection{Hamiltonian vector fields}

A very important class of vector fields, originating in classical
mechanics, are Hamiltonian vector fields.  We briefly recall the
definition of these.

Let $(M^{2n},\omega)$ be a symplectic manifold.  This means that
$\omega$ is a closed $2$-form on $M$ such that $\omega^n\neq 0$
everywhere, or equivalently $\omega$ defines a nondegenerate bilinear
form on the tangent space $T_xM$ at each point $x\in M$. The basic
example of a symplectic manifold is $\R^{2n}$ with coordinates
$x_1,\ldots,x_n,y_1,\ldots,y_n$ and the standard symplectic form
\[
\omega_{std}=\sum_{i=1}^ndx_i\wedge dy_i.
\]
We will drop the subscript `$std$' when it is understood.  Darboux's
theorem asserts that any symplectic manifold $(M^{2n},\omega)$ is
locally equivalent to $(\R^{2n},\omega_{std})$.  More precisely, for
any $x\in M$ there exists a neighborhood $U$ of $x$ in $M$, an open
set $V\subset\R^{2n}$, and a symplectomorphism
$\phi:(U,\omega)\stackrel{\simeq}{\to} (V,\omega_{std})$, ie a
diffeomorphism $\phi:U\stackrel{\simeq}{\to} V$ such that
$\phi^*\omega_{std}=\omega$.

If $H:M\to\R$ is a smooth function, the associated {\em
  Hamiltonian vector field\/} is the vector field $X_H$ on $M$
characterized by
\[
\omega(X_H,\cdot) = dH.
\]
The vector field $X_H$ exists and is unique by the nondegeneracy
of $\omega$.

For example, if $M=\R^{2n}$ with the standard symplectic
form, then
\[
X_H = \sum_{i=1}^n\left(\frac{\partial H}{\partial
    y_i}\frac{\partial}{\partial x_i} - \frac{\partial H}{\partial
    x_i}\frac{\partial}{\partial y_i}\right).
\]
So a trajectory $(x(t),y(t))$ of the vector field $X_H$ satisfies the
equations
\[
\frac{d x_i}{dt} = \frac{\partial H}{\partial y_i}, \quad
\quad \frac{d y_i}{d t} = -\frac{\partial H}{\partial x_i}.
\]
These are the equations of classical mechanics, where the $x_i$'s are
position coordinates, the $y_i$'s are momentum coordinates, and $H$ is
the Hamiltonian, or energy function.

Conservation of energy in classical mechanics translates to the fact
that on a general symplectic manifold, the flow of $X_H$ preserves
$H$, because
\[
X_H(H)=dH(X_H)=\omega(X_H,X_H)=0.
\]
Therefore if $E\in\R$ is a regular value of $H$, then the level
set $H^{-1}(E)$ is a $(2n-1)$-dimensional submanifold of $M$, which we
will call a ``regular level set'', and $X_H$ is a smooth vector field
on it.

Must the Hamiltonian vector field $X_H$ have a closed orbit on every
regular level set?  It turns out that under favorable circumstances
the answer is ``almost every''.  For example, suppose $M=\R^{2n}$ with
the standard symplectic form.  Then one has the following:

\begin{theorem}[Hofer-Zehnder, Struwe, 1990]
\label{thm:AEE}
Let $H:\R^{2n}\to\R$ be a proper smooth function.  Then the vector
field $X_H$ has a closed orbit on $H^{-1}(E)$ for almost every
$E\in\R$ such that $H^{-1}(E)\neq\emptyset$.
\end{theorem}

The original proof, see the exposition in \cite{hofer-zehnder}, uses
variational methods to find critical points of the ``symplectic
action'' functional on the loop space, which are the desired closed
orbits.  Using Floer theory for the symplectic action functional
(different versions of Floer theory will be introduced later in this
article), this result can be extended to some other symplectic
manifolds; for a survey of some of these developments see
\cite{ginzburg}.

However there are some symplectic manifolds to which this theorem does
not extend.  To give a simple example due to Zehnder \cite{zehnder},
consider the 4-torus $T^4=(\R/2\pi\Z)^4$ with coordinates
$x_1,x_2,x_3,x_4$ and the nonstandard symplectic form
\[
\omega = dx_1dx_2 + \varepsilon dx_2dx_3 + dx_3dx_4,
\]
where $\varepsilon$ is an irrational constant.  If $H=\sin x_4 $,
then
\[
X_H=\cos x_4\left(\frac{\partial}{\partial x_3} +
  \varepsilon\frac{\partial}{\partial x_1}\right).
\]
Since $\varepsilon$ is irrational, this vector field has no closed
orbits except where $\cos x_4=0$.  Thus no regular level set
of $H$ contains a closed orbit of $X_H$.

In addition, the word ``almost'' cannot be removed from
Theorem~\ref{thm:AEE}.  Namely, Ginzburg-G\"{u}rel
\cite{ginzburg-gurel} proved that there is a proper $C^2$ Hamiltonian
$H$ on $\R^4$ with a regular level set on which the vector field $X_H$
has no closed orbit.  Moreover $C^\infty$ examples are known in
$\R^{2n}$ for $n>2$, see the references in \cite{ginzburg}.  So even
for Hamiltonian vector fields on $\R^{2n}$ one needs a further
assumption to guarantee the existence of closed orbits.

\subsection{Weinstein's conjecture}
\label{sec:OSWC}

To see what kind of assumption might be appropriate, let $Y$ be a
hypersurface in a symplectic manifold $(M,\omega)$ which is a regular
level set of a Hamiltonian $H:M\to\R$.  Observe that the existence of
a closed orbit of $X_H$ on $Y$ depends only on the hypersurface $Y$
and not on the Hamiltonian $H$.  For suppose $K:M\to\R$ is another
Hamiltonian which also has $Y$ as a regular level set.  Then
$dK|_Y=fdH|_Y$ for some nonvanishing function $f:Y\to\R$, so
$X_K=fX_H$ on $Y$.  Thus the periodic orbits of $X_K$ and $X_H$ on $Y$
differ only in their parametrizations.  In fact one can describe the
existence problem for periodic orbits on $Y$ without referring to a
Hamiltonian at all by noting that the Hamiltonian vector field on $Y$
always lives in the rank $1$ subbundle
\[
L_Y\eqdef \Ker(\omega|_Y)\subset TY,
\]
called the {\em charateristic foliation\/}.  Thus the existence of a
closed orbit on $Y$ for any Hamiltonian having it as a regular level
set is equivalent to the existence of a closed curve tangent to
$L_Y$, sometimes called a {\em closed characteristic\/}.

Under what circumstances must a hypersurface in a symplectic manifold
have a closed characteristic?  In the late 1970's, Weinstein
\cite{weinstein78} showed that in $\R^{2n}$ with the standard
symplectic form, if $Y$ is a convex compact hypersurface, then $Y$ has
a closed characteristic.  Rabinowitz \cite{rabinowitz} showed that the
above statement holds with ``convex'' replaced by ``star-shaped'',
meaning transverse to the radial vector field
\begin{equation}
\label{eqn:RVF}
\rho \eqdef \frac{1}{2}\sum_{i=1}^n(x_i\partial_{x_i} + y_i\partial_{y_i}).
\end{equation}
Now the existence of a closed characteristic is invariant under
symplectomorphisms of $\R^{2n}$, while the star-shaped condition is
not.  This suggests that one should look for a more general,
symplectomorphism invariant condition on the hypersurface $Y$ that might
guarantee the existence of a closed characteristic.  Weinstein
\cite{weinstein79} proposed such a condition as follows.

\begin{definition}
  A hypersurface $Y$ in a symplectic manifold $(M,\omega)$ is of {\em
    contact type\/} if there is a 1-form $\lambda$ on $Y$ such that
  $d\lambda = \omega|_Y$, and $\lambda(v)\neq 0$ for all nonzero $v\in
  L_Y$.
\end{definition}

This condition is
clearly invariant under symplectomorphisms of $(M,\omega)$.

If $Y$ is a star-shaped hypersurface in $\R^{2n}$, then $Y$ is of
contact type, because one can take
\[
\lambda = \frac{1}{2}\sum_{i=1}^n(x_idy_i-y_idx_i)|_Y.
\]
To see why $\lambda(v)\neq 0$ for nonzero $v\in L_Y$, note that
nondegeneracy of $\omega$ implies that $\omega(v,w)\ne 0$ for any
vector $w$ transverse to $Y$.  The star-shaped condition says that
$\rho$ is transverse to $Y$, so $\omega(\rho,v)\neq 0$; on the other
hand it follows from the above formulas that
$\omega(\rho,v)=\lambda(v)$.

More generally, a {\em Liouville vector field\/} on a symplectic
manifold $(M,\omega)$ is a vector field $\rho$ such that the Lie
derivative $\mathcal{L}_\rho\omega=\omega$.  The radial vector field
\eqref{eqn:RVF} is an example of a Liouville vector field.  It turns
out that a hypersurface $Y$ in $(M,\omega)$ is of contact type if and
only if there exists a Liouville vector field defined in a
neighborhood of $Y$ which is transverse to $Y$, see eg \cite[Prop.\
3.58]{mcduff-salamon}.  In particular, given such a vector field
$\rho$, the contact type condition is fulfilled by $\lambda =
\omega(\rho,\cdot)|_Y$.

One could now conjecture that if $Y$ is a compact hypersurface of
contact type in a symplectic manifold $(M,\omega)$, then $Y$ has a
closed characteristic.  This is essentially what Weinstein conjectured
in \cite{weinstein79}.

\subsection{Contact forms}
\label{sec:CF}

In fact one can remove the symplectic manifold $(M,\omega)$ from the
discussion as follows.  Let $Y$ be an oriented manifold of dimension
$2n-1$.  A {\em contact form\/} on $Y$ is a $1$-form $\lambda$ on $Y$
such that
\[
\lambda\wedge(d\lambda)^{n-1}>0
\]
everywhere.  A contact form $\lambda$ determines a vector field $R$ on
$Y$, called the {\em Reeb vector field\/}, characterized by
\[
d\lambda(R,\cdot)=0, \quad\quad \lambda(R)=1.
\]

In the definition of a contact type hypersurface, one can replace the
condition ``$\lambda(v)\neq 0$ for all nonzero $v\in L_Y$'' by the
condition ``$\lambda$ is a contact form''.  Note that for such a contact
form, the associated Reeb vector field is tangent to the
characteristic foliation $L_Y$.

Conversely, any manifold with a contact form $(Y,\lambda)$ arises as a
hypersurface of contact type in an associated symplectic manifold.
Namely, the {\em symplectization\/} of $(Y,\lambda)$ is the manifold
$\R\times Y$ with the symplectic form $\omega=d(e^s\lambda)$, where
$s$ denotes the $\R$ coordinate.  The slice $\{0\}\times Y$ is a
hypersurface of contact type in $\R\times Y$, with $\lambda$
fulfilling the definition of contact type.

In conclusion, the conjecture at the end of \S\ref{sec:OSWC} is
equivalent to the following:

\begin{weinsteinconjecture}
  Let $Y$ be a closed oriented odd-dimensional manifold with a contact
  form $\lambda$.  Then the associated Reeb vector field $R$ has a
  closed orbit.
\end{weinsteinconjecture}

The eventual goal of this article is to explain Taubes's proof of the
Weinstein conjecture in dimension three:

\begin{theorem}[Taubes]
If $Y$ is a closed oriented three-manifold with a contact form, then
the associated Reeb vector field has a closed orbit.
\end{theorem}

\subsection{Some remarks on Taubes's proof}
Before starting on the long story which follows, let us briefly
comment on what is involved in Taubes's proof, and the history leading
up to it.

First of all, as we will see in \S\ref{sec:basicExamples},
\S\ref{sec:CS3D}, and \S\ref{sec:CCH} below, methods from symplectic
geometry have been used since the 1980's to prove the Weinstein
conjecture in many cases, in both 3 and higher dimensions.  But it is
not currently known if such methods can be used to prove it in all
cases.

Taubes's proof instead uses Seiberg-Witten theory.  As we describe in
\S\ref{sec:BigPicture}, the proof uses part of a 3-dimensional version
of Taubes's ``Seiberg-Witten=Gromov'' (SW=Gr) theorem.  The SW=Gr
theorem relates Seiberg-Witten invariants of symplectic 4-manifolds to
counts of holomorphic curves.  Its 3-dimensional counterpart relates
Seiberg-Witten theory of a 3-manifold with a contact form to closed
orbits of the Reeb vector field (and holomorphic curves in the
3-manifold cross $\R$).

The SW=Gr theorem was proved in the 1990's, and one might wonder why
the 3-dimensional Weinstein conjecture was not also proved around that
time using similar methods.  Part of what was missing was a sufficient
development of Seiberg-Witten Floer homology, the 3-dimensional
counterpart of Seiberg-Witten invariants of 4-manifolds, which has
since been provided by Kronheimer-Mrowka.  In particular
Kronheimer-Mrowka proved a nontriviality result for Seiberg-Witten
Floer homology, which is the key input on the Seiberg-Witten side in
the proof of the Weinstein conjecture.  Finally, as we explain in
\S\ref{sec:OTP}, the 3-dimensional analogue of SW=Gr is not just a
straightforward adaptation of the 4-dimensional story, but rather
involves some nontrivial new ingredients, for example a new result of
Taubes estimating the spectral flow of families of Dirac operators in
terms of the Chern-Simons functional.

What follows now is a somewhat leisurely tour, gradually working
towards the above.  The impatient reader may wish to skip ahead to
\S\ref{sec:BigPicture} for an overview of the big picture surrounding
Taubes's proof, or to \S\ref{sec:OTP} for an outline of the proof
itself.

\subsection{Some terminology}
\label{sec:terminology}

Below, to save space, we usually say ``Reeb orbit'' instead of
``closed orbit of the Reeb vector field''.  Two Reeb orbits
$\gamma,\gamma':\R/T\Z\to Y$ are considered equivalent if they differ
by reparametrization, ie precomposition with a translation of
$\R/T\Z$.  If $\gamma:\R/T\Z\to Y$ is a Reeb orbit and $k$ is a
positive integer, then the $k$-fold {\em iterate\/} of $\gamma$ is the
pullback of $\gamma$ to $\R/kT\Z$.  A Reeb orbit $\gamma$ is embedded
if and only if it is not the $k$-fold iterate of another Reeb orbit
where $k>1$.

\section{Basic examples in $2n-1$ dimensions}
\label{sec:basicExamples}

\subsection{Hypersurfaces in $\R^{2n}$}
\label{sec:hct}

The Weinstein conjecture for compact hypersurfaces of contact type in
$\R^{2n}$ was proved in 1987 by Viterbo \cite{viterbo}.  In fact the
almost-existence result in Theorem~\ref{thm:AEE} is a generalization
of this.  To see why, let $Y$ be a compact hypersurface in $\R^{2n}$
of contact type.  As discussed in \S\ref{sec:OSWC}, there exists a
Liouville vector field $\rho$ defined on a neighborhood of $Y$ which
is transverse to $Y$.  Flowing the hypersurface $Y$ along the vector
field $\rho$ for a small time $\varepsilon$ gives another hypersurface
$Y_\varepsilon$ and a diffeomorphism $Y\simeq Y_\varepsilon$.  Since
$\rho$ is transverse to $Y$, there exists $\delta>0$ such that the
different hypersurfaces $Y_\varepsilon$ for $|\epsilon|<\delta$ are
disjoint and sweep out a neighborhood that can be identified in the
obvious way with $(-\delta,\delta)\times Y$.  Also, the Liouville
condition $\mathcal{L}_\rho\omega=\omega$ implies that the above
diffeomorphism $Y\simeq Y_\varepsilon$ respects the characteristic
foliations.  Thus $Y_\varepsilon$ has a closed characteristic for
either all $\varepsilon$ or none.  On the other hand we can choose a
proper Hamiltonian $H:\R^{2n}\to\R$ such that $H=\varepsilon$ on
$Y_\varepsilon$ and $|H|>\delta$ outside of our neighborhood
$(-\delta,\delta)\times Y$.  Then Theorem~\ref{thm:AEE} implies that
$Y_\varepsilon$ has a closed characteristic for almost every
$\varepsilon$.  Hence there is a closed characteristic for every
$\varepsilon$, and in particular for $\varepsilon=0$.

\subsection{Cotangent bundles}

Another important example of a manifold with a contact form is the
unit cotangent bundle of a Riemannian manifold.

To start, let $Q$ be a smooth manifold.  There is a canonical $1$-form
$\lambda$ on the cotangent bundle $T^*Q$, defined as follows.  Let
$\pi:T^*Q\to Q$ denote the projection.  If $q\in Q$, and if $p\in
T^*_qQ$, then $\lambda:T_{(q,p)}T^*Q\to\R$ is the composition
\[
T_{(q,p)}T^*Q \stackrel{\pi_*}{\longrightarrow} T_qQ
\stackrel{p}{\longrightarrow} \R.
\]
More explicitly, if $q_1,\ldots,q_n$ are local coordinates on a
coordinate patch $U\subset Q$, then one can write any cotangent vector
at a point in $U$ as $p=\sum_{i=1}^np_idq_i$, and this gives local
coordinates $q_1,\ldots,q_n,p_1,\ldots,p_n$ on $\pi^{-1}(U)\subset
T^*Q$.  In these coordinates
\[
\lambda = \sum_{i=1}^np_idq_i.
\]
It follows from this last equation that $d\lambda$ defines a
symplectic form on $T^*Q$.

Now suppose that $Q$ has a Riemannian metric.  This induces a metric
on $T^*Q$, and we consider the unit cotangent bundle
\[
ST^*Q = \{p\in T^*Q \mid |p|=1\}.
\]
The restriction of $\lambda$ to $ST^*Q$ is a contact form.  Indeed
$ST^*Q$ is a hypersurface of contact type in the symplectic manifold
$(T^*Q,d\lambda)$, with transverse Liouville vector field
$\rho=\sum_{i=1}^np_i\partial_{p_i}$, and $\imath_\rho d\lambda =
\lambda$.  It turns out that the associated Reeb vector field agrees
with the geodesic flow, under the identification $T^*Q=TQ$ given by
the metric.  Thus Reeb orbits in $ST^*Q$ are equivalent to closed
geodesics in $Q$.  If $Q$ is compact, then so is $ST^*Q$, so the
Weinstein conjecture is applicable in this case, where it is
equivalent to the classical Lyusternik-Fet theorem asserting that
every compact Riemannian manifold has at least one closed geodesic.

More generally, the above Liouville vector field $\rho$ shows that a
hypersurface $Y\subset T^*Q$ which intersects each fiber transversely
in a star-shaped subset of the fiber is of contact type.  Rabinowitz
proved the existence of closed characteristics for a related class of
hypersurfaces in $\R^{2n}$.  Alan Weinstein points out to me that this
example and the unit cotangent bundle example were important
motivation for his conjecture.

\subsection{Prequantization spaces}

Another general example of manifolds with contact forms is given by
circle bundles, or ``prequantization spaces''.  Let $(\Sigma,\omega)$
be a symplectic manifold of dimension $2n-2$, and suppose that the
cohomology class $-[\omega]/2\pi\in H^2(\Sigma;\R)$ is the image of an
integral class $e\in H^2(\Sigma;\Z)$.  Let $p:V\to\Sigma$ be the
principal $S^1$-bundle with first Chern class $e$. This means that
there is a free $S^1$ action on $V$ whose quotient is $\Sigma$, and
$e\in H^2(\Sigma;\Z)$ is the primary obstruction to finding a section
$\Sigma\to V$.  Let $R$ denote the derivative of the $S^1$ action;
this is a vector field on $V$ which is tangent to the fibers. Since
$\omega$ is a closed form in the cohomology class $-2\pi e$, one can
find a (real-valued) connection $1$-form $\lambda$ on $V$ whose
curvature equals $\omega$.  These conditions mean that $\lambda$ is
invariant under the $S^1$ action, $\lambda(R)=1$, and
$d\lambda=p^*\omega$.  It follows that $\lambda$ is a contact form on
$V$ whose Reeb vector field is $R$.  In particular, the Reeb orbits
are the fibers (which all have period $2\pi$) and their iterates.

The fact that Reeb orbits appear here in $(2n-2)$-dimensional smooth
families is a special feature arising from the symmetry of the
picture.  For a ``generic'' contact form on a manifold $Y$ the Reeb
orbits are isolated, in the sense that if $\gamma$ is a Reeb orbit of
length $T$ which goes through a point $x\in Y$, then there is no other
Reeb orbit through a point close to $x$ with length close to $T$.  For
example, on a circle bundle as above, one can get rid of most of the
Reeb orbits by perturbing the contact form $\lambda$ to
\[
\lambda'=(1+p^*H)\lambda,
\]
where $H:\Sigma\to\R$ is a smooth function.  This is still a contact
form as long as $|H|<1$. The new Reeb vector field is given by
\begin{equation}
\label{eqn:R'}
R' = (1+p^*H)^{-1}R + (1+p^*H)^{-2}\widetilde{X}_H,
\end{equation}
where $X_H$ is the Hamiltonian vector field on $\Sigma$ determined by
$H$, and $\widetilde{X}_H$ denotes its horizontal lift, ie the unique
vector field on $V$ with $p_*\widetilde{X}_H=X_H$ and
$\lambda(\widetilde{X}_H)=0$.  Consequently, for this new contact
form, the only fibers that are Reeb orbits are the fibers over the
critical points of $H$.  On the other hand there may be additional
Reeb orbits that cover closed orbits of $X_H$.  However it follows
from \eqref{eqn:R'} that if $H$ and $dH$ are small, then these Reeb
orbits all have period much greater than $2\pi$.  In any case, the
Weinstein conjecture here asserts that there is no way to eliminate
all of the remaining Reeb orbits without introducing new ones.  The
Weinstein conjecture in this case can be proved using cylindrical
contact homology \cite[\S2.9]{egh}, about which we will have more to
say in \S\ref{sec:CCH}.

\section{More about contact geometry in three dimensions}
\label{sec:CS3D}

We now restrict attention to the three-dimensional case.  Much more is
known about contact geometry in three dimensions than in higher
dimensions, and to gain an appreciation for Taubes's result and its
proof it will help to review some of the basics of this.  For much
more about this subject we refer the reader to \cite{etnyre} and
\cite{geiges}.

\subsection{Contact structures in three dimensions.}

Recall that a contact form on a closed oriented three-manifold $Y$ is
a $1$-form $\lambda$ such that $\lambda\wedge d\lambda>0$ everywhere.
The associated {\em contact structure\/} is the $2$-plane field
$\xi=\Ker(\lambda)$.  This has an orientation induced from the
orientation of $Y$ and the direction of the Reeb vector field.  In
general one defines a contact structure\footnote{Sometimes this is
  called a ``co-oriented contact structure''.  There is also a notion
  of unoriented contact structure in which $\xi$ is not assumed to be
  oriented and is only required to be locally the kernel of a contact
  form.} to be an oriented $2$-plane field which is the kernel of some
contact form as above.  A contact structure is a ``totally
nonintegrable'' $2$-plane field, which means that in a sense it as far
as possible from being a foliation: the kernel of $\lambda$ is a
foliation if and only if $\lambda\wedge d\lambda \equiv 0$.

Different contact forms can give rise to the same contact structure.
To be precise, if $\lambda$ is a contact form, then $\lambda'$ is
another contact form giving rise to the same contact structure if and
only if $\lambda'=f\lambda$ where $f:Y\to\R$ is a positive smooth
function.  For a given contact structure, the Reeb vector field
depends on the choice of contact form, but it is always transverse to
the contact structure.

A (three-dimensional closed) {\em contact manifold\/} is a pair
$(Y,\xi)$ where $Y$ is a closed oriented three-manifold and $\xi$ is a
contact structure on $Y$.  Two contact manifolds $(Y,\xi)$ and
$(Y,\xi')$ are {\em isomorphic\/}, or {\em contactomorphic\/}, if there is an
orientation-preserving diffeomorphism $\phi:Y\to Y'$ such that
$\phi_*$ sends $\xi$ to $\xi'$ preserving the orientiations.  Two
contact structures $\xi$ and $\xi'$ on $Y$ are {\em isotopic\/} if
there is a one-parameter family of contact structures $\{\xi_t\mid
t\in[0,1]\}$ on $Y$ such that $\xi_0=\xi$ and $\xi_1=\xi'$.  Gray's
stability theorem asserts that $\xi$ and $\xi'$ are isotopic if and
only if there is a contactomorphism between them which is isotopic to
the identity.

A version of Darboux's theorem asserts that any contact structure on a
3-manifold is locally isomorphic to the ``standard contact structure''
on $\R^3$, which is the kernel of the contact form
\begin{equation}
\label{eqn:SCS}
\lambda = dz - y\, dx.
\end{equation}
In fact any contact form is locally diffeomorphic to this one.  The
contact structure defined by \eqref{eqn:SCS} is invariant under
translation in the $z$ direction.  The contact planes are horizontal
along the $x$ axis, but rotate as one moves in the $y$ direction; the
total rotation angle as $y$ goes from $-\infty$ to $+\infty$ is $\pi$.
The Reeb vector field associated to $\lambda$ is simply
\[
R=\partial_z.
\]
In particular there are no Reeb orbits.  Of course this does not
contradict the Weinstein conjecture since $\R^3$ is not compact, but
it does indicate that any proof of the Weinstein conjecture will need to use
global considerations.

\subsection{Tight versus overtwisted}
\label{sec:overtwisted}

Define an {\em overtwisted disk\/} in a contact
3-manifold $(Y,\xi)$ to be a smoothly embedded closed disk $D\subset
Y$ such that for each $y\in\partial D$ we have $T_y D =
\xi_y$.  A contact 3-manifold is called {\em
  overtwisted\/} if it contains an overtwisted disk; otherwise it is
called {\em tight\/}.

An example of an overtwisted contact structure on $\R^3$ is the kernel
of the contact form given in cylindrical coordinates by
\[
\lambda = \cos r \, dz + r \sin r \, d\theta.
\]
This contact structure is invariant under translation in the $z$
direction.  Here the contact planes are horizontal on the $z$ axis,
but rotate infinitely many times as one moves out from the $z$ axis
along a horizontal ray.  An overtwisted disk is given by a horizontal
disk of radius $r$, where $r$ is a positive number such that $\sin r = 0$.

On the other hand, the standard contact structure defined by
\eqref{eqn:SCS} is tight, although this is less trivial to prove.
More generally, $(Y,\xi)$ is called (strongly symplectically) {\em
  fillable\/} if there is a compact symplectic 4-manifold $(X,\omega)$
with boundary $\partial X=Y$ as oriented manifolds, and a contact form
$\lambda$ for $(Y,\xi)$, such that $d\lambda=\omega|_Y$.  (In
particular $Y$ is a hypersurface of contact type in $X$, and an
associated Liouville vector field points out of the boundary of $X$.)
A fundamental theorem in the subject asserts that any fillable contact
structure is tight \cite{eft}.

\subsection{Simple examples}
\label{sec:SE}

To get some examples of contact forms on 3-manifolds, recall
from \S\ref{sec:OSWC} that any star-shaped hypersurface $Y$ in $\R^4$
has a contact form
\[
\lambda = \frac{1}{2}\sum_{i=1}^2(x_idy_i-y_idx_i)|_Y.
\]
The resulting contact structure on $Y\simeq S^3$ is tight, because it
is filled by the solid region that $Y$ bounds in $\R^4$. A theorem of
Eliashberg asserts that all tight contact structures on $S^3$ are
isotopic to this one.

If $Y$ is the unit sphere in $\R^4$, then the Reeb vector field is
tangent to the fibers of the Hopf fibration $S^3\to S^2$.  In
particular, there is a family of Reeb orbits parametrized by $S^2$.
This is in fact a special case of the circle bundle example that we
considered previously, and in particular it is not ``generic''.  If
one replaces the sphere with the ellipsoid
\begin{equation}
\label{eqn:ellipsoid}
\frac{x_1^2+y_1^2}{a_1} + \frac{x_2^2+y_2^2}{a_2} = 1
\end{equation}
where $a_1,a_2$ are positive real numbers, then the Reeb vector field
is given by
\[
R = 2\sum_{i=1}^2 a_i^{-1}\left(x_i\partial_{y_i} - y_i\partial_{x_i}\right).
\]
This vector field rotates in the $x_i,y_i$ plane at angular speed
$2a_i^{-1}$.  Thus if $a_1/a_2$ is irrational, then there are just two
embedded Reeb orbits, namely the circles $x_1=y_1=0$ and $x_2=y_2=0$.
The Weinstein conjecture says that we can not further modify the
contact form to eliminate these two remaining orbits without
introducing new ones.

We next consider some examples of contact forms on the 3-torus $T^3$.
Write $T^3=(\R/2\pi\Z)^3$ with coordinates $x,y,z$.  For each positive
integer $n$, define a contact form $\lambda_n$ on $T^3$ by
\begin{equation}
\label{eqn:lambdan}
\lambda_n \eqdef \cos(nz) dx + \sin(nz) dy.
\end{equation}
The associated Reeb vector field is given by
\begin{equation}
\label{eqn:ReebTorus}
R_n = \cos(nz) \partial_x + \sin(nz) \partial_y.
\end{equation}
We can regard $T^3$ as a $T^2$-bundle over $S^1$, where $z$ is the
coordinate on $S^1$ and $x,y$ are the fiber coordinates.  The Reeb
vector field is then a linear vector field tangent to each fiber,
whose slope rotates as $z$ increases.  Whenever the slope is rational,
the fiber is foliated by Reeb orbits.  That is, there is a circle of
embedded Reeb orbits for each $z$ such that
$\tan(nz)\in\Q\cup\{\infty\}$.  Again this is a non-generic situation,
and it turns out for any such $z$, one can perturb the contact form so
that the corresponding circle of Reeb orbits disintegrates into just
two Reeb orbits.  We will see in \S\ref{sec:CCH} below that the
contact structures $\xi_n\eqdef \Ker(\lambda_n)$ are pairwise
non-contactomorphic.  Also they are all tight.  Note that $\xi_1$ is
isomorphic to the canonical contact structure on the unit contangent
bundle of $T^2$ with a flat metric.

\begin{remark}
  One can use the above example to illustrate that Reeb vector fields
  are somewhat special.  To see how, note that in general if $R$ is the
  Reeb vector field associated to a contact form $\lambda$ on a
  3-manifold, then $R$ is volume preserving with respect to the volume
  form $\lambda\wedge d\lambda$, because the definition of Reeb vector
  field implies that the Lie derivative $\mathcal{L}_R\lambda=0$.  But
  not every volume-preserving vector field is a Reeb vector field.  In
  fact the Reeb vector field \eqref{eqn:ReebTorus} can easily be
  perturbed to a volume-preserving vector field with no closed orbits
  (which by the Weinstein conjecture cannot be a Reeb vector field).
  Namely, consider the vector field
\[
V = \cos(nz) \partial_x + (\sin(nz)+\varepsilon_1) \partial_y +
\varepsilon_2\partial_z
\]
where $\varepsilon_1,\varepsilon_2\neq 0$ and
$\varepsilon_1/\varepsilon_2$ is irrational.  Suppose that
$(x(t),y(t),z(t))$ is a trajectory of $V$.  Then
$z(t)=z(0)+\varepsilon_2t$, so if this is a closed orbit then the
period must be $2\pi k/\varepsilon_2$ for some positive integer $k$.
But the path $(x(t),y(t))$ moves in the sum of a circular motion with
period $2\pi/(n\varepsilon_2)$ and a vertical motion of speed
$\varepsilon_1$, so we have $(x(2\pi k/\varepsilon_2),y(2\pi
k/\varepsilon_2))=(0,2\pi k\varepsilon_1/\varepsilon_2)$.  Thus our
assumption that $\varepsilon_1/\varepsilon_2$ is irrational implies
that there is no closed orbit.
\end{remark}

\subsection{Classification of overtwisted contact structures}
\label{sec:cot}

A contact structure is a particular kind of oriented 2-plane field,
and if two contact structures are isotopic then the corresponding
oriented 2-plane fields are homotopic.  A remarkable theorem of
Eliashberg implies that for overtwisted contact structures, the
converse is true:

\begin{theorem}[Eliashberg \cite{eliashberg}]
  For any closed oriented 3-manifold $Y$, the inclusion of the set of
  overtwisted contact structures on $Y$ into the set of oriented
  $2$-plane fields on $Y$ is a homotopy equivalence.
\end{theorem}

A detailed exposition of the proof may be found in \cite{geiges}.  In
particular, this theorem implies that overtwisted contact structures
modulo isotopy are equivalent to homotopy classes of oriented 2-plane
fields.  Note that the latter always exist\footnote{In higher
  dimensions, not as much is known about which manifolds admit contact
  structures.  For example it was only in 2002 that odd-dimensional
  tori were shown to admit contact structures, by Bourgeois
  \cite{bourgeois02}.  In dimensions greater than three, there is a
  homotopy-theoretic obstruction: a closed oriented $(2n-1)$ manifold
  $Y$ admits a (cooriented) contact structure only if the structure
  group of the tangent bundle $TY$ reduces to $U(n-1)$.  I am not
  aware of any further known obstructions.  For some positive results
  in the 5-dimensional case see eg \cite[\S8]{geiges}.}, because an
oriented 3-manifold has trivial tangent bundle.

\begin{remark}
\label{rem:opf}
To prepare for the discussion of Seiberg-Witten theory later, it is
worth saying a bit more here about what the set of homotopy classes of
oriented 2-plane fields on a given closed oriented (connected) $Y$
looks like.  First note that a homotopy class of oriented 2-plane
fields on $Y$ is equivalent to a homotopy class of nonvanishing
sections of $TY$.  In particular, if $\xi$ and $\xi'$ are two oriented
2-plane fields on $Y$, then the primary obstruction to finding a
homotopy between them is an element of $H^2(Y;\Z)$.  If the primary
obstruction vanishes, then it turns out that the remaining obstruction
lives in $\Z/d$, where $d$ denotes the divisibility of
$c_1(\xi)=c_1(\xi')$ in $H^2(Y;\Z)$ mod torsion, which is always an
even integer.
\end{remark}

The classification of tight contact structures is more complicated,
and a subject of ongoing research.  In particular, the map from tight
contact structures to homotopy classes of oriented $2$-plane fields is
in general neither injective nor surjective.  Failure of injectivity
is illustrated by the contact structures $\xi_n$ on $T^3$ in
\S\ref{sec:SE}, which are pairwise non-contactomorphic even though
they all represent the same homotopy class of oriented 2-plane fields.
Failure of surjectivity follows for example from a much stronger
theorem of Colin-Giroux-Honda \cite{colin-giroux-honda}, which asserts
that on any given closed oriented 3-manifold there are only finitely
many homotopy classes of oriented 2-plane fields that contain tight
contact structures (even though there are always infinitely many
homotopy classes of oriented 2-plane fields, as follows from
Remark~\ref{rem:opf} above).

\subsection{Open book decompositions}
\label{sec:OB}

There is a useful classification of all contact three-manifolds, not
just the overtwisted ones, in terms of open book decompositions.

Let $\Sigma$ be a compact oriented connected surface with nonempty
boundary.  Let $\phi:\Sigma\to\Sigma$ be an orientation-preserving
diffeomorphism which is the identity near the boundary.  One can then
define a closed oriented three-manifold
\[
\begin{split}
Y_\phi &\eqdef [0,1]\times\Sigma/\sim,\\
(1,x)&\sim (0,\phi(x)) \quad\forall x\in\Sigma,\\
(t,x)&\sim (t',x) \quad\quad\forall x\in\partial\Sigma,\; t,t'\in[0,1],
\end{split}
\]
called an {\em open book\/}.  The image of a set $\{t\}\times\Sigma$
in $Y_\phi$ is called a {\em page\/}.  The boundaries of the different
pages are all identified with each other, to an oriented link in
$Y_\phi$ called the {\em binding\/}.  The map $\phi$ is called the
{\em monodromy\/} of the open book.  An {\em open book
  decomposition\/} of a closed oriented three-manifold $Y$ is a
diffeomorphism of $Y$ with an open book $Y_\phi$ as above.

\begin{definition}
  A contact structure $\xi$ on a closed three-manifold $Y$ is {\em
    compatible\/} with, or {\em supported\/} by, an open book
  decomposition $Y\simeq Y_\phi$ if $\xi$ is isotopic to a contact
  structure given by a contact form $\lambda$ such that:
\begin{itemize}
\item
The Reeb vector field is tangent to the binding (oriented positively).
\item
The Reeb vector field is transverse to
the interior of each page, intersecting positively.
\end{itemize}
\end{definition}

\begin{example}
The standard contact structure on the unit sphere described in
\S\ref{sec:SE} is compatible with an open book decomposition of $S^3$
in which $\Sigma$ is a disk and $\phi$ is the identity map.
\end{example}

A short argument by Thurston-Winkelnkemper shows that every open book
decomposition has a compatible contact structure.
In fact the compatible contact structure is determined up to isotopy
by the open book.  Moreover, a theorem of Giroux asserts that every
contact structure can be obtained in this way, and two open books
determine isotopic contact structures if and only if they are related
to each other by ``positive stabilizations''.  For more about this see
\cite{giroux,etnyre06}.

\subsection{Some previous results on the 3d Weinstein conjecture}
\label{sec:previous}

There is a long history of work proving the Weinstein conjecture for
various classes of contact three-manifolds.  Often one can further
show that Reeb orbits with certain properties must exist.  To give
just a few examples here:

\begin{theorem}[Hofer \cite{hofer93}]
\label{thm:hoferwc}
Let $(Y,\xi)$ be a closed contact 3-manifold in which $\xi$ is
overtwisted or $\pi_2(Y)\neq 0$.  Then for any contact form with
kernel $\xi$, there exists a {\em contractible\/} Reeb orbit.
\end{theorem}

\begin{theorem}[Abbas-Cieliebak-Hofer \cite{ach}]
\label{thm:ach}
Suppose $(Y,\xi)$ is supported by an open book in which the pages have
genus zero.  Then for any contact form with kernel $\xi$, there exists
a nonempty finite collection of Reeb orbits $\{\gamma_i\}$ with
\[
\sum_i [\gamma_i]=0\in H_1(Y).
\]
\end{theorem}

We remark that the above theorems find Reeb orbits as ends of
punctured holomorphic spheres in $\R\times Y$ with only positive ends,
cf \S\ref{sec:HC}.  In the first theorem there is only one puncture,
while in the second theorem the number of punctures can be any positive
integer.

Colin and Honda used linearized contact homology (see \S\ref{sec:CCH}
below) to prove the Weinstein conjecture for contact three-manifolds
supported by open books in which the monodromy is periodic, and in
many cases where the monodromy is pseudo-Anosov.  See
\cite{colin-honda} for the precise statement.  In fact, for many
contact structures supported by open books with pseudo-Anosov
monodromy, they proved a much stronger statement: that for any contact
form, there are infinitely many free homotopy classes of loops that
contain an embedded Reeb orbit.

\section{Some strategies for proving the Weinstein conjecture}
\label{sec:strategy}

One naive strategy for proving the Weinstein conjecture might be to
try the following:
\begin{itemize}
\item
Define some kind of count of Reeb orbits with appropriate signs.
\item
Show that this count is a topological invariant.
\item
Calculate this invariant and show that it is nonzero.
\end{itemize}
This strategy is too simple for at least two reasons.  First of all,
often there are actually infinitely many embedded Reeb orbits (see
\S\ref{sec:ILB}), so it is not clear how to obtain a well-defined
count of them.  Second, even if the above difficulty can be overcome,
the signed count might be zero, despite the existence of some Reeb
orbits with opposite signs that cannot be eliminated.  In general, as
one deforms the contact form, pairs of Reeb orbits can be created or
destroyed (more complicated bifurcations such as period-doubling are
also possible), and one needs some way of keeping track of when this
can happen.

A more refined strategy, which avoids the above two problems, is
as follows:
\begin{itemize}
\item Define some kind of chain complex which is generated by Reeb
  orbits, such that, roughly speaking, there are differentials between
  pairs of Reeb orbits that can potentially be destroyed in a
  bifurcation.
\item
Show that the homology of this chain complex is a topological
invariant.
\item
Compute this homology and show that it is nontrivial.
\end{itemize}
It turns out that there does exist a chain complex along these lines
which is sufficient to prove the Weinstein conjecture.  However it is
not the first chain complex that one might think of, and the proof that
it works uses Seiberg-Witten theory.  We will now attempt
to explain this story.

\section{Prototype for a chain complex: Morse homology}

The prototype for the type of chain complex we want to consider is
Morse homology, which we now review.  We will not give any proofs, as
these require a fair bit of analysis; details can be found in
\cite{schwarz}.  There is also an interesting history of the
development of the Morse complex, for which we refer the reader to
\cite{bott}.

\subsection{Morse functions}

Let $X$ be an $n$-dimensional closed smooth manifold and let
$f:X\to\R$ be a smooth function.  A {\em critical point\/} of $f$ is a
point $p\in X$ such that $0=df_p:T_pX\to\R$.  The basic goal of Morse
theory is to relate the critical points of $f$ to the topology of $X$.
A first question is, on a given $X$, what is the minimum number of
critical points that a smooth function $f$ can have?

To make this question easier, we can require that the critical points
of $f$ be ``generic'', in a sense which we now specify.  If $p\in X$
is a critical point of $f:X\to\R$, define the {\em Hessian\/}
\begin{equation}
\label{eqn:HessianPairing}
H(f,p): T_pX\tensor T_pX \longrightarrow \R
\end{equation}
as follows.  Let $\psi:X\to T^*X$ denote the section corresponding to
$df$.  Then $H(f,p)$ is the composition
\[
T_pX \stackrel{d\psi_p}{\longrightarrow} T_{(p,0)}T^*X = T_pX \oplus
T_p^*X \stackrel{\pi}{\longrightarrow} T_p^*X,
\]
where $\pi$ denotes the projection onto the second factor.  To be more
explicit, if $(x_1,\ldots,x_n)$ are local coordinates on $X$ centered
at $p$, then
\[
H(f,p)\left(\frac{\partial}{\partial x_i}, \frac{\partial}{\partial
    x_j}\right) = \frac{\partial^2f}{\partial x_i\partial x_j}.
\]
In particular, $H(f,p)$ is a symmetric bilinear form.  If one chooses
a metric on $X$, then the Hessian can be identified with a
self-adjoint operator
\begin{equation}
\label{eqn:HessianOperator}
H(f,p): T_pX \longrightarrow T_pX.
\end{equation}
The critical point $p$ is said to be {\em nondegenerate\/} if the
Hessian pairing \eqref{eqn:HessianPairing} is nondegenerate, or
equivalently the Hessian operator \eqref{eqn:HessianOperator} does not have
zero as an eigenvalue, or equivalently the graph of $df$ in $T^*X$ is
transverse to the zero section at $(p,0)$.  In particular,
nondegenerate critical points are isolated in $X$.

We say that $f$ is a {\em Morse function\/} if all of its critical
points are nondegenerate.  One can show that ``generic'' smooth
functions are Morse.  More precisely, the set of Morse functions is
open and dense in the set of all smooth functions $f:X\to\R$, with the
$C^\infty$ topology.  We can now ask, what is the minimum possible number of
critical points of a Morse function on $X$?

One could start by counting the critical points with signs.  If $p$ is
a (nondegenerate) critical point of $f$, define the {\em index\/} of
$p$, denoted by $\op{ind}(p)$, to be the maximal dimension of a
subspace on which the Hessian pairing \eqref{eqn:HessianPairing} is
negative definite, or equivalently the number of negative eigenvalues,
counted with multiplicity, of the Hessian operator
\eqref{eqn:HessianOperator}.  For example, a local minimum has index
$0$, and a local maximum has index $n$.  It turns out that an
appropriate sign with which to count an index $i$ critical point is
$(-1)^i$, and the signed count is the Euler characteristic of $X$.
That is, if $c_i(f)$ denotes the number of critical points of index
$i$, then
\[
\sum_{i=0}^n (-1)^ic_i(f) = \chi(X).
\]
One can prove this by choosing a metric on $X$, applying the
Poincar\'{e}-Hopf index theorem to the resulting gradient
vector field $\nabla f$, and checking that the sign of the zero of
$\nabla f$ at an index $i$ critical point is $(-1)^i$.  In
particular, the number of critical points is at least $|\chi(X)|$.
But if $\chi(X)=0$ then this tells us nothing.

\subsection{The Morse complex}
\label{sec:MC}

One way to obtain better lower bounds on the number of critical points
is to consider ``gradient flow lines'' between critical points and
package these into a chain complex.  Here is how this works.

Choose a metric $g$ on $X$ and use it to define the gradient vector
field $\nabla f$.  If $p$ and $q$ are critical points, a (downward)
{\em gradient flow line\/} from $p$ to $q$ is a map $\gamma:\R\to X$
such that
\[
\frac{d\gamma(s)}{ds} = \nabla f (\gamma(s))
\]
and $\lim_{s\to +\infty}\gamma(s)=p$ and $\lim_{s\to
  -\infty}\gamma(s)=q$.  Let $\mathcal{M}(p,q)$ denote the set of
downward gradient flow lines from $p$ to $q$.

If the metric $g$ is generic, then $\mathcal{M}(p,q)$ is naturally a
manifold of dimension
\begin{equation}
\label{eqn:dimMpq}
\dim \mathcal{M}(p,q) = \op{ind}(p)-\op{ind}(q).
\end{equation}
To see why, for $s\in\R$ let $\psi_s:X\to X$ denote the time $s$ flow
of the vector field $\nabla f$.  The {\em descending manifold\/} of a
critical point $p$ is the set
\[
\mathcal{D}(p) \eqdef \{x\in X \mid \lim_{s\to +\infty}\psi_s(x) =
p\}.
\]
Informally, this is the set of points in $X$ that ``can be reached by
downward gradient flow starting at $p$''.  Similarly, the {\em
  ascending manifold\/} of a critical point $q$ is
\[
\mathcal{A}(q) \eqdef \{x\in X \mid \lim_{s\to -\infty}\psi_s(x)=q\}.
\]
One can show that the descending manifold $\mathcal{D}(p)$ is a
smoothly embedded open ball in $X$ of dimension $\op{ind}(p)$, and
also $T_p\mathcal{D}(p)$ is the negative eigenspace of the Hessian
\eqref{eqn:HessianOperator}.  Likewise, $\mathcal{A}(q)$ is a smoothly
embedded open ball of dimension $n-\op{ind}(q)$.

It follows from the definitions that there is a bijection
\[
\mathcal{M}(p,q) \stackrel{\simeq}{\longrightarrow} \mathcal{D}(p) \cap
\mathcal{A}(q)
\]
sending a flow line $\gamma$ to the point $\gamma(0)\in X$.  Consequently,
if $\mathcal{D}(p)$ is transverse to $\mathcal{A}(q)$, then equation
\eqref{eqn:dimMpq} follows by dimension counting.  The pair $(f,g)$ is
said to be {\em Morse-Smale\/} if $\mathcal{D}(p)$ is transverse to
$\mathcal{A}(q)$ for every pair of critical points $p,q$.  One can show
that for a given Morse function $f$, for generic metrics $g$ the pair
$(f,g)$ is Morse-Smale.  Henceforth we assume by default that this
condition holds.  Observe also that $\R$ acts on $\mathcal{M}(p,q)$ by
precomposition with translations, and if $p\neq q$ then this action is
free, so that
\begin{equation}
\label{eqn:dimension}
\dim (\mathcal{M}(p,q)/\R) = \op{ind}(p) - \op{ind}(q) - 1.
\end{equation}
In particular, if $\op{ind}(q)\ge \op{ind}(p)$ then $\mathcal{M}(p,q)$ is
empty (except when $p=q$), and if $\op{ind}(q)=\op{ind}(p)-1$ then
$\mathcal{M}(p,q)/\R$ is discrete.

We now define the {\em Morse complex\/} $C_*^{\op{Morse}}(X,f,g)$ as
follows.  The chain module in degree $i$ is the free $\Z$-module
generated by the index $i$ critical points:
\[
C_i^{\op{Morse}}(X,f,g) \eqdef \Z\{p\in X\mid df_p=0, \; \op{ind}(p)=i\}.
\]
The differential $\partial: C_i^{\op{Morse}}(X,f,g) \to
C_{i-1}^{\op{Morse}}(X,f,g)$ is defined by counting gradient flow
lines as follows:  if $p$ is an index $i$ critical point, then
\[
\partial p \eqdef \sum_{\op{ind}(p)-\op{ind}(q)=1}
\#\frac{\mathcal{M}(p,q)}{\R} \cdot q.
\]
Here `$\#$' denotes a signed count.  We wil not say more about the
signs here except to note that the signs are determined by choices of
orientations of the descending manifolds of the critical points, but
the chain complexes resulting from different sign choices are
canonically isomorphic to each other.

To show that this chain complex is well-defined, one must prove that
$\partial$ is well-defined, ie $\mathcal{M}(p,q)/\R$ is finite
whenever $\op{ind}(p)-\op{ind}(q)=1$, and one must also prove that
$\partial^2=0$.  The first step is to prove a compactness theorem
which asserts that given critical points $p\neq q$, any sequence
$\{\gamma_n\}_{n=1}^\infty$ in $\mathcal{M}(p,q)/\R$ has a subsequence
which converges in an appropriate sense to a ``$k$-times broken flow
line'' from $p$ to $q$.  This is a tuple
$(\widehat{\gamma}_0,\ldots,\widehat{\gamma}_k)$ for some $k\ge 0$,
such that there are critical points $p=r_0,r_1,\ldots,r_{k+1}=q$ for
which $\widehat{\gamma}_i\in \mathcal{M}(r_i,r_{i+1})/\R$ is a
nonconstant flow line.

If $\op{ind}(p)-\op{ind}(q)=1$, then the Morse-Smale condition implies
that there are no $k$-times broken flow lines with $k>0$, so
$\mathcal{M}(p,q)/\R$ is compact, and hence finite, so $\partial$ is
well-defined.

If $\op{ind}(p)-\op{ind}(q)=2$, then $\mathcal{M}(p,q)/\R$ is not
necessarily compact and may contain a sequence converging to a
once-broken flow line.  However we can add in these broken flow lines
to obtain a compactification $\overline{\mathcal{M}(p,q)/\R}$ of
$\mathcal{M}(p,q)/\R$.  it turns out that this is a compact oriented
1-manifold with boundary, whose boundary as an oriented manifold is
\begin{equation}
\label{eqn:compactification1}
\partial \overline{\mathcal{M}(p,q)/\R} =
\bigcup_{\op{ind}(p)-\op{ind}(r)=1}\frac{\mathcal{M}(p,r)}{\R} \times
\frac{\mathcal{M}(r,q)}{\R}.
\end{equation}
The main ingredient in the proof of this is a gluing theorem asserting
that each broken flow line $(\widehat{\gamma}_0,\widehat{\gamma}_1)$
with $\widehat{\gamma}_0\in\mathcal{M}(p,r)/\R$ and
$\widehat{\gamma}_1\in\mathcal{M}(r,q)/\R$ can be ``patched'' to an
unbroken flow line in a unique end of $\mathcal{M}(p,q)/\R$.  (One
also has to show that $\mathcal{M}(p,q)/\R$ can be oriented so that
the orientations on both sides of \eqref{eqn:compactification1}
agree.)

It follows from the boundary equation \eqref{eqn:compactification1}
that $\partial^2=0$.  Namely, counting the points on both sides of
\eqref{eqn:compactification1} with signs gives
\begin{equation}
\label{eqn:FTDT}
\#\partial\overline{\mathcal{M}(p,q)/\R} =
\sum_{\op{ind}(p)-\op{ind}(r)=1} \langle \partial p,r\rangle
\langle\partial r,q\rangle.
\end{equation}
Here $\langle \partial p,r\rangle\in\Z$ denotes the coefficient of $r$
in $\partial p$.  Thus the right hand side of \eqref{eqn:FTDT} is, by
definition, the coefficient $\langle\partial^2p,q\rangle$.  On the
other hand since a compact oriented 1-manifold has zero boundary
points counted with signs, the left hand side of \eqref{eqn:FTDT} is
zero.

\subsection{Morse homology}

The {\em Morse homology\/} $H_*^{\op{Morse}}(X,f,g)$ is the homology
of the above chain complex.

\begin{example}
  Consider a Morse function $f:S^2\to\R$ with two index $2$ critical
  points $x_1,x_2$, one index $1$ critical point $y$, and one index
  $0$ critical point $z$.  One can visualize $f$ as the height
  function on a ``heart-shaped'' sphere embedded in $\R^3$.  Pick any
  metric $g$ on $S^2$; it turns out that $(f,g)$ will automatically
  be Morse-Smale in this example.  There is (up to reparametrization)
  a unique downward gradient flow line from each $x_i$ to $y$.  There
  are two gradient flow lines from $y$ to $z$.  The latter turn out to
  have opposite signs, and so for suitable orientation choices the
  Morse complex is given by
\begin{gather*}
  C_2^{\op{Morse}}=\Z\{x_1,x_2\}, \quad C_1^{\op{Morse}}=\Z\{y\},
  \quad C_0^{\op{Morse}}=\Z\{z\},\\
\partial x_1 = y, \quad \partial x_2 = -y, \quad \partial y = 0.
\end{gather*}
Thus $H_2^{\op{Morse}}\simeq \Z$, generated by $x_1+x_2$;
$H_1^{\op{Morse}}=0$; and $H_0^{\op{Morse}}=\Z$, generated by $z$.
\end{example}

The above example illustrates a fundamental theorem in the subject:

\begin{theorem}
\label{thm:MorseHomology}
There is a canonical isomorphism between Morse homology and singular
homology,
\begin{equation}
\label{eqn:MHSH}
H_*^{\op{Morse}}(X,f,g) \simeq H_*(X).
\end{equation}
\end{theorem}
An immediate corollary is that there must be enough critical points to
generate a chain complex whose homology is $H_*(X)$.  In particular:

\begin{corollary}
  If $f$ is a Morse function on a closed smooth manifold $X$, then
  $c_i(f)\ge \op{rank}(H_i(X))$.
\end{corollary}

\subsection{Continuation maps}
\label{sec:continuation}

We will later construct analogues of Morse homology which generally do
not have interpretations in terms of previously known invariants such
as singular homology.  What is most important here as a model for
these later constructions is that the Morse homology
$H_*^{\op{Morse}}(X,f,g)$ is a topological invariant of $X$ which does
not depend on $f$ or $g$.  One can prove this directly, without making
the comparison with singular homology, as follows.

Let $(f_0,g_0)$ and $(f_1,g_1)$ be two Morse-Smale pairs.  Let
$\{(f_s,g_s)\mid s\in\R\}$ be a smooth family of pairs of functions
and metrics on $X$ such that $(f_s,g_s)=(f_0,g_0)$ for $s\le 0$ and
$(f_s,g_s)=(f_1,g_1)$ for $s\ge 1$.  We do not (and in general cannot)
assume that the pair $(f_s,g_s)$ is Morse-Smale for all $s$.  One now
defines a map
\[
\Phi:C_*^{\op{Morse}}(X,f_1,g_1) \longrightarrow
C_*^{\op{Morse}}(X,f_0,g_0),
\]
called the {\em continuation map\/}, as follows.  If $p_0$ is an index $i$
critical point of $f_0$ and $p_1$ is an index $i$ critical point of
$f_1$, then $\langle\Phi(p_1),p_0\rangle$ is a signed count of maps
$\gamma:\R\to X$ satisfying
\[
\frac{d\gamma(s)}{ds} = \nabla f_s(\gamma(s))
\]
and $\lim_{s\to -\infty}\gamma(s) = p_0$ and
$\lim_{s\to+\infty}\gamma(s)=p_1$.  Here the gradient of $f_s$ is
computed using the metric $g_s$.  Similarly to the proof that
$\partial$ is well defined and $\partial^2=0$, one can show that if
the family of metrics $\{g_s\}$ is generic then $\Phi$ is a
well-defined chain map.  For example, if the family $\{(f_s,g_s)\}$ is
constant then $\Phi$ is the identity map.  One can also show that up
to chain homotopy, $\Phi$ depends only on the homotopy class of the
path $\{(f_t,g_t)\}$ rel endpoints, which since the space of pairs
$(f,g)$ is contractible means that up to chain homotopy $\Phi$ depends
only on $(f_0,g_0)$ and $(f_1,g_1)$.  Finally, related considerations
show that if
\[
\Phi': C_*^{\op{Morse}}(X,f_2,g_2) \longrightarrow
C_*^{\op{Morse}}(X,f_1,g_1)
\]
is the continuation map induced by a generic path from $(f_1,g_1)$ to
$(f_2,g_2)$, then the composition $\Phi\Phi'$ is chain homotopic to
the continuation map induced by a path from $(f_0,g_0)$ to
$(f_2,g_2)$.  If $(f_2,g_2)=(f_0,g_0)$, then it follows that
$\Phi\Phi'$ and $\Phi'\Phi$ are chain homotopic to the respective
identity maps.  Thus $\Phi$ induces an isomorphism on homology.

It follows from the above homotopy properties of continuation maps
that the Morse homologies $H_*^{\op{Morse}}(X,f,g)$ for different
Morse-Smale pairs $(f,g)$ are canonically isomorphic to each other via
continuation maps.  It turns out that the isomorphism \eqref{eqn:MHSH}
commutes with these continuation maps.

\subsection{Spectral flow}
\label{sec:SF}

We now recall a more analytical way to understand the dimension
formula \eqref{eqn:dimension}, which is used in the infinite
dimensional variants of Morse theory to be considered later.

Given critical points $p$ and $q$, let $\mc{P}$ denote the space of
smooth paths $\gamma:\R\to X$ satisfying
$\lim_{s\to+\infty}\gamma(s)=p$ and $\lim_{s\to-\infty}\gamma(s)=q$.
An element $\gamma\in\mc{P}$ is a gradient flow line if and only if it
satisfies the equation
\begin{equation}
\label{eqn:GFL}
\frac{d\gamma(s)}{ds} - \nabla f(\gamma(s)) = 0.
\end{equation}
Note that the left hand side of \eqref{eqn:GFL} is a section of the
pullback bundle $\gamma^*TX$ over $\R$.

Now assume that $\gamma$ is a flow line.  To understand flow lines
near $\gamma$, we consider the linearization of the equation
\eqref{eqn:GFL}.  This could be regarded a linear operator $D_\gamma$
which sends the tangent space to $\mc{P}$ at $\gamma$ (namely the
space of smooth sections of $\gamma^*TX$ which converge to $0$ as
$s\to \pm\infty$) to $\gamma^*TX$.  But in order to apply the tools of
functional analysis one wants to work with suitable Banach space
completions, for example to regard $D_\gamma$ as an operator from the
Sobolev space of $L^2_1$ sections of $\gamma^*TX$ to the space of
$L^2$ sections of $\gamma^*TX$.  To write this operator more
explicitly, choose a trivialization of $\gamma^*TX$ which converges to
some fixed trivializations of $T_pX$ and $T_qX$ as $s\to\pm\infty$.
Then with respect to this trivialization, the operator $D_\gamma$ has
the form
\begin{equation}
\label{eqn:Dgamma}
D_\gamma=\partial_s + A_s: L^2_1(\R,\R^n) \longrightarrow L^2(\R,\R^n),
\end{equation}
where $A_s$ is an $n\times n$ matrix depending on $s\in\R$.  Moreover
as $s\to\pm\infty$, the matrix $A_s$ converges to the (negative) Hessians,
\begin{equation}
\label{eqn:AHessian}
\lim_{s\to+\infty}A_s = -H(f,p), \quad\quad \lim_{s\to-\infty}A_s =
-H(f,q).
\end{equation}

The first step in the analytic treatment is
to show that if $f$ is Morse then the operator $D_\gamma$ is Fredholm;
and if the metric $g$ is generic, then for all flow lines $\gamma$ the
operator $D_\gamma$ is surjective, and $\mc{M}(p,q)$ is a smooth
manifold with $T_\gamma\mc{M}(p,q)=\Ker(D_\gamma)$.  In particular
$\dim\mc{M}(p,q)=\op{ind}(D_\gamma)$.

To compute the index of $D_\gamma$, there is a general principle that
if $\{A_s\}$ is a family of operators on a Hilbert space satisfying
appropriate technical hypotheses, then the index of $\partial_s+A_s$
from $L^2_1$ to $L^2$ is the {\em spectral flow\/} of the family
$\{A_s\}$, which roughly speaking is the number of eigenvalues of
$A_s$ that cross from negative to positive as $s$ goes from $-\infty$
to $+\infty$, minus the number of eigenvalues that cross from positive
to negative.  For some theorems realizing this principle in different
situations see eg \cite{robbin-salamon,km}.  For the operator
\eqref{eqn:Dgamma}, no additional technical hypotheses are necessary
and the spectral flow is simply the number of positive eigenvalues of
$\lim_{s\to+\infty}A_s$ minus the number of positive eigenvalues of
$\lim_{s\to-\infty}A_s$.  Using \eqref{eqn:AHessian} one obtains
\[
\dim\mc{M}(p,q) = \op{ind}(D_\gamma) = -\op{ind}(p) -(-\op{ind}(q)),
\]
which recovers \eqref{eqn:dimension}.

More generally we will need to apply the ``index=spectral flow''
principle to
\[
D = \partial_s + A_s: L^2_1(\R\times Y,E)\longrightarrow L^2(Y,E)
\]
where $\{A_s\}$ is a family of elliptic first-order differential
operators on a vector bundle $E$ over a manifold $Y$ parametrized by
$s\in\R$, which converge as $s\to\pm\infty$ to self-adjoint operators
with zero kernel.  The operator \eqref{eqn:Dgamma} corresponds to the
case where $Y$ is a point.  Later in this article $Y$ will be a circle
(for cylindrical contact homology) or a three-manifold (for
Seiberg-Witten Floer homology).

\section{First attempt at a chain complex: cylindrical contact
  homology}
\label{sec:CCH}

Let $Y$ be a closed oriented $3$-manifold and let $\lambda$ be a
contact form on $Y$.  We would like to define an analogue of the Morse
complex on the loop space of $Y$, which is generated by Reeb orbits,
and whose differential counts an appropriate notion of ``flow lines''
between them.  Although the analogy with Morse homology breaks down
somewhat, this idea leads naturally to the definition of cylindrical
contact homology.  This theory can be used to prove the Weinstein
conjecture in many cases.  Although cylindrical contact homology can
be defined for contact manifolds of any odd dimension, for
definiteness we stick to the three-dimensional case.

\subsection{Nondegenerate Reeb orbits}
\label{sec:NRO}

We now explain the appropriate analogue of nondegenerate critical
point in this context.

Let $\gamma:\R/T\Z\to Y$ be a Reeb orbit.  Let $\psi_T:Y\to Y$ denote
the diffeomorphism obtained by flowing along the Reeb vector field for
time $T$.  This preserves the contact form, because by the definition
of Reeb vector field, the Lie derivative $\mc{L}_R\lambda=0$.  It
follows that for any $t\in \R/T\Z$, we have a symplectic linear map
\[
P_\gamma\eqdef d\psi_T:(\xi_{\gamma(t)},d\lambda) \longrightarrow
(\xi_{\gamma(t)},d\lambda).
\]
This map is called the {\em linearized return map\/}.

Another way to describe this map is as follows.  Let $D$ be a small
embedded disk in $Y$ centered at $\gamma(t)$ and transverse to
$\gamma$, such that $T_{\gamma(t)}D=\xi_{\gamma(t)}$.  For $x\in D$
close to the center, there is a unique point in $D$ which is reached
by following the Reeb flow for a time close to $T$.  This gives a
partially defined ``return map'' $\phi:D\to D$ which is defined near
the origin.  The derivative of this map at the origin is the
linearized return map $P_\gamma$.

We say that the Reeb orbit $\gamma$ is {\em nondegenerate\/} if
$P_\gamma$ does not have $1$ as an eigenvalue.  This condition does
not depend on the choice of $t\in\R/T\Z$, because the linearized
return maps for different $t$ are conjugate to each other.  If the
Reeb orbit $\gamma$ is nondegenerate then it is isolated, because Reeb
orbits close to $\gamma$ give rise to fixed points of the map $\phi$,
and the condition that $1-d\phi$ is invertible at the origin implies
that $\phi$ has no fixed points near the origin.

One can show that for a given contact structure $\xi$, for generic
contact forms $\lambda$, all Reeb orbits are nondegenerate.  We will
always assume unless otherwise stated that all Reeb orbits are
nondegenerate.

One can classify (nondegenerate) Reeb orbits into three types,
according to the eigenvalues $\lambda,\lambda^{-1}$ of the linearized
return map:
\begin{itemize}
\item
{\em elliptic\/}: $\lambda,\lambda^{-1}=e^{\pm 2\pi i \theta}$.
\item
{\em positive hyperbolic\/}: $\lambda,\lambda^{-1}>0$.
\item
{\em negative hyperbolic\/}: $\lambda,\lambda^{-1}<0$.
\end{itemize}

\subsection{Holomorphic cylinders}
\label{sec:HC}

The appropriate analogue of ``gradient flow line'' in this context is
a certain kind of holomorphic cylinder in $\R\times Y$.  We now
explain what these are.

In general, recall that a {\em complex structure\/} on an
even-dimensional real vector bundle $E\to X$ is a bundle map $J:E\to
E$ satisfying $J^2=-1$.  An {\em almost complex structure\/} on an
even-dimensional manifold $X$ is a complex structure $J$ on the
tangent bundle $TX$.  A {\em holomorphic curve\/}\footnote{Often these
  are instead called ``pseudoholomorphic curves'' or ``$J$-holomorphic
  curves'', in order to emphasize the fact that we are working with {\em
    almost\/} complex geometry, as opposed to complex manifolds.} in
$(X,J)$ is a map $u:\Sigma\to X$ where $\Sigma$ is a surface with an
almost complex structure $j$ (ie a Riemann surface), and $J\circ du =
du \circ j$.  Two holomorphic curves $u:(\Sigma,j)\to X$ and
$u':(\Sigma',j')\to X$ are considered equivalent if there is a
biholomorphic map $\phi:(\Sigma,j)\to (\Sigma',j')$ with $u=u'\circ
\phi$.  If $u$ is an embedding then the equivalence class of $u$ is
determined by its image.  That is, an embedded holomorphic curve in
$(X,J)$ is just a 2-dimensional submanifold $C\subset X$ such that
$J(TC)=TC$.

Returning now to the situation of interest:

\begin{definition}
\label{def:AACS}
Let $Y$ be a three-manifold with a contact form $\lambda$.  An almost
complex structure $J$ on the $4$-manifold $\R\times Y$ is {\em
  admissible\/} if:
\begin{enumerate}
\item
$J$ sends $\xi$ to itself, rotating $\xi$ positively with respect to
the orientation of $\xi$ given by $d\lambda$.
\item
If $s$ denotes the $\R$ coordinate on $\R\times Y$, then
$J(\partial_s)=R$.
\item
$J$ is invariant under the $\R$ action on $\R\times Y$ that translates
 $s$.
\end{enumerate}
\end{definition}
Note that the space of such $J$ is nonempty and contractible.  Indeed,
the choice of such a $J$ is equivalent to the choice of a complex
structure on the $2$-plane bundle $\xi$ over $Y$ which rotates
positively with respect to $d\lambda$.  Fix an admissible almost
complex structure $J$ on $Y$ below.

Observe that if $\gamma$ is an embedded Reeb orbit, then
$\R\times\gamma$ is an embedded holomorphic cylinder in $\R\times Y$.
This follows from condition (2) above.  More generally, we can study
holomorphic curves in $\R\times Y$ that are asymptotic to such
$\R$-invariant cylinders, or covers thereof, as the $\R$ coordinate
goes to plus or minus infinity.  To define what we mean by this,
consider a ``half-cylinder'' $[0,\infty)\times S^1$ or
$(-\infty,0]\times S^1$ with coordinates $s,t$, with the almost
complex structure $j$ sending $\partial_s$ to $\partial_t$.  Let
$\pi_\R:\R\times Y\to\R$ and $\pi_Y:\R\times Y\to Y$ denote the two
projections.  If $u:\Sigma\to\R\times Y$ is a holomorphic curve and if
$\gamma$ is a Reeb orbit (not necessarily embedded), we define a {\em
  positive end\/} of $u$ at $\gamma$ to be an end of $\Sigma$ which
can be parametrized as $[0,\infty)\times S^1$ with the almost complex
structure $j$ as above, such that
$\lim_{s\to\infty}\pi_\R(s,\cdot)=\infty$, and
$\lim_{s\to\infty}\pi_Y(s,\cdot)$ is a reparametrization of $\gamma$.
Likewise, a {\em negative end\/} of $u$ at $\gamma$ is an end of
$\Sigma$ which can be parametrized as $(-\infty,0]\times S^1$, with
the almost complex structure $j$ as above, such that
$\lim_{s\to-\infty}\pi_\R(s,\cdot)=-\infty$, and
$\lim_{s\to-\infty}\pi_Y(s,\cdot)$ is a reparametrization of $\gamma$.

If $\gamma_+$ and $\gamma_-$ are two Reeb orbits, define
$\mathcal{M}(\gamma_+,\gamma_-)$ to be the set of holomorphic
cylinders in $\R\times Y$ that have a positive end at $\gamma_+$ and a
negative end at $\gamma_-$.  It turns out that these holomorphic
cylinders are the appropriate ``gradient flow lines'' from $\gamma_+$
to $\gamma_-$.  Note that there is an $\R$ action on
$\mathcal{M}(\gamma_+,\gamma_-)$ given by translating the $\R$
coordinate on the target\footnote{This is not to be confused
  with the $\R\times S^1$ action on the set of holomorphic maps
  $\R\times S^1\to\R\times Y$ given by compositions with translations
  of the domain, which we have already modded out by in our definition
  of holomorphic curve.} space $\R\times Y$.  This action is free except on
the $\R$-invariant cylinders $\R\times\gamma$ in
$\mathcal{M}(\gamma,\gamma)$.

\subsection{The action functional}

Holomorphic cylinders in $\mc{M}(\gamma_+,\gamma_-)$ can be regarded
as ``gradient flow lines'' of the {\em symplectic action\/} functional
on the loop space of $Y$ defined by
\begin{equation}
\label{eqn:action}
\mc{A}(\gamma) \eqdef \int_{S^1}\gamma^*\lambda
\end{equation}
for $\gamma:S^1\to Y$.  Without trying to make this analogy precise,
let us just note the following essential lemma:

\begin{lemma}
\label{lem:actionDecreases}
  Suppose there exists a holomorphic cylinder $u\in
  \mathcal{M}(\gamma_+,\gamma_-)$.  Then
\[
\mathcal{A}(\gamma_+)\ge \mathcal{A}(\gamma_-),
\]
with equality if and only if $\gamma_+=\gamma_-$ and the image of $u$
is an $\R$-invariant cylinder.
\end{lemma}

\begin{proof}
Let $u:\R\times S^1\to \R\times Y$ be a holomorphic cylinder in
$\mathcal{M}(\gamma_+,\gamma_-)$.  By Stokes' theorem,
\[
\mc{A}(\gamma_+) - \mc{A}(\gamma_-) = \int_{\R\times S^1}u^*d\lambda.
\]
(The integral on the right converges because of the asymptotics of
$u$.)  By condition (1) in the definition of admissible almost complex
structure, $u^*d\lambda\ge 0$ at each point in $\R\times S^1$, with
equality only where $u$ is tangent to $\R$ cross the Reeb direction.
\end{proof}

Later the symplectic action will play a key role in Taubes's proof of
the Weinstein conjecture.

\subsection{The chain complex}

We can now define an analogue of the Morse complex in this setting.
We will give a ``quick and dirty'' definition to save space; for the
more general context into which this definition fits, see
\cite{egh}.

To start, for reasons we will explain below, one must discard certain
``bad'' Reeb orbits for the construction to work:

\begin{definition}
  A Reeb orbit $\gamma$ is said to be {\em bad\/} if it is the
  $k$-fold iterate of a negative
  hyperbolic orbit with $k$ even.  Otherwise $\gamma$ is said to be {\em
    good\/}.
\end{definition}

Now fix $\Gamma\in H_1(Y)$.  Define $CC(Y,\lambda,\Gamma)$ to be the
free $\Q$-module generated by the good Reeb orbits $\gamma$
representing the homology class $\Gamma$.  One then defines a
differential
\[
\partial:CC(Y,\lambda,\Gamma) \longrightarrow CC(Y,\lambda,\Gamma)
\]
as follows.  Fix a generic admissible almost complex structure $J$ on
$\R\times Y$.  If $\gamma_+$ is a good Reeb orbit, then
\begin{equation}
\label{eqn:CCHD}
\partial\gamma_+ \eqdef
\sum_{\gamma_-}{k_{\gamma_-}}n(\gamma_+,\gamma_-)\gamma_-.
\end{equation}
Here the sum is over good Reeb orbits $\gamma_-$, and $k_{\gamma}$
denotes the unique positive integer such that $\gamma$ is the
$k_{\gamma}$-fold iterate of an embedded Reeb orbit.  Meanwhile,
$n(\gamma_+,\gamma_-)\in\Q$ is a signed count of holomorphic cylinders
in $\mc{M}(\gamma_+,\gamma_-)/\R$ that live in zero-dimensional moduli
spaces.  Multiply covered cylinders are counted with weight $\pm1$
divided by the covering multiplicity.  (We will not explain the signs
here.)  The homology of this chain complex, when defined (see below),
is called {\em cylindrical contact homology\/}, and we denote it by
$CH(Y,\xi,\Gamma)$.

The following is a special case of a result to be proved in
\cite{bee}, see \cite[\S3.2]{colin-honda} for the statement,
asserting that a more general theory called ``linearized contact
homology'' is well-defined.

\begin{theorem}
\label{thm:CCH}
Suppose there are no contractible Reeb orbits.  Then $\partial$ is
well-defined\footnote{The expert reader may worry that even for
  generic $J$, multiply covered holomorphic cylinders might have
  smaller index than the cylinders that they cover, leading to
  failure of the compactness needed to show that $\partial$
  is defined.  It turns out that this does not happen for holomorphic
  cylinders in the symplectization of a contact $3$-manifold.  However
  this is an issue in defining the continuation maps and chain
  homotopies needed to prove the invariance statement in
  Theorem~\ref{thm:CCH}, for which some abstract perturbations of the
  moduli spaces are needed.},
$\partial^2=0$, and the homology $CH(Y,\xi,\Gamma)$ depends only on
$Y$, the contact structure $\xi$, and the homology class $\Gamma$, and
not on the contact form $\lambda$ or almost complex structure $J$.
\end{theorem}

A few comments are in order.  First, the factors of $k_\gamma$ in
\eqref{eqn:CCHD} are needed to make $\partial^2=0$ work, because when
one glues two (not multiply covered) holomorphic cylinders along a
Reeb orbit $\gamma$ which is the $k$-fold iterate of an embedded Reeb
orbit, there are $k$ different ways to glue, compare \S\ref{sec:MC}.
This is also why bad Reeb orbits need to be discarded: it turns out
that these $k$ different gluings all have the same sign when $\gamma$
is good, but have cancelling signs when $\gamma$ is bad.  Finally, the
assumption that there are no contractible Reeb orbits ensures that the
necessary compactness arguments go through, by ruling out bubbling off
of holomorphic planes.

\subsection{The index}

Unlike Morse homology, cylindrical contact homology is not
$\Z$-graded.  Rather, it is relatively $\Z/d(2c_1(\xi))$-graded, where
$d(2c_1(\xi))$ denotes the divisibility of $2c_1(\xi)$ in $H^2(Y;\Z)$
mod torsion.  This means that any two generators $\gamma_+$ and
$\gamma_-$ have a well-defined relative grading, which can be regarded
as the grading difference between $\gamma_+$ and $\gamma_-$, and which
is an element of $\Z/d(2c_1(\xi))$.  In this sense the differential
$\partial$ decreases the grading by $1$.  The reason why there is no
absolute grading analogous to the Morse index is that the analogue of
the Hessian in this setting has infinitely many negative and
infinitely many positive eigenvalues.  Nonetheless it still makes
sense to define the relative grading of $\gamma_+$ and $\gamma_-$ to
be the expected dimension of the moduli space of holomorphic cylinders
$\mc{M}(\gamma_+,\gamma_-)$ that represent some relative homology
class $Z$.  This is given by a certain spectral flow (see
\S\ref{sec:SF}), which is computed by a topological formula which we
will not state here.  It is only defined modulo $d(2c_1(\xi))$,
because if $Z'$ is a different relative homology class of cylinder
then the corresponding expected dimensions differ by $\langle
2c_1(\xi),Z-Z'\rangle$, where $Z-Z'\in H_2(Y)$ denotes the difference
between the two relative homology classes.

There is also a canonical absolute $\Z/2$-grading: a Reeb orbit has
odd grading if it is positive hyperbolic, and even grading if it is
elliptic or negative hyperbolic.  The differential $\partial$ also has
degree $-1$ with respect to this $\Z/2$-grading.

\subsection{Examples}
\label{sec:CHE}

(1) Consider the contact form $\lambda_n$ on $T^3$ defined in
\eqref{eqn:lambdan}.  Recall that all Reeb orbits represent homology
classes of the form $(a,b,0)\in H_1(T^3)$ with $(a,b)\neq (0,0)$.  As
a consequence, the cylindrical contact homology
$CH_*(T^3,\xi_n,\Gamma)$ is nonzero only for $\Gamma$ of this form.
Fix such a class $\Gamma=(a,b,0)$.  All Reeb orbits $\gamma$ in the
homology class $\Gamma$ have symplectic action
\[
\mc{A}(\gamma) = 2\pi\sqrt{a^2+b^2}.
\]
So by Lemma~\ref{lem:actionDecreases}, there are no non-$\R$-invariant
holomorphic cylinders between them.  Now the cylindrical contact
homology is not yet defined because $\lambda_n$ is a ``Morse-Bott''
contact form whose Reeb orbits are not isolated but rather appear in
one-parameter families.  But one can show, see \cite{bmb}, that one
can perturb $\lambda_n$ to a contact form $\lambda_n'$ such that each
of the $n$ circles of Reeb orbits in the homology class $\Gamma$
splits into two Reeb orbits, one elliptic and one positive hyperbolic;
there are no other Reeb orbits in the class $\Gamma$, except possibly
for some much longer Reeb orbits which can be disregarded in the
computation using a direct limit argument; and the differential on
$CC_*(Y,\lambda_n',\Gamma)$ vanishes for any choice of admissible
almost complex structure $J$.  (Each $S^1$ of Reeb orbits is perturbed
using a Morse function $f:S^1\to\R$ with two critical points which
become the two Reeb orbits after perturbation.  There are two
holomorphic cylinders from the elliptic orbit to the hyperbolic orbit
after perturbation, counting with opposite signs, corresponding to the
Morse complex of $f$ on $S^1$.)  The conclusion is that the
$\Z/2$-graded cylindrical contact homology is given by
\[
CH_{\op{even}}(Y,\xi_n,(a,b,0)) \simeq CH_{\op{odd}}(Y,\xi_n,(a,b,0))
\simeq \Q^n.
\]

It now follows from Theorem~\ref{thm:CCH} that the different contact
structures $\xi_n$ are pairwise non-contactomorphic.  Also, they all
satisfy the Weinstein conjecture.  Because for any contact form
$\lambda$ with $\xi_n=\Ker(\lambda)$, either there is no contractible
Reeb orbit, in which case the cylindrical contact homology is
well-defined and nonzero, or else there is a contractible Reeb orbit.

In fact one can do a little better.  There is a generalization of
cylindrical contact homology, called {\em linearized contact
  homology\/}, which can sometimes be defined even when there are
contractible Reeb orbits, by adding certain correction terms to the
cylindrical contact homology differential.  Linearized contact
homology can be used in this example to prove that for every contact
form there is a Reeb orbit in the homology class $(a,b,0)$.

(2) Let us compute the cylindrical contact homology of the irrational
ellipsoid $(S^3,\lambda)$ from \S\ref{sec:SE}.  Here of course we must
take $\Gamma=0$.  Strictly speaking Theorem~\ref{thm:CCH} is not
applicable here because all Reeb orbits are contractible, but it turns
out that the cylindrical contact homology is still defined in this
example because all Reeb orbits satisfy a certain Conley-Zehnder index
condition which rules out troublesome holomorphic planes.

Denote the two embedded Reeb orbits by $\gamma_1$ and $\gamma_2$.
These are elliptic. Let $\gamma_i^k$ denote the $k$-fold iterate of
$\gamma_i$, see \S\ref{sec:terminology}.  The chain complex
$CC_*(S^3,\lambda,0)$ has a relative $\Z$-grading, and because
$\Gamma=0$ there is in fact a canonical way to normalize it to an
absolute $\Z$-grading.  This grading is given as follows: there are
positive irrational numbers $\phi_1$ and $\phi_2$ with
$\phi_1\phi_2=1$ (in the notation of equation \eqref{eqn:ellipsoid},
$\phi_1$ and $\phi_2$ are $a_1/a_2$ and $a_2/a_1$) such that the
grading of $\gamma_i^k$ is
\begin{equation}
\label{eqn:combinatorics1}
|\gamma_i^k| = 2\lfloor k(1+\phi_i)\rfloor.
\end{equation}
It is an exercise to deduce from \eqref{eqn:combinatorics1} that there
is one generator of each positive even grading.  Hence the
differential vanishes for degree reasons, and
\[
CH_*(S^3,\xi,0) \simeq \left\{\begin{array}{cl} \Q, & *=2,4,\ldots,\\
0, & \mbox{otherwise}.
\end{array}
\right.
\]
Here $\xi$ denotes the contact structure determined by $\lambda$,
which as mentioned in \S\ref{sec:SE} is the unique tight contact
structure on $S^3$.  See \cite{bce} for some applications of contact
homology to the Reeb dynamics of other contact forms determining this
contact structure.

(3) As mentioned previously, Colin-Honda used linearized contact
homology to prove the Weinstein conjecture in many cases.  However it
is not currently known whether linearized contact homology can be used
to prove the Weinstein conjecture for all tight contact
three-manifolds.  (It turns out that in the overtwisted case
linearized contact homology is never defined, but the failure of
linearized contact homology to be defined implies the existence of a
contractible Reeb orbit, reproducing Hofer's
Theorem~\ref{thm:hoferwc}.)  Taubes's proof of the Weinstein
conjecture for all contact three-manifolds needs Seiberg-Witten
theory.

\section{The big picture surrounding Taubes's proof of the Weinstein
  conjecture}
\label{sec:BigPicture}

\subsection{Seiberg-Witten invariants of four-manifolds}

The Seiberg-Witten invariants (and the conjecturally equivalent
Ozsvath-Szabo invariants \cite{os4d}) are the most powerful tool
currently available for distinguishing smooth four-manifolds.  To
briefly outline what these are, let $X$ be a closed oriented connected
smooth four-manifold.  If $b_2^+(X)>1$, then the Seiberg-Witten
invariant of $X$ is, after certain orientation choices have been made,
a function
\[
SW: \Spinc(X)\longrightarrow \Z.
\]
Here $b_2^+(X)$ denotes the maximal dimension of a subspace of
$H_2(X;\R)$ on which the intersection pairing is positive definite.
Also $\Spinc(X)$ denotes the set of spin-c structures on $X$.  This is
an affine space over $H^2(X;\Z)$, which we will say more about in
\S\ref{sec:spinc}.  Given a spin-c structure $\frak{s}$, the
Seiberg-Witten invariant $SW(X,\frak{s})$ is defined by appropriately
counting solutions to the Seiberg-Witten equations on $X$.  (We will
not write down the Seiberg-Witten equations here, but we will see a
three-dimensional version of them in \S\ref{sec:SWE}, and a
two-dimensional version in \S\ref{sec:vortex}.)  The Seiberg-Witten
invariants depend only on the diffeomorphism type of $X$, and can
distinguish many pairs of smooth four-manifolds that are homeomorphic
but not diffeomorphic.  A detailed definition of the Seiberg-Witten
invariants of four-manifolds may be found in \cite{morgan}, and a
review of the early results in this area is given in
\cite{donaldson-survey}.

\subsection{Taubes's ``SW=Gr'' theorem}

Suppose now that our four-manifold $X$ is symplectic.  A major result
of Taubes from the 1990's asserts that the Seiberg-Witten invariants
of $X$ are equivalent to a certain count of holomorphic curves.

To describe this result, let $\omega$ denote the symplectic form on
$X$.  It turns out that $\omega$ defines a bijection
\[
\imath_\omega: \Spinc(X) \stackrel{\simeq}{\longrightarrow} H_2(X).
\]
Now let $J$ be a generic, $\omega$-tame almost complex structure on
$X$; the tameness condition means that $\omega(v,Jv)>0$ for all
nonzero tangent vectors $v$.  Taubes then defines a ``Gromov
invariant''
\[
Gr: H_2(X)\longrightarrow\Z
\]
roughly as follows.  For each $Z\in H_2(X)$, the integer $Gr(X,Z)$ is
a count of certain $J$-holomorphic curves $C$ in $X$ representing the
homology class $Z$.  The curves $C$ that are counted are required to
be embedded, except that they may include multiple covers of tori of
square zero.  Such a curve $C$ is not required to be connected, but
each component of $C$ must live in a zero-dimensional moduli space.
Each such holomorphic curve $C$ is counted with a certain integer
weight.  (The weight is $\pm 1$, except when $C$ includes multiply
covered tori, in which case the weight is given by a somewhat
complicated recipe.)  Taubes's theorem is now:

\begin{theorem}[Taubes]
\label{thm:swgr}
Let $X$ be a closed connected symplectic 4-manifold with $b_2^+(X)>1$.
Then for each $\frak{s}\in\Spinc(X)$ we have
\[
SW(X,\frak{s}) = Gr(X,\imath_\omega(\frak{s})).
\]
\end{theorem}

The proof of this theorem is given in \cite{taubes:sw=gr}; for an
introduction to it see \cite{ht97}, and for a discussion from a
physics perspective see \cite{witten}.  The basic idea is that one
deforms the Seiberg-Witten equations using a large multiple of the
symplectic form, and shows that solutions to the deformed
Seiberg-Witten equations ``concentrate along'' holomorphic curves.
Taubes's proof of the Weinstein conjecture involves a
three-dimensional version of this.

\subsection{Three-dimensional Seiberg-Witten theory}

Let $Y$ be a closed oriented 3-manifold.  The set of spin-c structures
on $Y$, denoted by $\Spinc(Y)$, is again an affine space over
$H^2(Y;\Z)$.  A spin-c structure on $Y$ determines a product spin-c
structure on $\R\times Y$. There are then various ways to define
topological invariants of the 3-manifold $Y$ by studying solutions to
the Seiberg-Witten equations on the noncompact 4-manifold $\R\times
Y$.

To start, one can consider $\R$-invariant solutions to the
Seiberg-Witten equations on $\R\times Y$.  These are equivalent to
solutions to the three-dimensional Seiberg-Witten equations on $Y$,
which we will discuss in \S\ref{sec:SWE}.  When $b_1(Y)>0$, one can
count these solutions with signs to obtain the three-dimensional {\em
  Seiberg-Witten invariant\/}\footnote{When $b_1(Y)=1$, for
  ``torsion'' spin-c structures, ie spin-c structures whose first
  Chern class is torsion, see \S\ref{sec:spinc}, this invariant
  depends on the choice of one of two possible ``chambers''.}
$SW:\Spinc(Y)\to\Z$.  As mostly\footnote{ie after identifying spin-c
  structures that differ by torsion elements in $H^2(Y;\Z)$} shown in
\cite{donaldson,meng-taubes} and fully shown in \cite{turaev}, this
invariant agrees with the Turaev torsion of $Y$, which generalizes the
Alexander polynomial of a knot, and which is explicitly computable in
terms of the determinants of certain matrices of polynomials
associated to a triangulation of $Y$.

To get more interesting invariants of $Y$, one observes that solutions
to the Seiberg-Witten equations on $Y$ are critical points of a
certain functional $\mathcal{F}$ on a configuration space associated
to $Y$.  Moreover, solutions to the four-dimensional Seiberg-Witten
equations on $\R\times Y$ (not necessarily $\R$-invariant) are
equivalent to gradient flow lines of this functional.  It turns out
that one can then define a version of Morse homology for the
functional $\mathcal{F}$, called {\em Seiberg-Witten Floer
  homology\/}, which we will say more about in \S\ref{sec:SWF}.  The
analytical details of this construction are highly nontrivial and have
been carried out by Kronheimer-Mrowka \cite{km}.  Roughly speaking,
the Seiberg-Witten Floer homology is the homology of a chain complex
which is generated by solutions to the Seiberg-Witten equations on
$Y$, and whose differential counts solutions to the Seiberg-Witten
equations on $\R\times Y$.  When $b_1(Y)>0$, for non-torsion spin-c
structures, the Euler characteristic of the Seiberg-Witten Floer
homology agrees with the Seiberg-Witten invariant discussed above.
However Seiberg-Witten Floer homology can also be defined for torsion
spin-c structures and without any assumption on $b_1(Y)$.

In fact there are two basic versions of Seiberg-Witten Floer theory
that one can define, depending on how one treats ``reducibles'', see
\S\ref{sec:gauge}. The first is denoted by $\check{HM}_*(Y)$, and
pronounced ``HM-to''; this assigns a relatively graded homology group
$\check{HM}_*(Y,\frak{s})$ to each spin-c structure $\frak{s}$ on $Y$,
which is conjecturally isomorphic to the Heegaard Floer homology
$HF_*^+(Y,\frak{s})$ defined in \cite{os}.  The second is pronounced
``HM-from'', denoted by $\widehat{HM}_*(Y,\frak{s})$, and
conjecturally isomorphic to the Heegaard Floer homology
$HF^-_*(Y,\frak{s})$.  For non-torsion spin-c structures, there are no
reducibles and $\check{HM}_*$ and $\widehat{HM}_*$ are the same.  For
any spin-c structure $\frak{s}$, there is a canonical isomorphism
\[
\check{HM}_*(-Y,\frak{s}) = \widehat{HM}^{-*}(Y,\frak{s}),
\]
where $-Y$ denotes $Y$ with its orientation reversed, and
$\widehat{HM}^*$ denotes the cohomological version of $\widehat{HM}_*$
obtained by dualizing the chain complex.

\subsection{Embedded contact homology}
\label{sec:ech}

Suppose now that our three-manifold $Y$ is equipped with a contact
form $\lambda$.  Recall from \S\ref{sec:CF} that the four-manifold
$\R\times Y$ then has a symplectic form $d(e^s\lambda)$, where $s$
denotes the $\R$ coordinate.  It is natural to seek an analogue of
Taubes's ``SW=Gr'' theorem for the noncompact symplectic four-manifold
$\R\times Y$.  That is, one would like to understand the
Seiberg-Witten Floer homology of $Y$ in terms of holomorphic curves in
$\R\times Y$.  For this purpose it is appropriate to use an admissible
almost complex structure as in Definition~\ref{def:AACS}.

The analogy suggests that the Seiberg-Witten Floer homology of $Y$
should be isomorphic to the homology of a chain complex whose
differential counts certain (mostly) embedded holomorphic curves in
$\R\times Y$, and which is generated by certain $\R$-invariant
holomorphic curves in $\R\times Y$, that is to say unions of Reeb
orbits.  The resulting theory is called {\em embedded contact
  homology\/}, or ECH for short.  There is some resemblance between
ECH and cylindrical contact homology; but among other differences, ECH
does not require the holomorphic curves that are counted to be
cylinders, while cylindrical contact homology does not require them to
be embedded.

To say a bit more about what ECH is, assume as usual that all Reeb
orbits are nondegenerate.

\begin{definition}
An {\em orbit set\/} is a finite set of pairs
$\alpha=\{(\alpha_i,m_i)\}$ where:
\begin{itemize}
\item
The $\alpha_i$'s are distinct embedded Reeb orbits.
\item
The $m_i$'s are positive integers.
\end{itemize}
The homology class of the orbit set $\alpha$ is defined by
\[
[\alpha] \eqdef \sum_im_i[\alpha_i] \in H_1(Y).
\]
The orbit set $\alpha$ is {\em admissible\/} if
$m_i=1$ whenever $\alpha_i$ is hyperbolic (see \S\ref{sec:NRO}).
\end{definition}

Given $\Gamma\in H_1(Y)$, the embedded contact homology
$ECH_*(Y,\lambda,\Gamma)$ is the homology of a chain complex which is
freely generated over $\Z$ by admissible orbit sets $\alpha$ with
$[\alpha]=\Gamma$.  The differential counts certain (mostly) embedded
holomorphic curves in $\R\times Y$.  It has a relative
$\Z/d(c_1(\xi)+2\op{PD}(\Gamma))$ grading, where `$d$' denotes
divisibility in $H^2(Y;\Z)$ mod torsion.  For the full definition of
this theory, see \cite{t3} for an overview and \cite{ir,obg1} for
more details.

Note that the empty set of Reeb orbits is a legitimate
generator\footnote{It is always a cycle in the ECH chain complex, by
  the argument in Lemma~\ref{lem:actionDecreases}.  Its homology class
  in $ECH_*(Y,\lambda,0)$ agrees with the invariant of contact
  structures in Seiberg-Witten Floer homology under the isomorphism
  \eqref{eqn:echswf}, and conjecturally also with the contact
  invariant in Heegaard Floer homology.}  of the ECH chain complex
with $\Gamma=0$.

\begin{example}
\label{ex:IE}
Consider again the irrational ellipsoid $(S^3,\lambda)$ as discussed
in \S\ref{sec:CHE}.  Of course we must take $\Gamma=0$.  The
generators of the ECH chain complex have the form
$\gamma_1^{m_1}\gamma_2^{m_2}$ where $\gamma_1$ and $\gamma_2$ are the
two embedded Reeb orbits, $m_1$ and $m_2$ are nonnegative integers,
and $\gamma_1^{m_1}\gamma_2^{m_2}$ is shorthand for the orbit set
consisting of the pair $(\gamma_1,m_1)$ (when $m_1\neq 0$) together
with the pair $(\gamma_2,m_2)$ (when $m_2\neq 0$).  In this case the
chain complex has a relative $\Z$-grading, and since $\Gamma=0$ this
has a canonical refinement to an absolute $\Z$-grading such that the
grading of the empty set is zero.  It is shown in \cite{wh} that the
grading is given by
\[
\left|\gamma_1^{m_1}\gamma_2^{m_2}\right| = 2
\left(m_1+m_2+m_1m_2+\sum_{i=1}^2\sum_{k=1}^{m_i}\lfloor k\phi_i\rfloor\right)
\]
where $\phi_1$ and $\phi_2$ are as in \S\ref{sec:CHE}.  One can deduce
from this formula, see \cite{wh}, that there is one generator of each
nonnegative even grading, so that
\[
ECH_*(S^3,\lambda,0) \simeq \left\{\begin{array}{cl} \Z, &
    *=0,2,\ldots,\\
0, & \mbox{otherwise}
\end{array}\right.
\]
\end{example}

In general it turns out that ECH depends only on $Y$, the contact
structure $\xi$, and the homology class $\Gamma$, and not on the
choice of contact form $\lambda$ or admissible almost complex
structure $J$.  This follows from a much stronger result recently
proved by Taubes \cite{echswf1,echswf234}, which is the analogue of
Gr=SW in this setting:

\begin{theorem}[Taubes]
\label{thm:echswf}
  Let $Y$ be a closed oriented 3-manifold with a contact form
  $\lambda$ such that all Reeb orbits are nondegenerate.  Then for
  each $\Gamma\in H_1(Y)$ there is an isomorphism
\begin{equation}
\label{eqn:echswf}
ECH_*(Y,\lambda,\Gamma) \simeq
\widehat{HM}^{-*}(Y,\frak{s}_\xi + \op{PD}(\Gamma)),
\end{equation}
up to a grading shift\footnote{In fact both sides of
  \eqref{eqn:echswf} have canonical absolute gradings by homotopy
  classes of oriented 2-plane fields, see \cite{km,ir}, and it is
  natural to conjecture that the isomorphism \eqref{eqn:echswf}
  respects these gradings.  It is further shown in \cite{echswf5} that
  the isomorphism \eqref{eqn:echswf} respects some additional
  structures on embedded contact homology and Seiberg-Witten Floer
  cohomology, namely the aforementioned contact invariants, the ``$U$
  maps'', the actions of $H_1(Y)$ mod torsion, and the twisted
  versions.}.
\end{theorem}
Here $\xi$ denotes the contact structure determined by $\lambda$, and
$\frak{s}_\xi$ is a spin-c structure determined by $\xi$, see
Example~\ref{ex:SPF} below.  

\subsection{Significance for the Weinstein conjecture}

To prove the Weinstein conjecture, it is enough to show that the
embedded contact homology is nontrivial.  More precisely, if
$(Y,\lambda)$ were a counterexample to the Weinstein conjecture, then
one would have
\begin{equation}
\label{eqn:trivialECH}
ECH(Y,\lambda,\Gamma) \simeq \left\{\begin{array}{cl} \Z, &
    \Gamma=0,\\
0, & \Gamma\neq 0.
\end{array}\right.
\end{equation}
Here the $\Z$ corresponds to the empty set of Reeb orbits.  However,
by Theorem~\ref{thm:echswf}, the ECH cannot be trivial as in
\eqref{eqn:trivialECH}, because the Seiberg-Witten Floer homology is
always infinitely generated:

\begin{theorem}
{\em (Kronheimer-Mrowka \cite[\S35.1]{km})\/}
\label{thm:nontrivial}
Let $Y$ be a closed oriented 3-manifold and let $\frak{s}$ be a spin-c
structure with $c_1(\frak{s})$ torsion.  Then
$\widehat{HM}^*(Y,\frak{s})$ is nonzero for infinitely many values of
the grading $*$, which are bounded from above.
\end{theorem}

Here $c_1(\frak{s})$ denotes the first Chern class of the spin-c
structure, which is defined in \S\ref{sec:spinc} below.  In terms of
the correspondence between $\Spinc(Y)$ and $H_1(Y)$ in
\eqref{eqn:echswf}, one has
\[
c_1(\frak{s}_\xi+\op{PD}(\Gamma)) = c_1(\xi) + 2\op{PD}(\Gamma).
\]
Since $TY$ is trivial, one can always find a spin-c structure
$\frak{s}$ such that $c_1(\frak{s})=0$, and in particular $c_1(\frak{s})$
is torsion.  Thus Theorems~\ref{thm:echswf} and \ref{thm:nontrivial}
imply the following version of the Weinstein conjecture:

\begin{itemize}
\item {\em Let $Y$ be a closed oriented $3$-manifold with a contact form
  $\lambda$ {\em such that all Reeb orbits are nondegenerate\/}.  Let
  $\Gamma\in H_1(Y)$ such that $c_1(\xi)+2\op{PD}(\Gamma)\in
  H^2(Y;\Z)$ is torsion.  (Such $\Gamma$ always exist.)  Then there is
  a nonempty {\em admissible\/} orbit set $\alpha$ with $[\alpha]=\Gamma$.\/}
\end{itemize}

Note that the fact that ECH is infinitely generated does not imply
that there are infinitely many embedded Reeb orbits, as shown by the
irrational ellipsoid in Example~\ref{ex:IE}.

In fact one does not need the full force of the isomorphism in
Theorem~\ref{thm:echswf} to prove the Weinstein conjecture; rather one
just needs a way of passing from generators of Seiberg-Witten Floer
homology to ECH generators.  This is what Taubes's original proof of
the Weinstein conjecture in \cite{tw} establishes, yielding a proof of
the following theorem, which is slightly different than the statement
above:

\begin{theorem}[Taubes]
\label{thm:tw}
Let $Y$ be a closed oriented $3$-manifold with a contact form
$\lambda$.  Let $\Gamma\in H_1(Y)$ such that
$c_1(\xi)+2\op{PD}(\Gamma)\in H^2(Y;\Z)$ is torsion.  (Such $\Gamma$
always exist.)  Then there is a nonempty orbit set $\alpha$ with
$[\alpha]=\Gamma$.
\end{theorem}

It is interesting to compare Theorem~\ref{thm:tw} with Theorems
\ref{thm:hoferwc} and \ref{thm:ach}, which for certain contact
structures produce a nonempty orbit set $\alpha$ with $[\alpha]=0$.
In a sequel \cite{tw2} to the paper proving the Weinstein conjecture,
Taubes uses more nontriviality results for Seiberg-Witten Floer
homology to find nonempty orbit sets in some other homology classes.

Our goal in \S\ref{sec:SW3D}-\ref{sec:MD} below is to explain
Taubes's proof of Theorem~\ref{thm:tw}.

\section{The three-dimensional Seiberg-Witten equations}
\label{sec:SW3D}

To proceed further, we now need to recall the three-dimensional
Seiberg-Witten equations.

\subsection{Spin-c structures}
\label{sec:spinc}

Let $Y$ be a closed oriented connected 3-manifold, and choose a
Riemannian metric on $Y$.

\begin{definition}
  A {\em spin-c structure\/} on $Y$ is a pair $\frak{s} = ({\mathbb
    S},\rho)$ where ${\mathbb S}$ is a rank $2$ Hermitian vector
  bundle on $Y$, called the {\em spinor bundle\/} (a section $\psi$ of
  ${\mathbb S}$ is often called a {\em spinor\/}), and
\[
\rho: TY \longrightarrow \op{End}({\mathbb S})
\]
is a bundle map, called {\em Clifford multiplication\/}, such that:
\begin{enumerate}
\item
If $a,b\in T_yY$, then
\[
\rho(a)\rho(b) + \rho(b)\rho(a) = -2\langle a,b\rangle.
\]
\item
If $e_1,e_2,e_3$ is an oriented orthonormal frame for $T_yY$,
then
\[
\rho(e_1)\rho(e_2)\rho(e_3)=1.
\]
\end{enumerate}
\end{definition}

Properties (1) and (2) of Clifford multiplication are equivalent to
the following: For each $y\in Y$, there is an oriented orthonormal
frame $e_1,e_2,e_3$ for $T_yY$, and a basis for ${\mathbb S}_y$, in
which Clifford multiplication is given by the Pauli matrices
\begin{equation}
\label{eqn:PauliMatrices}
\rho(e_1) = \begin{pmatrix} i & 0 \\ 0 & -i \end{pmatrix},
\quad\quad
\rho(e_2) = \begin{pmatrix} 0 & -1 \\ 1 & 0 \end{pmatrix},
\quad\quad
\rho(e_3) = \begin{pmatrix} 0 & i \\ i & 0 \end{pmatrix}.
\end{equation}

\begin{example}
\label{ex:SPF}
  An oriented $2$-plane field $\xi\subset TY$ determines a spin-c
  structure $\frak{s}_\xi$ as follows.  The spinor bundle is defined
  by
\begin{equation}
\label{eqn:SPF}
{\mathbb S} \eqdef \underline{\C} \oplus \xi,
\end{equation}
where $\underline{\C}$ denotes the trivial complex line bundle, and
$\xi$ is regarded as a Hermitian line bundle using its orientation and
the Riemannian metric on $Y$.  Clifford multiplication is defined at a
point $y\in Y$ by the equations \eqref{eqn:PauliMatrices}, where
$e_1,e_2,e_3$ are an oriented orthonormal frame for $T_yY$ such that
$e_2,e_3$ is an oriented orthonormal basis for $\xi_y$, and the basis
for ${\mathbb S}_y$ is given in terms of the decomposition
\eqref{eqn:SPF} by $(1,e_2)$.
\end{example}

Two spin-c structures $({\mathbb S},\rho)$ and $({\mathbb S}',\rho')$
are isomorphic if there is a Hermitian vector bundle isomorphism
$\phi:{\mathbb S} \stackrel{\simeq}{\to} {\mathbb S}'$ such that
$\rho'(v)\circ\phi=\phi\circ\rho(v)$ for every tangent vector $v$.
Let $\Spinc(Y)$ denote the set of isomorphism classes of spin-c
structures on $Y$.  This does not depend on the choice of Riemannian
metric on $Y$.  There is an action of $H^2(Y;\Z)$ on $\Spinc(Y)$
defined as follows: Given $\alpha\in H^2(Y;\Z)$, let $L_\alpha$ denote
the complex line bundle on $Y$ with $c_1(L_\alpha)=\alpha$, assign it
a Hermitian metric, and define
\[
\alpha\cdot({\mathbb S},\rho) \eqdef ({\mathbb S}\tensor L_\alpha,
\rho\tensor 1).
\]
It turns out that this action is free and transitive, so that
$\Spinc(Y)$ is an affine space over $H^2(Y;\Z)$.

\begin{remark}
\label{rem:HPF}
If $\mathcal{P}(Y)$ denotes the set of homotopy classes of oriented
$2$-plane fields on $Y$, then the map $\mathcal{P}(Y)\to \Spinc(Y)$ defined
in Example~\ref{ex:SPF} is surjective.  Two oriented $2$-plane fields
give rise to isomorphic spin-c structures if and only if they are
homotopic over the $2$-skeleton of $Y$ (for some triangulation),
compare Remark~\ref{rem:opf}.
\end{remark}

Here is an alternate definition of a spin-c structure, on an oriented
manifold $Y$ of any dimension $n>1$ with a Riemannian metric.  Let $F\to
Y$ denote the {\em frame bundle\/}, whose fiber over $y\in Y$ is the
set of orientation-preserving linear isometries
$\R^n\stackrel{\simeq}{\to} Y_y$.  Note that $F$ is a principal
$\SO(n)$-bundle over $Y$, where $\SO(n)$ acts on the right on $F$ by
precomposition with automorphisms of $\R^n$.  Now the Lie group
$\Spinc(n)$ is defined by
\begin{equation}
\label{eqn:spinc}
\Spinc(n) \eqdef \Spin(n) \times_{\Z/2} \U(1).
\end{equation}
Here $\Spin(n)$ is the connected double cover of $\SO(n)$; and $\Z/2$
acts on $\Spin(n)$ as the nontrivial covering transformation, and on
$\U(1)$ as multiplication by $-1$.  A spin-c structure on $Y$ is then
defined to be a lift of $F$ to a principal $\Spinc(n)$-bundle, ie a
principal $\Spinc(n)$-bundle $\widetilde{F}\to Y$ together with a map
$\widetilde{F}\to F$ which commutes with the group actions and the
projections to $Y$.

When $n=3$, this definition is equivalent to the previous one.  In
particular, given a lift $\widetilde{F}$ of the frame bundle, the
spinor bundle and Clifford multiplication are recovered as follows.
We can identify $\Spin(3)=\SU(2)$ and $\Spinc(3)=\U(2)$.  The spinor
bundle is then associated to $\widetilde{F}$ via the fundamental
representation of $\U(2)$.  On the other hand the tangent bundle of
$Y$ is associated to $\widetilde{F}$ by the representation of $\U(2)$
on $\R^3$ given by the projection $\Spinc(3)\to\SO(3)$.  Clifford
multiplication is then defined on these associated bundles using a
model linear map $\R^3\to\End(\C^2)$ that sends the standard basis
vectors of $\R^3$ to the three Pauli matrices
\eqref{eqn:PauliMatrices}.

\subsection{The Dirac operator}

Let $\frak{s}=({\mathbb S},\rho)$ be a spin-c structure.

\begin{definition}
  A {\em spin-c connection\/} on $\frak{s}$ is a Hermitian connection
  on ${\mathbb S}$ such that the associated covariant deriviative
  $\nabla_A$ is compatible with Clifford multiplication in the
  following sense: If $v$ is a section of $TY$ and if $\psi$ is a section
  of ${\mathbb S}$, then
\[
\nabla_A(\rho(v)\psi) = \rho(\nabla v)\psi + \rho(v)\nabla_A\psi.
\]
Here $\nabla v$ denotes the covariant derivative of $v$ with respect
to the Levi-Civita connection on $TY$.
\end{definition}

It follows from the above definition that any two spin-c connections
differ by an imaginary-valued $1$-form on $Y$.  It is not hard to show
that spin-c connections exist.  In fact a spin-c connection is
equivalent to a Hermitian connection on the $\U(1)$-bundle
$\det({\mathbb S})$.  One can see this by using the second definition
of spin-c structure and noting from \eqref{eqn:spinc} that the Lie
algebra of $\Spinc(3)=\U(2)$ is the sum of the Lie algebras of
$\SO(3)$ and of $\U(1)$.  A connection on ${\mathbb S}$ is then
determined by a connection on $TY$ (which we take to be the
Levi-Civita connection) and a connection on the complex line bundle
associated to the determinant map $\U(2)\to\U(1)$, namely
$\det({\mathbb S})$.  So in the notation $\nabla_A$, we regard $A$ as
a connection on $\det({\mathbb S})$.  From this perspective, adding an
imaginary-valued $1$-form $a$ to the connection $A$ on $\det({\mathbb
  S})$ adds $a/2$ to the associated spin-c connection $\nabla_A$ on
${\mathbb S}$.

\begin{definition}
  Given a connection $A$ on $\det({\mathbb S})$, define the {\em Dirac
    operator\/} $D_A$ to be the composition
\[
C^\infty(Y,{\mathbb S}) \stackrel{\nabla_A}{\longrightarrow}
C^\infty(Y,T^*X\tensor {\mathbb S}) \stackrel{\rho}{\longrightarrow}
C^\infty(Y,{\mathbb S}).
\]
Here the Clifford action is extended to cotangent vectors using the metric.
\end{definition}

A key property of the Dirac operator is that its square is the
``connection Laplacian'', plus some zeroth order terms involving
curvature.  More precisely, it satisfies the
Bochner-Lichnerowitz-Weitzenbock formula
\begin{equation}
\label{eqn:BLW}
D_A^2\psi = \nabla_A^*\nabla_A\psi + \frac{s}{4}\psi
- \frac{1}{2}\rho(*F_A)\psi.
\end{equation}
Here $s$ denotes the scalar curvature of $Y$, which is a real-valued
function; $F_A$ denotes the curvature of the connection $A$ on
$\det({\mathbb S})$, which is an imaginary-valued closed $2$-form on
$Y$; and $*$ denotes the Hodge star.

Another important property is that the Dirac operator is formally
self-adjoint: if $\psi_1$ and $\psi_2$ are two spinors, then
\[
\int_Y\langle D_A\psi_1,\psi_2\rangle d\op{vol} =
\int_Y\langle\psi_1,D_A\psi_2\rangle d\op{vol}.
\]
For much more about Dirac operators, see eg \cite{bgv,lm}.

\subsection{The Seiberg-Witten equations}
\label{sec:SWE}

Fix a spin-c structure $({\mathbb S},\rho)$.  The Seiberg-Witten
equations concern a pair $(A,\psi)$ where $A$ is a connection on
$\det({\mathbb S})$ and $\psi$ is a section of ${\mathbb S}$.

\begin{definition}
  The (unperturbed) {\em Seiberg-Witten equations\/} for the pair
  $(A,\psi)$ are
\[
\begin{split}
D_A\psi &= 0,\\
*F_A &= \langle \rho(\cdot)\psi, \psi\rangle.
\end{split}
\]
\end{definition}

Note that it follows from the properties of
Clifford multiplication that if $\psi$ is any spinor then $\langle
\rho(\cdot)\psi, \psi\rangle$ is an imaginary-valued $1$-form.

\begin{remark}
\label{rem:conventions}
Conventions for the Seiberg-Witten equations vary in the literature
(and sometimes are not completely explicit).  For example, one could
multiply one side of the second equation by a positive constant, and
the solutions to the equations would be equivalent via rescaling the
spinor.  However the sign is crucial: switching the sign in the second
equation would ruin certain a priori estimates on the solutions, such
as those in Lemma~\ref{lem:APE} below, which play a key role in the
subject.
\end{remark}

One also needs to consider certain perturbations of the equations.  In
particular, one often needs to make small perturbations in order to
obtain transversality, while Taubes's proof of the Weinstein
conjecture will involve a large perturbation.

\begin{definition}
  Let $\mu$ be a real closed $2$-form on $Y$. The {\em Seiberg-Witten
    equations with perturbation $\mu$\/} for the pair $(A,\psi)$ are
\begin{equation}
\label{eqn:SW}
\begin{split}
D_A\psi &= 0,\\
*F_A &= \langle \rho(\cdot)\psi, \psi\rangle + i{*}\mu.
\end{split}
\end{equation}
\end{definition}

\subsection{Gauge transformations}
\label{sec:gauge}

The equations \eqref{eqn:SW} have a large amount of symmetry.  Namely,
the Seiberg-Witten equations are defined on the {\em configuration
  space\/}
\[
\mathcal{C} \eqdef \op{Conn}(\det({\mathbb S})) \times
C^\infty(Y,{\mathbb S}),
\]
where $\op{Conn}(\det({\mathbb S}))$ denotes the set of Hermitian
connections on $\det({\mathbb S})$.  Define the {\em gauge group\/}
\[
\mathcal{G} \eqdef C^\infty(Y,S^1).
\]
This can be regarded as the automorphism group of the spin-c structure
$({\mathbb S},\rho)$.  As such it has a natural action on the
configuration space $\mathcal{C}$ defined as follows: If $g:Y\to S^1$
is in $\mathcal{G}$, then regarding $S^1$ as the unit circle in $\C$,
one defines
\[
g\cdot (A,\psi) \eqdef (A - 2 g^{-1}dg, g \psi).
\]

\begin{lemma}
  The set of solutions to the Seiberg-Witten equations \eqref{eqn:SW}
  is invariant under the action of the gauge group $\mathcal{G}$.
\end{lemma}

\begin{proof}
The curvature equation is invariant because the curvature of a connection
is invariant under gauge transformations.  The Dirac equation is
invariant because
\begin{equation}
\label{eqn:DiracGauge}
\begin{split}
D_{A-2g^{-1}dg}(g\psi) &= \rho((\nabla_A-g^{-1}dg)g\psi)\\
&= \rho(dg\tensor\psi + g\nabla_A\psi - g^{-1}dg\tensor g\psi)\\
&= \rho(g\nabla_A\psi) \\
&= gD_A\psi.
\end{split}
\end{equation}
\end{proof}

Two solutions to the Seiberg-Witten equations are called {\em gauge
  equivalent\/} if they differ by the action of an element of
$\mathcal{G}$.  In general, one studies solutions only modulo gauge
equivalence.

Observe that the action of $\mathcal{G}$ on $\mathcal{C}$ is free,
except that configurations $(A,0)$ have $S^1$ stabilizer given by the
constant maps $Y\to S^1$.  To keep track of this, a configuration
$(A,\psi)$ is called {\em reducible\/} if $\psi\equiv 0$, and {\em
  irreducible\/} otherwise.

\begin{lemma}
\label{lem:reducibles}
\begin{enumerate}
\item Reducible solutions exist if and only if $[\mu]=-2\pi
  c_1(\frak{s})$ in $H^2(Y;\R)$.
\item In this case the set of reducible solutions modulo gauge
  equivalence can be identified with the torus $H^1(Y;\R)/2\pi
  H^1(Y;\Z)$.
\end{enumerate}
\end{lemma}

\begin{proof}
  Part (1) holds because if $A$ is any connection on $\det({\mathbb
    S})$ then the curvature $F_A$ is a closed 2-form representing the
  cohomology class $-2\pi i c_1(\frak{s})$.

  To prove part (2), note that any two connections with the same
  curvature differ by an imaginary-valued closed $1$-form.  Thus the
  set of reducible solutions modulo gauge equivalence is an affine
  space over the set of imaginary closed 1-forms modulo $\{g^{-1}dg
  \mid g:Y\to S^1\subset\C\}$.  To understand the latter subspace,
  note that any exact form $idf$ can be written as $g^{-1}dg$ where
  $g=e^{if}$. On the other hand the set of homotopy classes of maps
  $Y\to S^1$ can be identified with $H^1(Y;\Z)$, and the image of the
  homotopy class of $g:Y\to S^1$ under the map $H^1(Y;\Z)\to
  H^1(Y;\R)$ is the cohomology class of the $1$-form $(2\pi
  i)^{-1}g^{-1}dg$.  The claim follows.
\end{proof}

\section{Seiberg-Witten Floer homology}
\label{sec:SWF}

We now briefly review what we need to know about Seiberg-Witten Floer
homology, from \cite{km}.

\subsection{The Chern-Simons-Dirac functional}

We begin by realizing the solutions to the Seiberg-Witten equations as
the critical points of a functional.  The Seiberg-Witten Floer theory
will then be some kind of Morse homology for this functional.

Let $Y$ be a closed connected oriented $3$-manifold with a Riemannian
metric and a spin-c structure $\frak{s}=({\mathbb S},\rho)$.  Fix a real closed
$2$-form $\mu$ for use in defining the perturbed Seiberg-Witten
equations \eqref{eqn:SW}.  Also fix a reference connection $A_0$ on
$\det({\mathbb S})$, so that if $A$ is any other connection on
$\det({\mathbb S})$ then $A-A_0$ is an imaginary-valued $1$-form on
$Y$.

\begin{definition}
Define the (perturbed) {\em Chern-Simons-Dirac functional\/}
\[
{\mathcal
  F}:{\mathcal C}\to\R
\]
by
\begin{equation}
\label{eqn:CSD}
{\mathcal F}(A,\psi) \eqdef -\frac{1}{8}\int_Y(A-A_0) \wedge
(F_A+F_{A_0} - 2i\mu) + \frac{1}{2}\int_Y\langle D_A\psi,\psi\rangle d\op{vol}.
\end{equation}
\end{definition}

\begin{lemma}
$(A,\psi)$ is a critical point of $\mc{F}$ if and only if $(A,\psi)$
satisfies the perturbed Seiberg-Witten equations \eqref{eqn:SW}.
\end{lemma}

\begin{proof}
Let $(\dot{A},\dot{\psi})$ be a tangent vector to the configuration
space $\mc{C}$ at $(A,\psi)$.  This means that $\dot{A}$ is an
imaginary-valued $1$-form and $\dot{\psi}$ is a spinor.
We compute
\[
\begin{split}
  d\mc{F}_{(A,\psi)}(\dot{A},\dot{\psi}) = &
  -\frac{1}{8}\int_Y\dot{A}\wedge (F_A+F_{A_0}-2i\mu) -
  \frac{1}{8}\int_Y (A-A_0)\wedge d\dot{A}\\
& + \frac{1}{4}\int_Y\langle \rho(\dot{A})\psi,\psi\rangle d\op{vol} +
\frac{1}{2}\int_Y\langle D_A\dot{\psi},\psi\rangle d\op{vol}+ 
\frac{1}{2}\int_Y\langle D_A\psi,\dot{\psi}\rangle d\op{vol}.
\end{split}
\]
Applying Stokes' theorem to the second term, using the properties
\eqref{eqn:PauliMatrices} of Clifford multiplication to manipulate the
third term, and applying self-adjointness of the Dirac operator to the
fourth term, we obtain
\[
  d\mc{F}_{(A,\psi)}(\dot{A},\dot{\psi}) = 
-\frac{1}{4}\int_Y\dot{A}\wedge(F_A - i\mu
-*\langle\rho(\cdot)\psi,\psi\rangle) + \int_Y\op{Re}\langle
D_A\psi,\dot{\psi}\rangle d\op{vol}.
\]
This vanishes for all $\dot{A}$ and $\dot{\psi}$ if and only if
$(A,\psi)$ satisfy the perturbed Seiberg-Witten equations \eqref{eqn:SW}.
\end{proof}

We now consider the behavior of $\mc{F}$ under gauge transformations.
Recall that the set of homotopy classes of maps $Y\to S^1$ can be
identified with $H^1(Y;\Z)$, and denote the homotopy class of a map
$g$ by $[g]$.

\begin{lemma}
\label{lem:BGT}
If $g:Y\to S^1$ is a gauge transformation, then
\[
\mc{F}(g\cdot(A,\psi))-\mc{F}(A,\psi) = \pi\int_{[Y]}[g]\smile(2\pi
c_1(\frak{s})+[\mu]).
\]
\end{lemma}

\begin{proof}
We compute, using \eqref{eqn:DiracGauge}, that
\[
\begin{split}
\mc{F}(g\cdot(A,\psi))-\mc{F}(A,\psi) = &
-\frac{1}{8}\int_Y(-2g^{-1}dg)\wedge(F_A + F_{A_0} - 2i\mu)\\
& +
\frac{1}{2}\int_Y\langle gD_A\psi,g\psi\rangle d\op{vol} -
\frac{1}{2}\int_Y\langle D_A\psi,\psi\rangle d\op{vol}.
\end{split}
\]
The second line vanishes.  To process the first line, recall that
$g^{-1}dg$ represents the class $2\pi i[g]$, while
$F_A$ and $F_{A_0}$ both represent the class $-2\pi i c_1(\frak{s})$.
The lemma follows.
\end{proof}

In particular, $\mc{F}$ is gauge invariant if
and only if $ [\mu]=-2\pi c_1(\frak{s}) $ in $H^2(Y;\R)$.

\subsection{Seiberg-Witten Floer homology}

Roughly speaking, Seiberg-Witten Floer homology is the Morse homology
of the functional $\mc{F}$ on $\mc{C}/\mc{G}$, where the perturbation
2-form $\mu$ is taken to be exact\footnote{One can also define versions of
Seiberg-Witten Floer homology when $\mu$ is closed but not exact, but
these have different properties.}.  The detailed construction is
carried out in \cite{km}.  Some basic points to keep in mind are the
following:

(1) When $c_1(\frak{s})$ is not torsion, the functional $\mc{F}$ is
not gauge invariant, so it is not actually a real-valued functional on
$\mc{C}/\mc{G}$.  One can still define its Morse theory in this case,
but we will not explain the details of this because we will only be
concerned with the case where $c_1(\frak{s})$ is torsion below.

(2) The quotiented configuration space $\mc{C}/\mc{G}$ on which
$\mc{F}$ is defined is not a manifold in any natural sense, because
$\mc{G}$ does not act freely on the reducibles.  On the other hand,
$\mc{C}/\mc{G}$ is the quotient of a manifold by an $S^1$ action.
Namely, if one fixes a point $y_0\in Y$ and defines
$\mc{G}_0\eqdef\{g\in\mc{G}\mid g(y_0)=1\}$, then $\mc{G}_0$ acts
freely on $\mc{C}$, so $\mc{C}/\mc{G}_0$ is a manifold, and
$\mc{C}/\mc{G}=(\mc{C}/\mc{G}_0)/S^1$.  One now wants to define some
kind of $S^1$-equivariant Morse homology of $\mc{F}$ on
$\mc{C}/\mc{G}_0$.

The approach taken by Kronheimer-Mrowka is, roughly speaking, to blow
up the singularities of $\mc{C}/\mc{G}$, so as to obtain a
manifold-with-boundary $\widetilde{\mc{C}/\mc{G}}$, where the boundary
arises from the reducibles.  The gradient flow of $\mc{F}$ induces a
(partially defined) flow on $\widetilde{\mc{C}/\mc{G}}$, which is
tangent to the boundary.

There is now a finite-dimensional model for how to proceed.  Let $X$
be a finite dimensional compact manifold with boundary, and let
$(f,g)$ be a Morse-Smale pair on $X$ such that the gradient flow is
tangent to the boundary.  In this context, as explained in
\cite[\S2]{km}, there are three versions of Morse homology one can
define, which fit into a long exact sequence:
\[
\overline{H}_*^{\op{Morse}}(X,f,g) \longrightarrow
\check{H}_*^{\op{Morse}}(X,f,g) \longrightarrow
\widehat{H}_*^{\op{Morse}}(X,f,g) \longrightarrow
\overline{H}_{*-1}^{\op{Morse}}(X,f,g) \longrightarrow
\]
Here $\overline{H}_*$ is just the Morse homology of the boundary.  The
version $\check{H}_*$ is the homology of a chain complex which is
freely generated over $\Z$ by interior critical points and ``boundary
stable'' critical points on the boundary, and whose differential
counts certain configurations of flow lines between them.  The version
$\widehat{H}_*$ is similar but its generators include ``boundary
unstable'' critical points on the boundary instead.  The above exact
sequences turns out to agree with the relative homology exact sequence
\[
H_*(\partial X) \longrightarrow H_*(X) \longrightarrow H_*(X,\partial
X) \longrightarrow H_{*-1}(\partial X) \longrightarrow \cdots
\] 

Carrying out an analogue of this construction on the blown-up
configuration space now gives three versions of Seiberg-Witten Floer
homology, which fit into a long exact sequence:
\[
\overline{HM}_*(Y,\frak{s}) \longrightarrow \check{HM}_*(Y,\frak{s})
\longrightarrow \widehat{HM}_*(Y,\frak{s}) \longrightarrow
\overline{HM}_{*-1}(Y,\frak{s}) \longrightarrow \cdots
\]
The version $\overline{HM}$ comes entirely from the reducibles.  Since
the reducibles are described explicitly by Lemma~\ref{lem:reducibles},
it is possible (although not trivial) to compute $\overline{HM}$ in
terms of classical algebraic topology, specifically the triple cup
product on $H^*(Y)$, see \cite[\S35.1]{km}.  As such, $\overline{HM}_*$
may seem less interesting than the other two versions $\check{HM}_*$
and $\widehat{HM}_*$.  However the computation of $\overline{HM}_*$ is
used in conjunction with the above exact sequence to prove
Kronheimer-Mrowka's nontriviality result in
Theorem~\ref{thm:nontrivial}, which plays an essential role in the
proof of the Weinstein conjecture.

(3) The three versions of Seiberg-Witten Floer homology above all have
a relative $\Z/d(c_1(\frak{s}))$-grading (given by the expected
dimension of the moduli space of flow lines, which equals a certain
spectral flow).  So if $c_1(\frak{s})$ is torsion then the
Seiberg-Witten Floer homologies are relatively $\Z$-graded.  As
mentioned previously, there is in fact an absolute grading by homotopy
classes of oriented $2$-plane fields, which is compatible with the map
$\mc{P}(Y)\to\Spinc(Y)$ discussed in Remark~\ref{rem:HPF}.  However
for our purposes it is enough to just regard the grading as taking
values in some affine space over $\Z/d(c_1(\frak{s}))$.

(4) Heuristically the differentials in the Seiberg-Witten Floer chain
complexes count gradient flow lines of $\mc{F}$, but in fact some
abstract perturbations of the equations are required in order to
obtain the transversality needed to count solutions.  The
perturbations are explained in detail in \cite{km} and will be
suppressed in the exposition here.

\section{Outline of Taubes's proof}
\label{sec:OTP}

We now have the background in place to describe Taubes's proof of the
Weinstein conjecture.  This section gives an outline, and the next
section explains some more details.  Below we mostly follow Taubes's
paper \cite{tw} and MSRI lectures \cite{msri}.

\subsection{Geometric setup}

Let $Y$ be a closed oriented connected $3$-manifold with a contact
form $\lambda$.  Fix a Riemannian metric on $Y$ such that
$|\lambda|=1$ and $d\lambda = 2*\lambda$.  We can do this because
$\lambda\wedge d\lambda>0$.

To be consistent with Taubes, denote\footnote{This notation is carried
  over from the SW=Gr story, where $K$ denotes the canonical bundle of
  a symplectic $4$-manifold.  For a contact 3-manifold, $K$ is the
  canonical bundle of the symplectization.}  the oriented $2$-plane
field $\xi=\Ker(\lambda)$, regarded as a Hermitian line bundle, by
$K^{-1}$.  Recall that $\xi$ determines a distinguished spin-c
structure $\frak{s}_\xi$, in which
\begin{equation}
\label{eqn:sxi}
{\mathbb S}=\underline{\C}\oplus K^{-1}.
\end{equation}
Any spin-c structure $\frak{s}$ is obtained from $\frak{s}_\xi$ by
tensoring with a Hermitian line bundle $E$, so that
\begin{equation}
\label{eqn:SEKE}
{\mathbb S} = E\oplus K^{-1}E.
\end{equation}
In this decomposition, $E$ is the $+i$ eigenspace of Clifford
multiplication by $\lambda$, while $K^{-1}E$ is the $-i$ eigenspace.
The significance of $E$ is that Taubes's Theorem~\ref{thm:echswf}
ultimately shows that the Seiberg-Witten Floer cohomology
$\widehat{HM}^{-*}(Y,\frak{s})$ is isomorphic to
$ECH_*(Y,\lambda,\Gamma)$ where $\Gamma$ is Poincare dual to $c_1(E)$.

For any spin-c structure as in \eqref{eqn:SEKE}, connections on
$\det({\mathbb S})$ can be written as $A_0+2A$ where $A_0$ is a
reference connection on $K^{-1}$ while $A$ is a connection on $E$.  In
fact Taubes picks out a distinguished connection $A_0$ on $K^{-1}$ as
follows.  For the distinguished spin-c structure $\frak{s}_\xi$, let
$\psi_0$ denote the spinor given by $(1,0)$ in the decomposition
\eqref{eqn:sxi}.

\begin{lemma}
\label{lem:A0}
There is a unique Hermitian connection $A_0$ on $K^{-1}$ such that
\[
D_{A_0}\psi_0=0.
\]
\end{lemma}

\begin{proof}
Uniqueness follows from the formula
\[
D_{A_0+2a}\psi_0 =
D_{A_0}\psi_0 +
\rho(a)\psi_0
\]
and the equations \eqref{eqn:PauliMatrices}.

To prove existence, let $A$ be any Hermitian connection on $K^{-1}$,
let $\nabla_A$ denote the associated spin-c connection, and write
$\nabla_A\psi_0=(\alpha,\beta)$ where $\alpha$ and $\beta$ are
$1$-forms on $Y$ with values in $\C$ and $K^{-1}$ respectively.  Since
$\nabla_A$ is Hermitian, we have $\op{Re}(\alpha)=0$.  It follows that
there is a unique $A_0$ such that
\begin{equation}
\label{eqn:A0}
\nabla_{A_0}\psi_0=(-i\lambda,\beta_0)
\end{equation}
for some $\beta_0$.  We now show that $D_{A_0}\psi_0=0$.

We need the following Leibniz-type formula for the Dirac operator: If
$A$ is any spin-c connection, $\omega$ is any differential form, and
$\psi$ is any spinor, then
\begin{equation}
\label{eqn:leibniz}
D_A(\rho(\omega)\psi) = \rho((d+d^*)\omega)\psi + \rho((1\tensor
\rho(\omega))\nabla_A\psi).
\end{equation}
Here Clifford multiplication is extended to an action of $\Lambda^*TX$
on ${\mathbb S}$ by the rule
\[
\rho(a\wedge b) =
\frac{1}{2}\left(\rho(a)\rho(b)+(-1)^{\op{deg}(a)\op{deg}(b)}\rho(b)\rho(a)\right).
\]
Taking $A=A_0$, $\psi=\psi_0$, and $\omega=\lambda$ in
\eqref{eqn:leibniz}, we obtain
\[
iD_{A_0}\psi_0
= (-2i,0) +
\rho(\lambda,-i\beta_0).
\]
On the other hand, applying $\rho$ to equation \eqref{eqn:A0} and
multiplying by $-i$ gives
\[
-iD_{A_0}\psi_0 = \rho(-\lambda,-i\beta_0).
\]
Subtracting the above equation from the previous one and using the
fact that $\rho(\lambda,0)=(i,0)$ gives $2iD_{A_0}\psi_0=0$.
\end{proof}

Henceforth, think of spin connections as being determined by the
distinguished connection $A_0$ on $K^{-1}$ together with a connection
$A$ on $E$.

\subsection{Taubes's perturbation}

The idea of Taubes's proof of the Weinstein conjecture is to deform
the Seiberg-Witten equations by a sequence of increasingly large
perturbations, and to use a sequence of solutions to the perturbed
equation, provided by the known nontriviality of Seiberg-Witten Floer
homology in Theorem~\ref{thm:nontrivial}, to yield a nonempty
collection of Reeb orbits.

The basic perturbation of the Seiberg-Witten equations considered by
Taubes is
\begin{equation}
\label{eqn:PSW}
\begin{split}
  *F_A &= r(\langle\rho(\cdot)\psi,\psi\rangle - i\lambda) +
  i\overline{\omega},\\
  D_A\psi &= 0.
\end{split}
\end{equation}
These should be regarded as equations for the pair $(A,\psi)$ that are
parametrized by $r$.  Here $r\ge 1$ (the deformation involves taking
$r\to \infty$), and remember that now $A$ denotes a connection on $E$,
while $\psi$ is a section of $E\oplus K^{-1}E$.  Also
$\overline{\omega}$ denotes the harmonic $1$-form whose Hodge star
represents the image of $\pi c_1(K^{-1})$ in $H^2(Y;\R)$.  Taubes's
equations \eqref{eqn:PSW} are equivalent to a case of the perturbed
Seiberg-Witten equations \eqref{eqn:SW}, via rescaling the spinor by a
factor of $\sqrt{r}$ and taking the perturbation to be the exact
2-form
\begin{equation}
\label{eqn:mutaubes}
\mu = -rd\lambda - iF_{A_0} + 2*\overline{\omega}.
\end{equation}

Some parts of the argument involve further perturbations of the
equations \eqref{eqn:PSW} in order to obtain necessary transversality.
For simplicity, we will suppress these in the exposition below.

\subsection{The trivial solution}

A first observation regarding the equations \eqref{eqn:PSW} is that
for the spin-c structure $\frak{s}_\xi$, if the $\overline{\omega}$
term were omitted, then the pair $(A_0,\psi_0)$ from
Lemma~\ref{lem:A0} would be a solution for any $r$.  In general we do
not want to remove the $\overline{\omega}$ term from \eqref{eqn:PSW}
because the perturbation 2-form \eqref{eqn:mutaubes} needs to be exact
in order to obtain the correct Seiberg-Witten Floer homology.  However
since the $\overline{\omega}$ term is much smaller than the term with
a factor of $r$ in \eqref{eqn:PSW} when $r$ is large, a perturbation
argument can be used to prove the following:

\begin{lemma}
\label{lem:trivsol}
\cite[Prop.\ 2.8]{tw2} For any $\delta>0$, if $r$ is sufficiently large then:
\begin{itemize}
\item
There exists a unique (up to gauge equivalence) solution
$(A_{triv},\psi_{triv})$ to the equations \eqref{eqn:PSW} for the spin-c
structure $\frak{s}_\xi$ such that $1-|\psi_{triv}|\le \delta$ on all of $Y$.
(In fact $1-|\psi_{triv}|= O(r^{-1/2})$.)
\item
The grading of $(A_{triv},\psi_{triv})$ in the
Seiberg-Witten Floer chain complex is independent of $r$.
\end{itemize}
\end{lemma}

We will call $(A_{triv},\psi_{triv})$ the ``trivial solution''.
Ultimately, in the isomorphism with embedded contact homology, the
trivial solution corresponds to the empty set of Reeb orbits.

\subsection{Convergence to Reeb orbits}

Taubes now proves the following theorem\footnote{Taubes actually
  proves more general versions of this theorem, and many of the other
  results of his that we are quoting here.} which finds Reeb orbits
from a sequence of solutions to \eqref{eqn:PSW}.

\begin{theorem}(\cite[Thm.\ 2.1]{tw})
\label{thm:convergence}
Fix the line bundle $E$ and let $(r_n,\psi_n,A_n)$ be a sequence of
solutions to the equations \eqref{eqn:PSW} with $r_n\to \infty$.
Suppose that:
\begin{enumerate}
\item
There is a constant $\delta>0$ with $\sup_Y(1-|\psi_n|) > \delta$.
\item
There is a constant $C<\infty$ with $i\int_Y\lambda\wedge F_{A_n} <
C$.
\end{enumerate}
Then there exists a nonempty orbit set $a$ with
$[a]=\op{PD}(c_1(E))$.
\end{theorem}

We will explain the proof of this theorem in some detail in
\S\ref{sec:MD}.  For now we remark that this is a three-dimensional
analogue of an earlier theorem of Taubes for symplectic four-manifolds
in \cite{taubes:swtogr}, part of the SW=Gr story, which obtains
holomorphic curves from sequences of solutions to the Seiberg-Witten
equations perturbed using the symplectic form.  However since the
dimension is one less here, subtle measure-theoretic arguments that
were used in the four-dimensional case can be avoided, and the proof
is considerably shorter.  The basic idea is to write
$\psi_n=(\alpha_n,\beta_n)$, where $\alpha_n$ is a section of $E$ and
$\beta_n$ is a section of $K^{-1}E$, and show that one can pass to a
subsequence such that $\alpha_n^{-1}(0)$ converges as a current to a
nonempty orbit set.  This will then, of course, represent the Poincare
dual of $c_1(E)$.  In fact, when $n$ is large, $|\beta_n|$ will be
close to zero everywhere, while $|\alpha_n|$ will be close to $1$
except near its zero set.  Also the curvature $F_{A_n}$ will be
concentrated near the zero set of $\alpha_n$, and its direction will
be approximately dual to the normal plane to $\alpha_n^{-1}(0)$.

Assumption (1) is needed to avoid solutions with $\alpha_n$
nonvanishing, which can exist when $c_1(E)=0$, as we have seen in
Lemma~\ref{lem:trivsol}.

The idea of assumption (2) is that when $n$ is large,
$i\int_Y\lambda\wedge F_{A_n}$ is approximately $2\pi$ times the
symplectic action \eqref{eqn:action} of the orbit set to which
$\alpha_n^{-1}(0)$ is converging.  A uniform upper bound on this
integral is needed in order to obtain an orbit set of finite length.
(In a more general situation without condition (2) one can still
obtain some weaker conclusions, see \S\ref{sec:ue}.)

Note also that assumption (2) guarantees that $(A_n,\psi_n)$ is
irreducible when $n$ is sufficiently large, because it follows from
the equations \eqref{eqn:PSW} that if $(A,0)$ is a reducible solution
to \eqref{eqn:PSW} then $i\int_Y\lambda\wedge F_A$ is a linear,
increasing function of $r$.

\subsection{Avoiding the empty set}

Now fix $E$ such that $c_1(K^{-1}) + 2c_1(E)$ is torsion in
$H^2(Y;\Z)$.  Let $\frak{s}$ denote the corresponding spin-c
structure.  Kronheimer-Mrowka's Theorem~\ref{thm:nontrivial}
guarantees the existence of solutions to the perturbed equations
\eqref{eqn:PSW} for all $r\ge 1$.  To complete the proof of
Theorem~\ref{thm:tw}, which implies the Weinstein conjecture, we need
to find a sequence of such solutions with $r\to\infty$ such that
conditions (1) and (2) in Theorem~\ref{thm:convergence} are satisfied.

One can achieve condition (1) using the following lemma, proved in
\S\ref{sec:MD}:

\begin{lemma}
\label{lem:aes}
If $c_1(E)\neq 0$, then there is a constant $c>0$ such that if $r$ is
sufficiently large, and if $(A,\psi)$ is a solution to the equations
\eqref{eqn:PSW}, then there exist points in $Y$ where $1-|\psi| \ge
1-c/\sqrt{r}$.
\end{lemma}

This means that if $c_1(E)\neq 0$ then any sequence of solutions will
automatically satisfy condition (1) (after discarding some initial
terms).  On the other hand, by Lemma~\ref{lem:trivsol}, if $c_1(E)=0$
then a sequence of solutions will likewise satisfy condition (1) as
long as we avoid the gauge equivalence class of the trivial solution
$(A_{triv},\psi_{triv})$.  We can easily do this since we know from
Theorem~\ref{thm:nontrivial} that the Seiberg-Witten Floer homology is
nonzero in infinitely many gradings.

\subsection{Three functionals}

The hardest part of the proof of the Weinstein conjecture is to
achieve condition (2) in Theorem~\ref{thm:convergence}.  This is a new
problem which does not arise in the four-dimensional SW=Gr story.  (On
a symplectic four-manifold $(X,\omega)$, to obtain convergence to a
holomorphic curve, one needs an analogue of condition (2) in which
$\lambda$ replaced by the symplectic form $\omega$; but there the
quantity that needs to be bounded is constant because the symplectic
form is closed.)

The first step is to write the Chern-Simons-Dirac functional $\mc{F}$
in \eqref{eqn:CSD}, of which Seiberg-Witten Floer homology is the
Morse homology, as the sum of two other functionals.  To do so, fix a
reference connection $A_1$ on $E$.

\begin{definition}
If $A$ is a connection on $E$, define the {\em Chern-Simons
  functional\/} $cs(A)$ by
\begin{equation}
\label{eqn:cs}
cs(A) \eqdef -\int_Y(A-A_1) \wedge (F_A + F_{A_1} -
2i*\overline{\omega}).
\end{equation}
Note that this is gauge invariant thanks to our assumption that
$2c_1(E) + c_1(K^{-1})$ is torsion.  Also, define the {\em energy\/}
\begin{equation}
\label{eqn:energy}
\mc{E}(A) \eqdef i\int_Y\lambda\wedge F_A.
\end{equation}
This is the quantity that we want to control.
\end{definition}

Observe now that for a given $r$, if in the definition of $\mc{F}$ we
take our reference connection on $\det({\mathbb S})$ to be $A_0+2A_1$,
then we have
\[
\mc{F}(A,\psi) = \frac{1}{2}\left(cs(A) - r\mc{E}(A)\right)
+\frac{r}{2}\int_Y\langle D_A\psi,\psi\rangle d\op{vol},
\]
up to the addition of an $r$-dependent constant.  Since adding a
constant to $\mc{F}$ does not affect its Morse homology, we will
ignore this constant and take the above equation to be the new
definition of $\mc{F}$.  In particular, if $(A,\psi)$ is a solution to
the perturbed Seiberg-Witten equations \eqref{eqn:PSW}, then the three
functionals in play are related by
\begin{equation}
\label{eqn:FCE}
\mc{F}(A,\psi) = \frac{1}{2}\left(cs(A) - r\mc{E}(A)\right).
\end{equation}

\subsection{A piecewise smooth family of solutions}

The next step is:

\begin{lemma}
\label{lem:PSF}
(Up to the perturbations we are suppressing)
  one can choose for each $r$ sufficiently large a solution
  $(A(r),\psi(r))$ to the equations \eqref{eqn:PSW} such that:
\begin{itemize}
\item
$(A(r),\psi(r))$ is a piecewise smooth function of $r$.
\item
$\mc{F}(A(r),\psi(r))$ is a continuous function of $r$.
\item For all $r$ at which $(A(r),\psi(r))$ is smooth as a function of
  $r$, $(A(r),\psi(r))$ is nondegenerate\footnote{A critical point
    $(A,\psi)$ of $\mc{F}$ is ``nondegenerate'' if the Hessian of the
    functional $\mc{F}$ at $(A,\psi)$ has kernel zero, so that the
    grading of $(A,\psi)$ in the Seiberg-Witten Floer chain complex is
    well-defined.} and its grading in the Seiberg-Witten Floer chain
  complex is independent of $r$.
\item 
  $(A(r),\psi(r))$ is not gauge equivalent to the trivial solution
  $(A_{triv},\psi_{triv})$ described in Lemma~\ref{lem:trivsol}.
\end{itemize}
\end{lemma}

The idea of the proof of Lemma~\ref{lem:PSF} is as follows.  First,
one shows that for any given grading, if $r$ is sufficiently large
then all generators of the chain complex defining
$\widehat{HM}_*(Y,\frak{s})$ with that grading are irreducible.  This
is proved using a spectral flow estimate related to
Proposition~\ref{prop:SF} below.  Thus, for any given range of
gradings, if $r$ is sufficiently large then the differential in the
Seiberg-Witten Floer chain complex just counts (perturbed) gradient
flow lines of $\mc{F}$, without the subtleties arising from
reducibles.

Now by Theorem~\ref{thm:nontrivial}, there is a nonzero class $\sigma$
in the Seiberg-Witten Floer homology $\widehat{HM}_*(Y,\frak{s})$, and
when $c_1(E)=0$ we can assume that the grading of $\sigma$ is not the
same as that of $(A_{triv},\psi_{triv})$.  Fix such a class $\sigma$.

One now sets up the perturbations (that we have suppressed in the
exposition) so that the Seiberg-Witten Floer chain complex is defined
for generic $r$.  For such $r$, we define $h(r)\in\R$ to be ``the
minimum height of $\mc{F}$ needed to represent the class $\sigma$''.
More precisely, a chain representing the class $\sigma$ can be
expressed as $\sum_in_ic_i$, where $n_i$ is a nonzero integer and
$c_i$ is a critical point of $\mc{F}$ for each $i$ in some finite set.
Define $h(r)$ to be the minimum, over all chains $\sum_in_ic_i$
representing $\sigma$, of $\max_i\mc{F}(c_i)$.

The idea is then to define $(A(r),\psi(r))$ to be a maximal $\mc{F}$
critical point in a representative of $\sigma$ realizing the minimum
$h(r)$.  One can choose this $(A(r),\psi(r))$ to vary piecewise
smoothly with $r$, jumping when the ``Morse complex'' undergoes a
bifurcation involving the critical point $(A(r),\psi(r))$.

To complete the proof of Lemma~\ref{lem:PSF}, one needs to show that
the function $h(r)$ defined above extends to a continuous function of
all sufficiently large $r$.  Taubes does so by explicitly studying the
bifurcations that can happen in a generic one-parameter family of
``Morse complexes''.  One can presumably also prove this by estimating
that the continuation maps that relate the Seiberg-Witten Floer
homologies for nearby values of $r$ do not increase the functional
$\mc{F}$ too much.  This method has been used to prove analogous
continuity results in symplectic Floer homology, see eg
\cite[\S2.4]{schwarz-action}.  Here is a model for this argument in
finite dimensional Morse theory:

\begin{proposition}
  Let $X$ be a finite dimensional closed manifold and let
  $\{(f_r,g_r)\}$ be a generic smooth family of functions $f_r:X\to
  \R$ and metrics $g_r$ on $X$ parametrized by $r\in[0,1]$.  Fix
  $0\neq \sigma\in H_*(X)$.  For generic $r$, such that the pair
  $(f_r,g_r)$ is Morse-Smale, define $h(r)$ to be the minimum height
  of a representative of the class $\sigma$.  Then $h$ extends to a
  continuous function on all of $[0,1]$.
\end{proposition}

\begin{proof}
  Note that the pair $(f_r,g_r)$ is Morse-Smale for all but finitely
  many $r$.  Recall from \S\ref{sec:continuation} that the Morse
  homologies of the Morse-Smale pairs $(f_r,g_r)$ are canonically
  isomorphic to each other, via continuation maps, and also to
  $H_*(X)$, so that $\sigma$ defines a class in the Morse homology for
  each Morse-Smale pair $(f_r,g_r)$ which is preserved by these
  continuation maps.

  Now suppose that $r<r'$ and the pairs $(f_r,g_r)$ and
  $(f_{r'},g_{r'})$ are both Morse-Smale. The continuation isomorphism
  from $r'$ to $r$ is induced by a chain map
\[
\Phi:C_*^{\op{Morse}}(X,f_{r'},g_{r'}) \longrightarrow
C_*^{\op{Morse}}(f_r,g_r)
\]
which counts maps $f:\R\to X$ satisfying the equation
\begin{equation}
\label{eqn:continuation}
\frac{d\gamma(s)}{ds} = \nabla f_{\phi(s)}(\gamma(s))
\end{equation}
where $\phi:\R\to[r,r']$ is a monotone smooth function satisfying $\phi(s)=r$
for $s\le 0$ and $\phi(s)=r'$ for $s\ge 1$.  Now if $\gamma$ is a
solution to \eqref{eqn:continuation} then by the chain rule we have
\begin{equation}
\label{eqn:chainRule}
\frac{d}{ds}f_{\phi(s)}(\gamma(s)) = |\nabla f_{\phi(s)}(\gamma(s))|^2
+ \frac{d\phi(s)}{ds}\frac{\partial f_r(x)}{\partial
  r}\big|_{r=\phi(s),x=\gamma(s)}.
\end{equation}
Since $[0,1]\times X$ is compact there is a constant $c$ such that
$|\partial f_r(x)/\partial r|<c$ for all $(r,x)\in[0,1]\times X$.  It
then follows from \eqref{eqn:chainRule} that if $p'$ and $p$ are
critical points of $f_{r'}$ and $f_r$ respectively with $\langle \Phi
p',p\rangle \neq 0$, then
\[
f_{r'}(p') \ge f_{r}(p) - c(r'-r).
\]
That is, the continuation map from $r'$ to $r$ increases the height by
at most $c(r'-r)$, so
\[
h(r) \le h(r') + c(r'-r).
\]
The same holds for the continuation map in the other direction from
$r$ to $r'$, so we conclude that
\[
|h(r) - h(r')| \le c|r-r'|.
\]
The proposition follows.
\end{proof}

Accepting Lemma~\ref{lem:PSF}, we now have:

\begin{lemma}
\label{lem:dFdr}
Let $\{(A(r),\psi(r))\}$ be a piecewise smooth family from
Lemma~\ref{lem:PSF}.  Then
\[
\frac{d}{dr}\mc{F}(A(r),\psi(r)) = -\frac{1}{2}\mc{E}(A(r)).
\]
\end{lemma}

\begin{proof}
  This follows from a general principle: If $X$ is a smooth manifold
  (finite or infinite dimensional), if $f:\R\times X\to\R$ is a smooth
  function, and if $\{x(t)\}$ is a smooth family of critical points of
  $f_t\eqdef f(t,\cdot)$ on $X$ defined for $t$ in some interval, then
\[
\frac{d}{dt}f_t(x(t)) = \frac{\partial f}{\partial t}(t,x(t)).
\]
To prove this one uses the chain rule as in \eqref{eqn:chainRule} to
compute
\[
\frac{d}{dt}f(t,x(t)) = \frac{\partial f}{\partial t}(t,x(t)) +
df_t\left(\frac{dx(t)}{dt}\right),
\]
and notes that the second term on the right vanishes because $x(t)$ is
a critical point of $f_t$.
\end{proof}

\subsection{The energy dichotomy}

Let $\{(A(r),\psi(r)\}$ be a piecewise smooth family given by
Lemma~\ref{lem:PSF}.  To prove Theorem~\ref{thm:tw} and thereby the
Weinstein conjecture, by Theorem~\ref{thm:convergence} we just need to
show that there is a sequence $r_n\to \infty$ such that the energy
$\mc{E}(A(r_n))$ is bounded.  The next step in Taubes's argument is to
show that if this is not the case, then there is a sequence $r_n\to
\infty$ such that the energy $\mc{E}(A(r_n))$ grows at least linearly,
and the Chern-Simons functional grows quadratically.  (The last step
will be to show that this quadratic growth of the Chern-Simons
functional leads to a contradiction.)

\begin{lemma}
\label{lem:ED}
Let $\{(A(r),\psi(r)\}$ be a piecewise smooth family given by
Lemma~\ref{lem:PSF}.  Then at least one of the following two
alternatives holds:
\begin{enumerate}
\item There is a sequence $r_n\to\infty$ and a constant $C$
  such that $\mc{E}(A(r_n))<C$ for all $n$.
\item There is a sequence $r_n\to\infty$ and a constant $c>0$ such
  that $\mc{E}(A(r_n))\ge cr_n$ and $cs(A(r_n))\ge cr_n^2$ for all
  $n$.
\end{enumerate}
\end{lemma}

The proof of Lemma~\ref{lem:ED} uses the following a priori estimate,
which is proved in \S\ref{sec:MD}.

\begin{lemma}
\label{lem:cse}
There is a constant $c$ such that if $(r,A,\psi)$ is a solution to
the equations \eqref{eqn:PSW} with $\mc{E}(A)>1$ then
\begin{equation}
\label{eqn:cse}
|cs(A)| \le c r^{2/3}\mc{E}(A)^{4/3}.
\end{equation}
\end{lemma}

Granted this, we can now give:

\begin{proof}[Proof of Lemma~\ref{lem:ED}.]
  Introduce the shorthand $cs(r)\eqdef
  cs(A(r))$, $\mc{E}(r)\eqdef\mc{E}(A(r))$ and
  $\mc{F}(r)\eqdef\mc{F}(A(r),\psi(r))$.  We can assume without loss
  of generality that $\mc{E}(r)> 1$ for all $r$ sufficiently large
  (since otherwise case (1) holds).  Now fix
  $\varepsilon_0\in(0,1/5)$.  We consider two cases.

{\em Case A:\/} There is a sequence $r_n\to\infty$ with
\[
cs(r_n) \ge \varepsilon_0 r_n\mc{E}(r_n)
\]
for all $n$.  It follows in this case from the inequality
\eqref{eqn:cse} that alternative (2) holds.

{\em Case B:\/} For all $r$ sufficiently large,
\begin{equation}
\label{eqn:CaseB}
cs(r) < \varepsilon_0 r\mc{E}(r).
\end{equation}
In this case we will show that alternative (1) holds.

To do so, define
\[
v(r) \eqdef \mc{E}(r) - \frac{cs(r)}{r} = -\frac{2\mc{F}}{r}.
\]
It then follows from Lemma~\ref{lem:dFdr} that
\begin{equation}
\label{eqn:dvdr}
\frac{dv}{dr} = \frac{cs}{r^2}.
\end{equation}
On the other hand, the hypothesis \eqref{eqn:CaseB} is equivalent to
\begin{equation}
\label{eqn:CaseB'}
\mc{E} < (1-\varepsilon_0)^{-1}v.
\end{equation}
It now follows from \eqref{eqn:dvdr}, \eqref{eqn:CaseB} and
\eqref{eqn:CaseB'} that
\[
\frac{dv}{dr} < \frac{\varepsilon v}{r},
\]
where $\varepsilon \eqdef (1-\varepsilon_0)^{-1}\varepsilon_0<1/4$.
Therefore
\begin{equation}
\label{eqn:vGrowth}
v< c_1r^\varepsilon
\end{equation}
for some constant $c_1$.  On the other hand, by \eqref{eqn:cse},
\eqref{eqn:CaseB'} and \eqref{eqn:vGrowth}, we have
\[
cs < c_2 r^{2/3+(4/3)\varepsilon}
\]
for some constant $c_2$.  Putting this back into \eqref{eqn:dvdr}, we
get
\[
\frac{dv}{dr} < c_2r^{(4/3)(\varepsilon-1)}.
\]
Since $\varepsilon<1/4$, the exponent in the above inequality is less
than $-1$.  Consequently the above inequality can be integrated to
show that $v$ is bounded from above.  Then $\mc{E}$ is also bounded
from above by \eqref{eqn:CaseB'}.
\end{proof}

\subsection{Controlling the Chern-Simons functional}

The last step in Taubes's proof of the Weinstein conjecture is the
following proposition relating the Seiberg-Witten Floer grading to the
Chern-Simons functional.  To state it, if $(A,\psi)$ is a solution to
the perturbed Seiberg-Witten equations \eqref{eqn:PSW}, let
$\op{deg}(A,\psi)$ denote its grading in the Seiberg-Witten Floer
chain complex, and recall that $(A_{triv},\psi_{triv})$ denotes the
distinguished Seiberg-Witten Floer generator given by
Lemma~\ref{lem:trivsol}.

\begin{proposition}
\label{prop:SF}
\cite[Prop. 5.1]{tw} There exists $\kappa>0$ such that for all $r$
sufficiently large, if $(A,\psi)$ is a nondegenerate solution to
\eqref{eqn:PSW}, then
\[
\left|\deg(A,\psi) - \deg(A_{triv},\psi_{triv}) + \frac{1}{4\pi^2}cs(A)\right|
< \kappa r^{31/16}.
\]
\end{proposition}

This is proved using a new estimate on the spectral flow of
one-parameter families of Dirac operators (the latter determines the
relative grading in Seiberg-Witten Floer homology).  While this is a
crucial new element of Taubes's proof of the Weinstein conjecture, is
too much for us to explain here, so we refer the reader to
\cite[\S5]{tw} for the spectral flow estimate that proves
Proposition~\ref{prop:SF}, and to \cite{tsf} for a higher-dimensional
generalization.

\subsection{Conclusion}

To prove the Weinstein Conjecture, more specifically
Theorem~\ref{thm:tw}, let $\Gamma\in H_1(Y)$ such that
$c_1(\xi)+2\op{PD}(\Gamma)$ is torsion, and let $E$ be the line bundle
with $c_1(E)=\op{PD}(\Gamma)$.  Let $\{(A(r),\psi(r)\}$ be a piecewise
smooth family of the perturbed Seiberg-Witten equations
\eqref{eqn:PSW} given by Lemma~\ref{lem:PSF}.  Alternative (2) in
Lemma~\ref{lem:ED} is impossible by Proposition~\ref{prop:SF}, because
$\deg(A,\psi)$ and $\deg(A_{triv},\psi_{triv})$ are independent of
$r$.  So alternative (1) in Lemma~\ref{lem:ED} holds.  Then we have a
sequence $r_n\to\infty$ such that condition (2) in
Theorem~\ref{thm:convergence} holds.  Condition (1) in
Theorem~\ref{thm:convergence} also holds by Lemmas~\ref{lem:trivsol}
and \ref{lem:aes}.  Thus Theorem~\ref{thm:convergence} applies to
produce the desired nonempty orbit set.

\section{More details of Taubes's proof}
\label{sec:MD}

We now fill in some of the (more basic) details of Taubes's proof that
were omitted from \S\ref{sec:OTP}.  In particular \S\ref{sec:ERO}
sketches the proof of Theorem~\ref{thm:convergence} on convergence of
Seiberg-Witten solutions to Reeb orbits.  

\subsection{Prelude: the vortex equations}
\label{sec:vortex}

Before studying the three-dimensional Seiberg-Witten equations in more
detail, it is useful to recall a two-dimensional version of the
equations: the {\em vortex equations\/} on $\C$.  These are equations
for a pair $(A,\alpha)$ where $A$ is an imaginary-valued $1$-form on
$\C$, and $\alpha$ is a complex-valued function on $\C$.  One can also
think of $A$ as a Hermitian connection on the trivial line bundle over
$\C$, and $\alpha$ as a section of this bundle.  The equations are now
\begin{equation}
\label{eqn:vortex}
\begin{split}
\dbar_A\alpha &= 0,\\
*F_A&=-i(1-|\alpha|^2).
\end{split}
\end{equation}
Here $\dbar_A\alpha=(\dbar + A^{0,1})\alpha$, and $F_A=dA$.  For
example there is a ``trivial solution'' $(A\equiv 0,\alpha\equiv 1)$.
More generally, solutions to the vortex equations arise as
$\R$-invariant solutions to the Seiberg-Witten equations on
$\R\times\C$, for a suitable perturbation. The following are some
basic properties of solutions to the vortex equations, see \cite{jt}
and \cite[\S2b]{taubes:grtosw}:
\begin{itemize}
\item
If $|\alpha|=1$ at any point, then $|\alpha|\equiv 1$.
\item
The zeroes of $\alpha$ are isolated and have positive multiplicity.
\item
If one further assumes the ``finite energy'' condition
\[
\int_\C(1-|\alpha|^2)<\infty,
\]
then $\int_\C(1-|\alpha|^2)=2\pi k$ where $k$ is a nonnegative integer,
and $\alpha$ has exactly $k$ zeroes counted with multiplicity.
\end{itemize}

The first and last of the above facts imply the following basic
observation which will be needed later:

\begin{lemma}
\label{lem:vortex}
If $(A,\alpha)$ is a finite energy solution to the vortex equations,
and if $|\alpha|\not\equiv 1$, then $\alpha$ has a zero.
\end{lemma}

For motivational purposes we now recall some additional facts about
solutions to the vortex equations:

\begin{itemize}
\item There is a constant $c>0$ such that if $d:\C\to[0,\infty)$
  denotes the distance to the set where $|\alpha|^2<1/2$, then where
  $d>c^{-1}$ one has
\begin{equation}
\label{eqn:ed}1-|\alpha|^2\le e^{-cd}.
\end{equation}
\item
There is a bijection from the set of solutions with
$\int_\C(1-|\alpha|^2)=2\pi k$, modulo gauge equivalence, to the
$k^{th}$ symmetric product of $\C$, sending $(A,\alpha)$ to
$\alpha^{-1}(0)$.
Here the gauge group $\mc{G}=\op{Maps}(\C,S^1)$ acts on the set
of solutions by $u\cdot(A,\alpha) \eqdef (A-u^{-1}du,u\alpha)$.
\end{itemize}

This last fact can be regarded as a two-dimensional version of the
``SW=Gr'' Theorem~\ref{thm:swgr}.  It also generalizes to vortices
over a closed surface, see \cite{gp}.  For some more recent work
relating this story to Seiberg-Witten theory see eg \cite{lt,perutz,usher}.

It is also useful to consider a variant of the vortex equations
\eqref{eqn:vortex}, let us call these the ``$r$-vortex equations'', in
which the second equation is replaced by
\[
*F_A=-ir(1-|\alpha|^2)
\]
where $r$ is a large positive constant.  A solution to the $r$-vortex
equations is equivalent to a solution to the original vortex equations
\eqref{eqn:vortex} under the rescaling $\C\to\C$ sending $z\mapsto
r^{1/2}z$.  So for a solution to the $r$-vortex equations, the
exponential decay \eqref{eqn:ed} is replaced by the stronger decay
\[
1-|\alpha|^2 \le e^{-cr^{1/2}d}.
\]
The upshot is that for $r$ large, given a solution $(A,\alpha)$ to the
$r$-vortex equations, the curvature $F_A$ is concentrated near the
zero set of $\alpha$, while away from the zero set $A$ is close to
flat and $|\alpha|$ is close to $1$.

The picture we are now aiming for in three dimensions is that if
$(A,\psi=(\alpha,\beta))$ is a solution to the three-dimensional
perturbed Seiberg-Witten equations \eqref{eqn:PSW} where $r$ is large
and the energy \eqref{eqn:energy} is bounded, then $F_A$ is
concentrated near the zero set of $\alpha$, the latter is approximated
by a union of Reeb orbits, and away from this zero set $|\alpha|$ is
close to $1$ and $\beta$ is close to $0$.

\subsection{The perturbed Seiberg-Witten equations, more explicitly}

If $(A,\psi)$ is a solution to the perturbed Seiberg-Witten equations
\eqref{eqn:PSW}, write $\psi=(\alpha,\beta)$ where $\alpha$ is a
section of $E$ and $\beta$ is a section of $K^{-1}E$.  We now rewrite
the perturbed Seiberg-Witten equations in terms of $A$, $\alpha$ and
$\beta$ and establish some notation which will be used below.

The curvature equation in \eqref{eqn:PSW} can be written as
\begin{equation}
\label{eqn:curvature}
*F_A = ir\left((|\alpha|^2-|\beta|^2-1)\lambda +
2\op{Im}(\alpha\beta^*)\right) + i\overline{\omega}.
\end{equation}
Here $\alpha\beta^*$ denotes the $\C$-valued $1$-form that projects a
vector onto $K^{-1}$, pairs it with dual of $\beta$ in $(K^{-1}E)^*$,
and then pairs the result with $\alpha$.

To rewrite the Dirac equation, let $\nabla_A$ denote the covariant
derivative on $E$ corresponding to $A$, and also the covariant
derivative on $K^{-1}E$ induced by $A$ together with the canonical
connection $A_0$ on $K^{-1}$ from Lemma~\ref{lem:A0}.  Let
$\dbar_A\alpha$ denote the complex antilinear part of the restriction
of $\nabla_A\alpha:TY\to E$ to $\xi=K^{-1}$.  Likewise let $\partial_A\beta$
denote the complex linear part of the restriction of
$\nabla_A\beta:TY\to K^{-1}E$ to $K^{-1}$.  Let
$\nabla_{A,R}$ denote the covariant derivative along the Reeb vector field.
The Dirac equation in \eqref{eqn:PSW} can now be written
as\footnote{Taubes obtains a slightly different equation here.  See
  Remark~\ref{rem:conventions}.}
\begin{equation}
\label{eqn:Dirac}
\begin{pmatrix} i\nabla_{A,R}\alpha - 2\partial_A\beta \\ -i\nabla_{A,R}\beta
  + 2\dbar_A\alpha \end{pmatrix} + \begin{pmatrix} f_0\alpha +
  f_1\beta \\ f_0'\alpha + f_1'\beta \end{pmatrix} = 0.
\end{equation}
Here $f_0,f_1,f_0',f_1'$ are bundle maps between $E$ and $K^{-1}E$
which do not depend on $r$;  they arise from the failure of the spin
covariant derivative on $E\oplus K^{-1}E$ to agree with the covariant
derivatives $\nabla_A$ on $E$ and $K^{-1}E$ (which in turn is related
to the failure of the contact geometry to be parallel with respect to
the Levi-Civita connection).

\subsection{A priori estimates}
\label{sec:APE}

The starting point for the analysis of the perturbed Seiberg-Witten
equations \eqref{eqn:PSW} is the following lemma giving a priori
estimates on their solutions.

\begin{lemma}
\label{lem:APE}
\cite[Lemmas 2.2 and 2.3]{tw} There exists a constant $c_0$ such that
if $r$ is sufficiently large and if $(A,\psi=(\alpha,\beta))$ is a
solution to the equations \eqref{eqn:PSW}, then:
\begin{align}
\label{eqn:APE1}
|\alpha| & \le 1 + \frac{c_0}{r},\\
\label{eqn:APE2}
|\beta|^2 &\le \frac{c_0}{r}\left|1-|\alpha|^2\right| + \frac{c_0}{r^2},\\
\label{eqn:APE3}
|\nabla_A\alpha| & \le c_0r^{1/2},\\
\label{eqn:APE4}
|\nabla_A\beta| & \le c_0.
\end{align}
\end{lemma}
\noindent

\begin{proof}
  Let $\widetilde{\nabla}_A$ denote the spin covariant derivative on
  ${\mathbb S}$, in order to avoid confusion with the covariant
  derivatives $\nabla_A$ on $E$ and $K^{-1}E$.  Putting the perturbed
  Seiberg-Witten equations \eqref{eqn:PSW} into the
  Bochner-Lichnerowitz-Weitzenbock formula \eqref{eqn:BLW}, we obtain
\begin{equation}
\label{eqn:APEBLW}
0 = \widetilde{\nabla}_A^*\widetilde{\nabla}_A\psi + \frac{s}{4}\psi +
\frac{r}{2}\left(|\psi|^2\psi + i\rho(\lambda)\psi\right) -
\frac{i}{2}\rho(\overline{\omega})\psi.
\end{equation}
Since the spin connection is compatible with the Hermitian metric we
have $\frac{1}{2}d^*d|\psi|^2 =
\op{Re}\langle\psi,\widetilde{\nabla}_A^*\widetilde{\nabla}_A\psi\rangle
- |\widetilde{\nabla}_A\psi|^2$.  Also
$\rho(\lambda)(\alpha,\beta)=(i\alpha,-i\beta)$.  So taking the real
inner product of \eqref{eqn:APEBLW} with $\psi$, we obtain
\[
0 \ge \frac{1}{2}d^*d|\psi|^2 + |\widetilde{\nabla}_A\psi|^2
+\frac{r}{2}\left(|\psi|^4 - (1+c_1r^{-1})|\psi|^2
\right),
\]
where the constant $c_1$ depends only on the minimum of the scalar
curvature $s$ and the maximum of $|\overline{\omega}|$.  At a point
where $|\psi|$ is maximized, the maximum principle tells us that
$d^*d|\psi|^2\ge 0$.  Therefore
\[
|\psi|^2\le 1+c_1r^{-1}
\]
where $|\psi|$ is maximized, and consequently everywhere.  This implies
\eqref{eqn:APE1}.  

To prove \eqref{eqn:APE2}, we consider separately the $E$ and
$K^{-1}E$ components of equation \eqref{eqn:APEBLW}.  These are
\[
\begin{split}
0 &= \nabla_A^*\nabla_A\alpha +
\frac{r}{2}(|\alpha|^2+|\beta|^2-1)\alpha +
f_2\alpha+f_3\beta+f_4\nabla_A\alpha + f_5\nabla_A\beta,\\
0 &= \nabla_A^*\nabla_A\beta +
\frac{r}{2}(|\alpha|^2+|\beta|^2+1)\beta +
f_2'\alpha+f_3'\beta+f_4'\nabla_A\alpha+f_5'\nabla_A\beta
\end{split}
\]
where $f_2,\ldots,f_5'$ are bundle maps between $E$ and $K^{-1}E$
depending only on the Riemannian metric on $Y$.  Taking the inner
products of these equations with $\alpha$ and $\beta$ respectively, we
obtain
\[
\begin{split}
0 &\ge \frac{1}{2}d^*d|\alpha|^2 + |\nabla_A\alpha|^2 +
\frac{r}{2}(|\alpha|^2+|\beta|^2-1)|\alpha|^2 -
c_2(|\alpha|^2+|\alpha||\beta|+|\alpha||\nabla_A\beta|),\\
0 &\ge \frac{1}{2}d^*d|\beta|^2 + |\nabla_A\beta|^2 +
\frac{r}{2}(|\alpha|^2+|\beta|^2+1)|\beta|^2 - c_2(|\beta|^2 +
|\alpha||\beta| + |\nabla_A\alpha||\beta|)
\end{split}
\]
where $c_2$ is a constant depending only on the
Riemannian metric on $Y$.  Now let $c_3$ and $c_4$ be large
constants.  Adding $c_3r^{-1}$ times the first inequality to the second
gives
\begin{gather*}
  0 \ge \frac{1}{2}d^*d
  \left(|\beta|^2-\frac{c_3}{r}(1-|\alpha|^2)-\frac{c_4}{r^{2}}\right)
  +
  \frac{r}{2}\left(|\beta|^2-\frac{c_3}{r}(1-|\alpha|^2)-\frac{c_4}{r^{2}}\right)
  \\
  + \frac{c_4}{2r} + \frac{c_3}{r}|\nabla_A\alpha|^2 +
  \frac{c_3}{2}((1-|\alpha|^2)^2+|\alpha|^2|\beta|^2) +
  |\nabla_A\beta|^2 + \frac{r}{2}(|\alpha|^2+|\beta|^2)|\beta|^2  \\
  - c_2(|\beta|^2 + |\alpha||\beta| + |\nabla_A\alpha||\beta|) -
  \frac{c_2c_3}{r}(|\alpha|^2 + |\alpha||\beta| +
  |\alpha||\nabla_A\beta|).
\end{gather*}
Inspection of the above inequality using \eqref{eqn:APE1} shows that
the sum of the terms on the second and third line is nonnegative
provided that $r>>c_4>>c_3>>0$.  (The $|\nabla_A\alpha||\beta|$ term
is a bit tricky and requires consideration of several cases.)  The
maximum principle now implies \eqref{eqn:APE2}.

The estimates \eqref{eqn:APE3} and \eqref{eqn:APE4} are then obtained
using elliptic regularity, as explained in \cite[\S6.2]{tw}.
\end{proof}

As a quick corollary, we have:

\begin{proof}[Proof of Lemma~\ref{lem:aes}.]
  Suppose $c_1(E)\neq 0$ and $(A,\psi)$ is a solution to the equations
  \eqref{eqn:PSW}.  Since $c_1(E)\neq 0$ there exists a point where
  $\alpha=0$.  At this point $1-|\psi|=1-|\beta|$, and it follows from
  \eqref{eqn:APE2} that this is at least $1-cr^{-1/2}$ for some
  constant $c$ if $r$ is sufficiently large.
\end{proof}

We also need an a priori estimate on the connection.  Recall that
$A_1$ denotes a reference connection on $E$.  Write a general
connection on $E$ as $A=A_1+\hat{a}$ where $\hat{a}$ is an
imaginary-valued $1$-form on $Y$.  Also recall that $\mc{E}=\mc{E}(A)$
denotes the energy functional \eqref{eqn:energy}.

\begin{lemma}
\label{lem:gauge}
\cite[Lemma 2.4]{tw}
Let $(A,\psi)$ be a solution to the equations \eqref{eqn:PSW}.  Then
by a gauge transformation one can arrange that
\begin{equation}
\label{eqn:hata}
|\hat{a}| \le c(r^{2/3}|\mc{E}|^{1/3} + 1)
\end{equation}
where $c$ is an $r$-independent constant.
\end{lemma}

\begin{proof}
  Recall the Hodge decomposition
\[
\Omega^1 = \mc{H}^1 \oplus d\Omega^0 \oplus d^*\Omega^2
\]
where $\mc{H}^1$ denotes the space of harmonic 1-forms.  Let
$\frak{h}\in\mc{H}^1$ denote the harmonic component of $-i\hat{a}$.  A
gauge transformation $g:Y\to S^1$ subtracts $2g^{-1}dg$ from
$\hat{a}$.  Thus by a gauge transformation one can eliminate the
$d\Omega^0$ component of $-i\hat{a}$, and one can shift the harmonic
component $\frak{h}$ by any element of $4\pi H^1(Y;\Z)$.  So we can
apply a (unique) gauge transformation to arrange that $\hat{a}$ is
co-closed, and $\frak{h}$ sends every element of some chosen basis of
$H_1(Y;\Z)$ mod torsion to a number in the interval $[0,4\pi)$.  In
particular $|\frak{h}|$ is bounded from above by a constant depending
only on the Riemannian metric.

We now have
\[
\hat{a} = \frak{h} + (d+d^*)^{-1}(F_A-F_{A_1})
\]
where $(d+d^*)^{-1}$ denotes the Green's function for $d+d^*$, and
$|\frak{h}-(d+d^*)^{-1}F_{A_1}|$ is bounded by an $r$-independent
constant.  Standard estimates for the Green's function then give
\begin{equation}
\label{eqn:GFE}
|\hat{a}(x)| \le c\left(1 + \int_Y\frac{|F_A|}{\op{dist}(x,\cdot)^2}\right).
\end{equation}

It follows from equation \eqref{eqn:curvature} and the a priori
estimates \eqref{eqn:APE1} and \eqref{eqn:APE2} that
\begin{equation}
\label{eqn:cest}
*F_A=-ir|1-|\alpha|^2|\lambda + O(1).
\end{equation}
In particular, it follows from
\eqref{eqn:cest} that the integrand in the energy functional
\eqref{eqn:energy} satisfies
\begin{equation}
\label{eqn:ei}
i\lambda\wedge F_A = |F_A| + O(1).
\end{equation}
(We are omitting the volume form from the notation here.)

To analyze the integral in \eqref{eqn:GFE}, divide the integration
domain into a ball of radius $\rho$ centered at $x$, and its
complement.  Since $|F_A|$ is bounded by a constant multiple of $r$,
the integral over the ball is bounded by a constant times $r\rho$.  On
the other hand, by \eqref{eqn:ei} the integral over the complement of
the ball is bounded by a constant times $|\mc{E}|\rho^{-2}$, plus some
other constant times $\rho^{-2}$.  Thus
\begin{equation}
\label{eqn:ahb}
|\hat{a}(x)| \le c\left(1 + r\rho + \frac{|\mc{E}|+c}{\rho^{2}}\right)
\end{equation}
for some new constant $c>0$.  Now take $\rho=r^{-1/3}|\mc{E}|^{1/3}$.
Note that $\rho=O(1)$.  The estimate \eqref{eqn:ahb} then implies
\eqref{eqn:hata}.
\end{proof}

We can now give:

\begin{proof}[Proof of Lemma~\ref{lem:cse}.]
Recall that we are assuming that $2c_1(E)+c_1(K^{-1})$ is torsion, so
that the Chern-Simons functional is gauge invariant.  Thus we can
assume that the connection is the one provided by
Lemma~\ref{lem:gauge}.  It then follows from the definition of the
Chern-Simons functional \eqref{eqn:cs} that
\[
|cs(A)| \le
\int_Yc(r^{2/3}\mc{E}^{1/3}+1)(|F_A|+c).
\]
On the other hand, by equation \eqref{eqn:ei} we have
$\int_Y|F_A|=\mc{E}+O(1)$.  The lemma follows.
\end{proof}

\subsection{Existence of a Reeb orbit}
\label{sec:ERO}

We can now sketch the proof of Taubes's Theorem~\ref{thm:convergence}.
To start, we will outline the proof of the following:

\begin{claim}
  Fix $\delta,C>0$.  Then there exists $c>0$ such that the following
  holds.  Suppose $r>c$. Let $(A,\psi)$ be a solution to the perturbed
  Seiberg-Witten equations \eqref{eqn:PSW} with
  $\sup_Y(1-|\psi|)>\delta$ and $\mc{E}(A)<C$.  Then there exists
  $n<c$ and a finite collection of arcs $\{\gamma_i\}$ in $Y$ indexed
  by $i\in\Z/n$ such that:
\begin{itemize}
\item
Each $\gamma_i$ is tangent to the Reeb vector field.
\item The length of each arc $\gamma_i$ is at most $c$.  Also the sum
  of the lengths of the arcs is at least $c^{-1}$.
\item The distance between the endpoint of $\gamma_i$ and the starting
  point of $\gamma_{i+1}$ is at most $cr^{-1/2}$.
\end{itemize}
\end{claim}

It follows from the claim that given a sequence $(r_n,A_n,\psi_n)$ of
solutions to the perturbed Seiberg-Witten equations satisfying
conditions (1) and (2) in Theorem~\ref{thm:convergence}, one can pass
to a subsequence such that the corresponding chains of arcs $\gamma_i$
converge to a Reeb orbit.  

To prove the claim, let $(A,\psi)$ be a solution to the perturbed
Seiberg-Witten equations with $\sup_Y(1-|\psi|)>\delta$ and
$\mc{E}(A)<C$.  Write $\psi=(\alpha,\beta)$ as usual.

It follows from the Dirac equation \eqref{eqn:Dirac} and the a priori
estimates \eqref{eqn:APE1}, \eqref{eqn:APE2}, and \eqref{eqn:APE4}
that
\begin{align}
\label{eqn:MAPE1}
|\nabla_{A,R}\alpha| &\le c_0,\\
\label{eqn:MAPE2}|\dbar_A\alpha| &\le c_0
\end{align}
where $c_0$ is an $r$-independent constant.

It also follows from \eqref{eqn:APE1} and \eqref{eqn:APE2} that if $r$
is sufficiently large, then there exist points where
$1-|\alpha|^2>\delta$.  Assume that it is and choose such a point $p$.
Also fix $r$-independent constants $\rho_1,\rho_2>0$ (to be specified
more later).  We can now choose an embedding of the disc $D$ of radius
$\rho_1 r^{-1/2}$ into $Y$ mapping the center of the disc to $p$, such
that the disc is orthogonal to the Reeb vector field at $p$, and the
induced metric on the disc is Euclidean to first order at the origin.
Denote the coordinates on the disc by $z=x+iy$.  We can uniquely
extend the embedding of the disc $D$ to a map from the cylinder
$D\times[0,\rho_2]$ to $Y$ such that the derivative of the interval
coordinate $t$ maps to the Reeb vector field.  We can assume that the
map of the cylinder is an embedding (for $r>>\rho_1$), otherwise we
already know that the claim is true.  We then take the arc $\gamma_1$
to be the image of $\{0\}\times[0,\rho_2]$ under this embedding.

Now let $\chi:[0,1]\to\R$ be a cutoff function which is $1$ on
$[0,1/3]$, monotone decreasing, and $0$ on $[2/3,1]$.  For
$t\in[0,\rho_2]$ define
\begin{equation}
\label{eqn:Lt}
L(t) \eqdef r\int_{D\times\{t\}}(1-|\alpha|^2)\chi(|z|/\rho_1).
\end{equation}
In view of \eqref{eqn:cest} and \eqref{eqn:ei}, this is the
contribution to the energy $\mc{E}(A)$ from the disc $D\times\{t\}$
(weighted by the cutoff function), up to an error of order
$r^{-1/2}$.  Now it follows from \eqref{eqn:APE3} that
$1-|\alpha|^2>\delta/2$ on a disc of radius order $r^{-1/2}$ in
$D\times\{0\}$, so
\begin{equation}
\label{eqn:lbe1}
L(0)\ge c_1
\end{equation}
for some $r$-independent constant $c_1>0$.  Also, differentiating
\eqref{eqn:Lt} with respect to $t$ and using the a priori estimates
\eqref{eqn:APE1} and \eqref{eqn:MAPE1} shows that
\begin{equation}
\label{eqn:lbe2}
\left|\frac{dL(t)}{dt}\right| < c_2
\end{equation}
for some $r$-independent constant $c_2$.  In particular, if $\rho_2>0$
is chosen sufficiently small, then $L(\rho_2)\ge c_1/2$, and the
contribution to the energy $\mc{E}(A)$ from the cylinder is bounded
from below by an $r$-independent constant.

To complete the proof of the claim, it is enough to show that the disk
$D\times\{\rho_2\}$ contains a point where $1-|\alpha|^2>\delta$.  We
can then repeat the above argument to construct a sequence of
cylinders in $Y$, and define the arcs $\gamma_i$ to be the core
intervals of these cylinders.  There is an upper bound to the number
of cylinders one can construct before a new cylinder overlaps an old
cylinder, because each cylinder contributes at least $c_2$ to the
energy, while the total energy is bounded from above by $C$.  When a
new cylinder overlaps an old one, we then obtain the desired chain of
arcs.

In fact, if the constant $\rho_1$ is sufficiently large, then $\alpha$
must have a zero on $D\times\{\rho_2\}$.  To see why, rescale
$D\times\{\rho_2\}$ by the map $z\mapsto r^{1/2}z$, so as to identify
$D\times\{\rho_2\}$ with the disc of radius $\rho_1$.  The restriction
of $(A,\alpha)$ to the latter now satisfies
\[
\begin{split}
\dbar_A\alpha &= O(r^{-1/2}),\\
*F_A &= -i(1-|\alpha|^2) + O(r^{-1/2}).
\end{split}
\]
as a result of \eqref{eqn:MAPE2}, \eqref{eqn:curvature},
\eqref{eqn:APE1} and \eqref{eqn:APE2}.  That is, $(A,\alpha)$
satisfies the vortex equations \eqref{eqn:vortex} on the disk, up to
an error of order $r^{-1/2}$.  Moreover, if $\rho_1$ is sufficiently
large, then there must exist a large $\rho'<\rho_1$ such that
$1-|\alpha|^2<\delta$ on the circle of radius $\rho'$; otherwise the
arguments giving \eqref{eqn:lbe1} and \eqref{eqn:lbe2} would imply
that the energy $\mc{E}(A)>C$.  A compactness argument then shows that
if $r$ is sufficiently large, then $(A,\alpha)$ not only approximately
solves the vortex equations, but is $C^0$-approximated by an actual
solution $(A',\alpha')$ to the vortex equations on $\C$ with finite
energy, where $|\alpha'|$ is close to $1$ outside of the disc of
radius $r'$.  It now follows from Lemma~\ref{lem:vortex} that $\alpha$
must have a zero in the disc of radius $\rho'$.

This completes the sketch of the proof of the claim, and hence the
existence of a Reeb orbit.  A more careful version of this argument
keeping track of all of the zero set of $\alpha$, see
\cite[\S6.4]{tw}, shows that in fact there exists a nonempty orbit set
Poincar\'{e} dual to $c_1(E)$.

\section{Beyond the Weinstein conjecture}
\label{sec:BWC}

With the three-dimensional Weinstein conjecture now proved, there are
various directions in which one might try to generalize.

\subsection{Improved lower bounds}
\label{sec:ILB}

The Weinstein conjecture gives a lower bound of one on the number of
embedded Reeb orbits for a contact form $\lambda$ on a closed oriented
3-manifold $Y$.  Can one improve this lower bound?

As we saw in \S\ref{sec:SE}, the standard contact form on an
irrational ellipsoid has exactly two embedded Reeb orbits.  Also one
can take the quotient of the ellipsoid by a $\Z/p$ action on $\C^2$
that rotates each $\C$ factor, to obtain a contact form on a lens
space with exactly two embedded Reeb orbits.  Perhaps surprisingly,
these are the only examples known to us of contact forms on closed
(connected) 3-manifolds with only finitely many embedded Reeb orbits.
It is shown in \cite{hwz} that for a large class of contact forms on
$S^3$ there are either two or infinitely many embedded Reeb orbits.
For other 3-manifolds one can ask:

\begin{question}
  If $Y$ is a closed oriented connected 3-manifold other than a sphere
  or a lens space, then does every contact form on $Y$ have infinitely many
  embedded Reeb orbits?
\end{question}

As mentioned in \S\ref{sec:previous}, Colin-Honda \cite{colin-honda}
used linearized contact homology to prove the existence of infinitely
many embedded Reeb orbits for many cases of contact structures
supported by open books with pseudo-Anosov monodromy.  Also it is
proved in \cite{wh}, using the isomorphism between Seiberg-Witten
Floer homology and embedded contact homology (see \S\ref{sec:ech}),
that if $Y$ is a closed oriented 3-manifold other than a lens space,
then any contact form on $Y$ with all Reeb orbits nondegenerate (see
\S\ref{sec:NRO}) has at least three embedded Reeb orbits.  Nonetheless
there remains a substantial gap between what we can prove and what
seems to be true.

\subsection{More general vector fields}
\label{sec:MGVF}

Next one might try to prove the existence of closed orbits for
somewhat more general vector fields than Reeb vector fields.  For
example, inspection of the proof of the Weinstein conjecture for a
compact hypersurface $Y$ in $\R^{2n}$ of contact type, see
\S\ref{sec:hct}, shows that the contact type hypothesis can be
replaced by the weaker assumption that the hypersurface is ``stable''.
A hypersurface $Y$ in a symplectic manifold $(M,\omega)$ is called
{\em stable\/} if it has a neighborhood $N$ with an identification
$N\simeq (-\delta,\delta)\times Y$ sending $Y$ to $\{0\}\times Y$,
such that the characteristic foliations on $\{\epsilon\}\times Y$ are
conjugate for all $\epsilon\in(-\delta,\delta)$.  It turns out (see
\cite[Lem.\ 2.3]{cm}) that a compact hypersurface $Y$ is stable if and
only if $\omega|_Y$ is part of a ``stable Hamiltonian structure'' on
$Y$.  A {\em stable Hamiltonian structure\/} on a $2n-1$ dimensional
oriented manifold $Y$ is a pair $(\lambda,\omega)$ where $\lambda$ is
a $1$-form on $Y$ and $\omega$ is a closed $2$-form on $Y$, such that
$\lambda\wedge \omega^{n-1}>0$ and
$\Ker(\omega)\subset\Ker(d\lambda)$.  A stable Hamiltonian structure
determines a ``Reeb vector field'' $R$ characterized by
$\omega(R,\cdot)=0$ and $\lambda(R)=1$.  A contact form is a special
case of this in which $\omega=d\lambda$.

It is shown in \cite{wh}, again using the isomorphism between
Seiberg-Witten Floer homology and embedded contact homology, that for
any closed oriented connected 3-manifold $Y$ that is not a
$T^2$-bundle over $S^1$, for any stable Hamiltonian structure on $Y$,
the associated Reeb vector field has a closed orbit.  The same
conclusion is proved without using Seiberg-Witten theory, but under
some additional hypotheses, by Rechtman \cite{rechtman}.  In general,
however, it remains unclear what exactly one needs to assume about a
vector field in order to guarantee the existence of a closed orbit.

\subsection{Non-unique ergodicity}
\label{sec:ue}

For even more general vector fields, one can try to prove weaker
statements than the existence of a closed orbit.  For example let $Y$
be a closed oriented 3-manifold with a volume form $\Omega$, and let
$V$ be a smooth vector field on $Y$.  The vector field $V$ generates a
1-parameter family of diffeomorphisms $\phi_t:Y\to Y$.  Assume that
$V$ is divergence free, meaning that each $\phi_t$ preserves the
volume form, or equivalently $d(\imath_V\Omega)=0$.  A measure
$\sigma$ on $Y$ is said to be ``$V$-invariant'' if
$(\phi_t)_*\sigma=\sigma$ for each $t\in\R$.  The vector field $V$ is
called {\em uniquely ergodic\/} if the only $V$-invariant measures on
$Y$ are real multiples of the volume form.  Note that if $V$ has a
closed orbit then it is not uniquely ergodic.  More generally, one
could look for conditions on $V$ that guarantee that it is not
uniquely ergodic.

Taubes \cite{tue} establishes such a condition as follows.
Call $V$ ``exact'' if $\imath_V\Omega=d\lambda$ for some
$1$-form $\lambda$.  In this case define the ``self-linking'' of $V$,
cf \cite{ak},
to be
\begin{equation}
\label{eqn:sv}
s_V\eqdef \int_Y\lambda\wedge \imath_V\Omega.
\end{equation}
By Stokes theorem, this does not depend on the choice of $\lambda$.
For example, if $\lambda(V)>0$ everywhere, then the integrand in
\eqref{eqn:sv} is positive, so $s_V>0$.  In this case $\lambda$ is a
contact form and $V$ is a positive multiple of its Reeb vector field.
More generally one has:

\begin{theorem}[Taubes \cite{tue}]
Let $Y$ be a closed oriented 3-manifold with a volume form $\Omega$.
Let $V$ be an exact vector field on $Y$.  Suppose that $s_V\neq 0$.
Then $V$ is not uniquely ergodic.
\end{theorem}

Taubes proves this similarly to the Weinstein conjecture.  One
considers the perturbed Seiberg-Witten equations \eqref{eqn:PSW} with
$\lambda$ replaced by $*d\lambda$, and proves a modification of
Theorem~\ref{thm:convergence} in which hypothesis (2) in that theorem
is dropped.  A nontrivial $V$-invariant measure is obtained as a limit
of a subsequence of the sequence of measures
\[
\sigma_n \eqdef \mc{E}(A_n)^{-1}r_n(1-|\alpha_n|^2)\Omega
\]
for an appropriately chosen sequence $(r_n,A_n,\psi_n)$.

\subsection{The Arnold chord conjecture}

There is also the following ``relative'' version of the Weinstein
conjecture.  Let $(Y,\xi)$ be a contact 3-manifold.  A {\em Legendrian
  knot\/} is a knot $L\subset Y$ such that $T_pL\subset \xi_p$ for
every $p\in L$.  Now choose a contact form $\lambda$ with
$\Ker(\lambda)=\xi$.  A {\em Reeb chord\/} of $L$ is a path
$\gamma:[0,T]\to Y$ for some $T>0$ such that $\gamma(0),\gamma(T)\in
L$ and $\gamma$ is a flow line of the Reeb vector field, ie
$\gamma'(t)=R(\gamma(t))$.  For example, any Legendrian knot $L$ in
$\R^3$ with the standard contact form \eqref{eqn:SCS} must have a Reeb
chord, because the projection of $L$ to the $x,y$ plane must have area
zero, so it must have a crossing.  A version of the Arnold chord
conjecture asserts that for any Legendrian knot in a closed 3-manifold
with a contact form, there exists a Reeb chord.  This has been proved
for the standard contact structure on $S^3$ \cite{mohnke}, and for
Legendrian unknots in tight contact 3-manifolds satisfying certain
assumptions \cite{abbas}, but it seems that not too much is known
about the general case.

\subsection{Higher dimensions}

In higher dimensions, although the Weinstein conjecture is known for
compact hypersurfaces of contact type in $\R^{2n}$ and for some other
cases, in general it is wide open.  The techniques used in Taubes's
proof are special to three dimensions.  In particular no good analogue
of the Seiberg-Witten invariants is currently known in dimensions
greater than four.  In addition there is no obvious higher-dimensional
analogue of embedded contact homology.  (In higher dimensions one
expects that generically all non-multiply-covered holomorphic curves
of the relevant index are embedded, compare \cite{oz}.)  One can still
use holomorphic curves in higher dimensions to define linearized
contact homology and related invariants from symplectic field theory
\cite{egh}.  It is unclear if these invariants are sufficient to prove
the Weinstein conjecture in all cases.  It is shown in \cite{ah} that
the Weinstein conjecture holds for ``PS-overtwisted'' contact
structures, which are a certain higher dimensional analogue of
overtwisted contact structures, cf \S\ref{sec:overtwisted}.  It is
also shown in \cite{bn} that symplectic field theory invariants
recover this fact.  In any case we will close with:

\begin{question}
Is there a proof of the Weinstein conjecture in three dimensions using
only holomorphic curves (and no Seiberg-Witten theory)?  What about in
higher dimensions?
\end{question}

\bibliographystyle{amsplain}

\begin{thebibliography}{99}

\bibitem{abbas} C. Abbas, {\em The chord problem and a new method of
    filling by pseudoholomorphic curves\/}, Int. Math. Res. Not. {\bf
    2004}, 913--927.

\bibitem{ach} C. Abbas, K. Cieliebak, and H. Hofer, {\em The Weinstein
    conjecture for planar contact structures in dimension three\/},
  Comment. Math. Helv. {\bf 80\/} (2005), 771-793.

\bibitem{ah} P. Albers and H. Hofer, {\em On the Weinstein conjecture
    in higher dimensions\/}, Comment. Math. Helv. {\bf 84\/} (2009), 429--436.

\bibitem{ak} V. Arnold and B. Khesin, {\em Topological methods in
    hydrodynamics\/}, Applied Mathematical Sciences 125, Springer
  Verlag, 1998.

\bibitem{auroux} D. Auroux, {\em La conjecture de Weinstein en
    dimension 3 [d'apr\`{e}s C.H. Taubes]}, S\'{e}minaire
  Bourbaki, 2008-2009, no. 1002.

\bibitem{bgv} N. Berline, E. Getzler, and M. Vergne, {\em Heat kernels
    and Dirac operators\/}, Springer Verlag, 1990.

\bibitem{bott} R. Bott, {\em Morse theory indomitable\/}, IHES
  Publ. Math. (1988) 99--114.

\bibitem{bourgeois02} F. Bourgeois, {\em Odd dimensional tori are
    contact manifolds\/}, Int. Math. Res. Not. {\bf 2002\/}, no. 30,
  1571--1574.

\bibitem{bmb} F. Bourgeois, {\em A Morse-Bott approach to contact
    homology\/}, Symplectic and contact topology: interactions and
  perspectives, 55--77, Fields Inst. Commum. 35, AMS, 2003.

\bibitem{bce} F. Bourgeois, K. Cieliebak, and T. Ekholm, {\em A note
    on Reeb dynamics on the tight 3-sphere\/}, J. Modern Dynamics {\bf
    1\/} (2007), 597--613.

\bibitem{bee} F. Bourgeois, T. Ekholm, and Y. Eliashberg, {\em A
    Legendrian surgery long exact sequence for linearized contact
    homology\/}, in preparation.

\bibitem{bn} F. Bourgeois and K. Niederkr\"{u}ger, {\em Towards a good
    definition of algebraically overtwisted\/}, arXiv:0709.3415

\bibitem{cm} K. Cieliebak and J. Mohnke, {\em Compactness of punctured
    holomorphic curves\/}, J. Symplectic Geom. {\bf 3\/} (2005),
  589--654.

\bibitem{colin-giroux-honda} V. Colin, E. Giroux, and K Honda, {\em
    Finitude homotopique et isotopique des structures de contact
    tendues\/}, Publ. Math. IHES {\bf 109\/} (2009), 245--293.

\bibitem{colin-honda} V. Colin and K. Honda, {\em Reeb vector fields
    and open book decompositions\/}, arXiv:0809.5088.

\bibitem{donaldson-survey} S.K. Donaldson, {\em The Seiberg-Witten
    equations and 4-manifold topology\/}, Bull. AMS {\bf 33\/} (1996), 45-70.

\bibitem{donaldson} S.K. Donaldson, {\em Topological field theories
    and formulae of Casson and Meng-Taubes\/}, Proceedings of the
  Kirbfest (Berkeley, CA, 1998), 87--102, Geom. Topol. Monogr. {\bf
    2\/}, Geom. Topol. Publ., 1999.

\bibitem{eliashberg} Y. Eliashberg, {\em Classification of overtwisted
    contact structures on 3-manifolds\/}, Invent. Math. {\bf 98}
  (1989), 623--637.

\bibitem{eft} Y. Eliashberg, {\em Filling by holomorphic discs and its
    applications\/}, Geometry of low-dimensional manifolds, 2 (Durham,
  1989), 45--67, London Math. Soc. Lecture Note Ser. {\bf 151\/},
  Cambridge Univ. Press, 1990.

\bibitem{egh} Y. Eliashberg, A. Givental, and H. Hofer, {\em
    Introduction to symplectic field theory\/}, GAFA 2000, Special
  Volume, Part II, 560--673.

\bibitem{etnyre} J. Etnyre, {\em Introductory lectures on contact geometry\/},
  Proc. Sympos. Pure Math. 71, pp. 81--107, Amer. Math. Soc. 2003.

\bibitem{etnyre06} J. Etnyre, {\em Lectures on open book
    decompositions and contact structures\/}, Floer homology, gauge
  theory, and low-dimensional topology, 103--141, Clay Math. Proc. 5,
  AMS, 2006.

\bibitem{gp} O. Garcia-Prada, {\em A direct existence proof for the
    vortex equations over a compact Riemann surface\/}, Bull. Londn
  Math. Soc. {\bf 26} (1994), 88--96.

\bibitem{geiges} H. Geiges, {\em An introduction to contact
    topology\/}, Cambridge University Press, 2008.

\bibitem{ginzburg} V. Ginzburg, {\em The Weinstein conjecture and the
    theorems of nearby and almost existence\/}, The Breadth of
  Symplectic and Poisson Geometry, Festschrift in Honor of Alan
  Wenstein, Birkh\"{a}user (2005), 139--172.

\bibitem{ginzburg-gurel} V. Ginzburg and B. G\"{u}rel, {\em A $C^2$-smooth
  counterexample to the Hamiltonian Seifert conjecture in $\R^4$\/}, 
    Ann. of Math. {\bf 158} (2003), 953--976.

  \bibitem{giroux} E. Giroux, {\em G\'{e}ometrie de contact: de la
      dimension trois vers les dimensions sup\'{e}rieures\/},
    Proceedings of the ICM, Beijing, 2002, Vol. II, 405--414.

\bibitem{harrison} J. Harrison, {\em $C^2$ counterexamples to the
    Seifert conjecture\/}, Topology {\bf 27\/} (1998), 249--278.

\bibitem{hofer93} H. Hofer, {\em Pseudoholomorphic curves in
    symplectizations with applications to the Weinstein conjecture in
    dimension three\/}, Invent. Math. {\bf 114\/} (1993), 515--563.

\bibitem{hwz} H. Hofer, K. Wysocki, and E. Zehnder, {\em Finite energy
    foliations of tight three-spheres and Hamiltonian dynamics\/},
  Ann. of Math. {\bf 157} (2003), 125-255.

\bibitem{hofer-zehnder} H. Hofer and E. Zehnder, {\em Symplectic
    invariants and Hamiltonian dynamics\/}, Birkh\"{a}user, 1994.

\bibitem{ir} M. Hutchings, {\em The embedded contact homology index
    revisited\/}, arXiv:0805.1240, to appear in the Yashafest proceedings.

\bibitem{t3} M. Hutchings and M. Sullivan, {\em Rounding corners of
    polygons and the embedded contact homology of $T^3$\/},
  Geom. Topol. {\bf 10} (2006), 269--266.

\bibitem{ht97} M. Hutchings and C.H. Taubes, {\em An introduction to
    the Seiberg-Witten equations on symplectic four-manifolds\/},
  Symplectic geometry and topology (Park City, UT, 1997), 103--142,
  AMS, 1999.

\bibitem{obg1} M. Hutchings and C.H. Taubes, {\em Gluing
    pseudoholomorphic curves along branched covered cylinders I\/},
  J. Symplectic Geom. {\bf 5\/} (2007), 43--137.

\bibitem{wh} M. Hutchings and C.H. Taubes, {\em The Weinstein
    conjecture for stable Hamiltonian structures\/}, Geom. Topol. {\bf
    13\/} (2009), 901--941.

\bibitem{jt} A. Jaffe and C.H. Taubes, {\em Vortices and
    monopoles. Structure of static gauge theories\/},
  Progress in Physics, 2. Birkh\"{a}user, Boston, 1980.

\bibitem{km} P.B. Kronheimer and T.S. Mrowka, {\em Monopoles and
    3-manifolds\/}, Cambridge University Press, 2008.

\bibitem{gkuperberg} G. Kuperberg, {\em A volume-preserving
    counterexample to the Seifert conjecture\/},
  Comment. Math. Helv. {\bf 71\/} (1996), 70--97.

\bibitem{kuperbergs} G. Kuperberg and K. Kuperberg, {\em Generalized
    counterexamples to the Seifert conjecture\/}, Ann. of Math. {\bf
    143} (1996), 547--576.

\bibitem{kkuperberg} K. Kuperberg, {\em A smooth counterexample to the
    Seifert conjecture\/}, Ann. of Math. {\bf 140} (1994), 723--732.

\bibitem{lm} H. Lawson and M. Michelson, {\em Spin geometry\/}, Princeton
  University Press, 1989.

\bibitem{lt} Y-J. Lee and C.H. Taubes, {\em Periodic Floer homology
    and Seiberg-Witten Floer cohomology\/}, arXiv:0906.0383.

\bibitem{msri} MSRI Hot Topics Workshop, {\em Contact structures,
    dynamics and the Seiberg-Witten equations in dimension 3\/}, June
  2008, videos at www.msri.org.

\bibitem{mcduff-salamon} D. McDuff and D. Salamon, {\em Introduction
    to symplectic topology\/}, 2nd edition, Oxford University Press, 1998.

\bibitem{meng-taubes} G. Meng and C.H. Taubes, {\em
    $\underline{SW}=$Milnor torsion\/}, Math. Res. Lett. {\bf 3\/}
  (1996), 661--674.

\bibitem{mohnke} K. Mohnke, {\em Holomorphic disks and the chord
    conjecture\/}, Ann. of Math. {\bf 154} (2001), 219--222.

\bibitem{morgan} J. Morgan, {\em The Seiberg-Witten equations and
    applications to the topology of smooth four-manifolds\/},
  Mathematical Notes 44, Princeton Univ. Press, 1996.

\bibitem{oz} Y-G. Oh and K. Zhu, {\em Embedding property of
    $J$-holomorphic curves in Calabi-Yau manifolds for generic $J$\/},
  arXiv:0805.3581.

\bibitem{os} P. Ozsv\'{a}th and Z. Szab\'{o}, {\em Holomorphic disks
    and topological invariants for closed three-manifolds\/}, Ann. of
  Math. {\bf 159\/} (2004), 1027--1158.

\bibitem{os4d} P. Ozsv\'{a}th and Z. Szab\'{o}, {\em Holomorphic
    triangles and invariants for smooth four-manifolds\/},
  Adv. Math. {\bf 202} (2006), 326--400.

\bibitem{perutz} T. Perutz, {\em Lagrangian matching invariants for
    fibred four-manifolds: I\/}, Geom. Topol. {\bf 11\/} (2007), 759--828.

\bibitem{rabinowitz} P. Rabinowitz, {\em Periodic solutions of
    Hamiltonian systems\/}, Comm. Pure Appl. Math {\bf 31} (1978),
  157--184.

\bibitem{rechtman} A. Rechtman, {\em Existence of periodic orbits for
    geodesible vector fields on closed 3-manifolds\/}, arXiv:0904.2719.

\bibitem{robbin-salamon} J. Robbin and D. Salamon, {\em The spectral
    flow and the Maslov index\/}, Bull. London Math. Soc. {\bf 27}
  (1995), 1--33.

\bibitem{schwarz} M. Schwarz, {\em Morse homology\/}, Progress in
  Mathematics, Birkh\"{a}user, 1993.

\bibitem{schwarz-action} M. Schwarz, {\em On the action spectrum for
    closed symplectially aspherical manifolds\/}, Pac. J. Math. {\bf
    193\/} (2000), 419--461.

\bibitem{schweizer} P. Schweizer, {\em Counterexamples to the Seifert
    conjecture and opening closed leaves of foliations\/}, Ann. of
  Math. {\bf 100} (1974), 386--400.

\bibitem{taubes:grtosw} C.H. Taubes, {\em $Gr\Rightarrow SW$: from
    pseudo-holomorphic curves to Seiberg-Witten solutions\/}, J,
  Diff. Geom. {\bf 51} (1999), 203-334.

\bibitem{taubes:swtogr} C.H. Taubes, {\em $SW\Rightarrow Gr$: from the
    Seiberg-Witten equations to pseudo-holomorphic curves\/}, in
  ``Seiberg-Witten and Gromov invariants for symplectic 4-manifolds'',
  Int. Press, 2000.

\bibitem{taubes:sw=gr} C.H. Taubes, {\em Seiberg-Witten and Gromov
    invariants for symplectic 4-manifolds\/}, First International
  Press Lecture Series 2, International Press, 2000.

\bibitem{tw} C.H. Taubes, {\em The Seiberg-Witten equations and the
    Weinstein conjecture\/}, Geom. Topol. {\bf 11\/} (2007), 2117--2202.

\bibitem{tw2} C.H. Taubes, {\em The Seiberg-Witten equations and the
    Weinstein conjecture II: More closed integral curves for the Reeb
    vector field\/}, Geom. Topol. {\bf 13\/} (2009), 1337-1417.

\bibitem{tsf} C.H. Taubes, {\em Asymptotic spectral flow for Dirac
    operators\/}, Comm. Anal. Geom. {\bf 15\/} (2007), 569--587.

\bibitem{tue} C.H. Taubes, {\em An observation concerning uniquely
    ergodic vector fields on 3-manifolds\/}, arXiv:0811.3983.

\bibitem{echswf1} C.H. Taubes, {\em Embedded contact homology and
    Seiberg-Witten Floer homology I\/}, arXiv:0811.3985.

\bibitem{echswf234} C.H. Taubes, {\em Embedded contact homology and
    Seiberg-Witten Floer homology II-IV\/}, preprints, 2008.

\bibitem{echswf5} C.H. Taubes, {\em Embedded contact homology and
    Seiberg-Witten Floer homology V\/}, preprint, 2008.

\bibitem{turaev} V. Turaev, {\em A combinatorial formulation for the
    Seiberg-Witten invariants of $3$-manifolds\/},
  Math. Res. Lett. {\bf 5\/} (1998), 583--598.

\bibitem{usher} M. Usher, {\em Vortices and a TQFT for Lefschetz
    fibrations on 4-manifolds\/}, Algebr. Geom. Topol. {\bf 6\/}
  (2006), 1677-1743.

\bibitem{viterbo} C. Viterbo, {\em A proof of Weinstein's conjecture
    in $\R^{2n}$\/}, Annales de l'Institut Henri Poincar\'{e}- Anal.
  Non Lin\'{e}aire, {\bf 4\/} (1987), 337--356.

\bibitem{weinstein78} A. Weinstein, {\em Periodic orbits for convex
    Hamiltonian systems\/}, Ann. of Math. {\bf 108\/} (1978),
  507--518.

\bibitem{weinstein79} A. Weinstein, {\em On the hypothesis of
    Rabinowitz' periodic orbit theorems\/}, J. Differential Equations
  {\bf 33} (1979), 353--358.

\bibitem{witten} E. Witten, {\em From superconductors and
    four-manifolds to weak interactions\/}, Bull. AMS {\bf 44\/}
  (2007), 361--391.

\bibitem{zehnder} E. Zehnder, {\em Remarks on periodic solutions on
    hypersurfaces\/}, Periodic solutions of hamiltonian systems and
  related topics, Reidel Publishing Co. (1987), 267--279.

\end{thebibliography}

\end{document}